\documentclass[11pt,a4paper,reqno]{amsart}
\usepackage[margin=2cm]{geometry}
\usepackage{hyperref,amssymb,amsfonts,amsthm,amsrefs,mathtools,verbatim,microtype,bm,nicefrac}
\usepackage{stmaryrd}
\usepackage[shortlabels]{enumitem}

\newtheorem{Theorem}{Theorem}[section]
\newtheorem{Proposition}[Theorem]{Proposition}
\newtheorem{Lemma}[Theorem]{Lemma}
\newtheorem{Corollary}[Theorem]{Corollary}
\newtheorem{ClaimC}{Claim}

\newtheorem{ClaimFinShuffle}{Claim}
\newtheorem{ClaimInfShuffle}{Claim}
\newtheorem*{Claim*}{Claim}

\theoremstyle{definition}
\newtheorem{Definition}[Theorem]{Definition}
\newtheorem{Example}[Theorem]{Example}

\DeclareFontFamily{U}{mathb}{\hyphenchar\font45}
\DeclareFontShape{U}{mathb}{m}{n}{
<-6> mathb5 <6-7> mathb6 <7-8> mathb7
<8-9> mathb8 <9-10> mathb9
<10-12> mathb10 <12-> mathb12
}{}
\DeclareSymbolFont{mathb}{U}{mathb}{m}{n}
\DeclareMathSymbol{\pprec}{\mathrel}{mathb}{"CE}
\DeclareMathSymbol{\ssucc}{\mathrel}{mathb}{"CF}

\DeclareFontFamily{U}{mathb}{\hyphenchar\font45}
\DeclareFontShape{U}{mathb}{m}{n}{
      <5> <6> <7> <8> <9> <10> gen * mathb
      <10.95> mathb10 <12> <14.4> <17.28> <20.74> <24.88> mathb12
      }{}
\DeclareSymbolFont{mathb}{U}{mathb}{m}{n}
\DeclareFontSubstitution{U}{mathb}{m}{n}
\DeclareMathSymbol{\monus}{2}{mathb}{"01}

\DeclareMathOperator{\dom}{\mathrm{dom}}

\DeclareMathOperator{\id}{\mathrm{id}}
\DeclareMathOperator{\condF}{\mathbf{c}_\textup{F}}

\DeclareMathOperator{\condSum}{\mathbf{c}_\Sigma}

\newcommand{\andd}{\wedge}
\newcommand{\orr}{\vee}
\newcommand{\la}{\langle}
\newcommand{\ra}{\rangle}
\newcommand{\da}{{\downarrow}}
\newcommand{\ua}{{\uparrow}}
\newcommand{\imp}{\rightarrow}
\newcommand{\Imp}{\Rightarrow}
\newcommand{\biimp}{\leftrightarrow}
\newcommand{\Biimp}{\Leftrightarrow}
\newcommand{\llb}{\llbracket}
\newcommand{\rrb}{\rrbracket}
\newcommand{\iso}{\cong}

\newcommand{\Nb}{\mathbb{N}}
\newcommand{\Zb}{\mathbb{Z}}
\newcommand{\Qb}{\mathbb{Q}}

\newcommand{\rst}{{\restriction}}
\newcommand{\keq}{\simeq}
\newcommand{\forae}{\forall^\infty}
\newcommand{\existsinf}{\exists^\infty}
\newcommand{\nmodels}{\not\models}

\newcommand{\mc}[1]{\mathcal{#1}}
\newcommand{\mf}[1]{\mathfrak{#1}}
\newcommand{\msf}[1]{\mathsf{#1}}
\newcommand{\ol}[1]{\overline{#1}}
\newcommand{\ora}[1]{\overrightarrow{#1}}
\newcommand{\wh}[1]{\widehat{#1}}

\newcommand{\leqT}{\leq_\mathrm{T}}

\newcommand{\equivT}{\equiv_\mathrm{T}}

\newcommand{\std}{\mathrm{std}}
\newcommand{\nonstd}{\mathrm{nonstd}}

\usepackage{xcolor}	
\usepackage{soul}

\AtBeginDocument{%
   \def\MR#1{}
}

\title{On cohesive powers of linear orders}

\author[Dimitrov]{Rumen Dimitrov}
\address{Department of Mathematics \& Philosophy\\
Western Illinois University\\
476 Morgan Hall\\
1 University Circle\\
Macomb, IL 61455\\
USA}
\email{rd-dimitrov@wiu.edu}
\urladdr{http://www.wiu.edu/users/rdd104/}

\author[Harizanov]{Valentina Harizanov}
\address{Department of Mathematics\\
The George Washington University\\
Phillips Hall\\
801 22\textsuperscript{nd} St.\ NW\\
Washington, DC 20052\\
USA
}
\email{harizanv@gwu.edu}
\urladdr{https://home.gwu.edu/~harizanv/}

\author[Morozov]{Andrey Morozov}
\address{Sobolev Institute of Mathematics\\
4 Acad.\ Koptyug avenue\\
630090 Novosibirsk\\
Russia}
\email{morozov@math.nsc.ru}
\urladdr{http://www.math.nsc.ru/~asm256/}

\author[Shafer]{Paul Shafer}
\address{School of Mathematics\\
University of Leeds\\
Leeds\\
LS2 9JT\\
United Kingdom}
\email{p.e.shafer@leeds.ac.uk}
\urladdr{http://www1.maths.leeds.ac.uk/~matpsh/}

\author[Soskova]{Alexandra A.\ Soskova}
\address{Department of Mathematical Logic and Applications\\
Faculty of Mathematics and Informatics\\
Sofia University\\
5 James Bourchier Blvd.\\
Sofia, 1164\\
Bulgaria}
\email{asoskova@fmi.uni-sofia.bg}
\urladdr{https://store.fmi.uni-sofia.bg/fmi/logic/asoskova/index.html}

\author[Vatev]{Stefan V.\ Vatev}
\address{Department of Mathematical Logic and Applications\\
Faculty of Mathematics and Informatics\\
Sofia University\\
5 James Bourchier Blvd.\\
Sofia, 1164\\
Bulgaria}
\email{stefanv@fmi.uni-sofia.bg}
\urladdr{https://store.fmi.uni-sofia.bg/fmi/logic/stefanv/}

\date{\today}

\begin{document}

\begin{abstract}
\emph{Cohesive powers} of computable structures are effective analogs of ultrapowers, where cohesive sets play the role of ultrafilters.  Let $\omega$, $\zeta$, and $\eta$ denote the respective order-types of the natural numbers, the integers, and the rationals when thought of as linear orders.  We investigate the cohesive powers of computable linear orders, with special emphasis on computable copies of $\omega$.  If $\mc{L}$ is a computable copy of $\omega$ that is computably isomorphic to the usual presentation of $\omega$, then every cohesive power of $\mc{L}$ has order-type $\omega + \zeta\eta$.  However, there are computable copies of $\omega$, necessarily not computably isomorphic to the usual presentation, having cohesive powers not elementarily equivalent to $\omega + \zeta\eta$.  For example, we show that there is a computable copy of $\omega$ with a cohesive power of order-type $\omega + \eta$.  Our most general result is that if $X \subseteq \Nb \setminus \{0\}$ is a Boolean combination of $\Sigma_2$ sets, thought of as a set of finite order-types, then there is a computable copy of $\omega$ with a cohesive power of order-type $\omega + \bm{\sigma}(X \cup \{\omega + \zeta\eta + \omega^*\})$, where $\bm{\sigma}(X \cup \{\omega + \zeta\eta + \omega^*\})$ denotes the shuffle of the order-types in $X$ and the order-type $\omega + \zeta\eta + \omega^*$.  Furthermore, if $X$ is finite and non-empty, then there is a computable copy of $\omega$ with a cohesive power of order-type $\omega + \bm{\sigma}(X)$.
\end{abstract}

\maketitle

\section{Introduction}

The ultimate inspiration for this work is Skolem's 1934 construction of a countable non-standard model of arithmetic~\cite{Skolem}.  Skolem's construction can be described roughly as follows.  For sets $X, Y \subseteq \Nb$, write $X \subseteq^* Y$ if $X \setminus Y$ is finite.  First, fix an infinite set $C \subseteq \Nb$ that is \emph{cohesive} for the collection of arithmetical sets:  for every arithmetical $A \subseteq \Nb$, either $C \subseteq^* A$ or $C \subseteq^* \ol{A}$.  Next, define an equivalence relation $=_C$ on the arithmetical functions $f \colon \Nb \imp \Nb$ by $f =_C g$ if and only if $C \subseteq^* \{n : f(n) = g(n)\}$.  Then define a structure on the $=_C$\nobreakdash-equivalence classes $[f]$ by $[f] + [g] = [f + g]$, $[f] \times [g] = [f \times g]$ (where $f + g$ and $f \times g$ are computed pointwise), and $[f] < [g] \;\Biimp\; C \subseteq^* \{n : f(n) < g(n)\}$.  Using the arithmetical cohesiveness of $C$, one then shows that this structure is elementarily equivalent to $(\Nb, +, \times, <)$.  The structure is countable because there are only countably many arithmetical functions, and it has non-standard elements, such as the element represented by the identity function.  See~\cite{DimitrovHarizanovHandbook} for a further discussion of Skolem's model.

Think of Skolem's construction as a more effective analog of an ultrapower construction.  Instead of building a structure from all functions $f \colon \Nb \imp \Nb$, Skolem builds a structure from only the arithmetical functions $f$.  The arithmetically cohesive set $C$ plays the role of the ultrafilter.  Feferman, Scott, and Tennenbaum~\cite{FefermanScottTennenbaum} investigate the question of whether Skolem's construction can be made more effective by assuming that $C$ is only \emph{r-cohesive} (i.e., cohesive for the collection of computable sets) and by restricting to computable functions $f$.  They answer the question negatively by showing that it is not even possible to obtain a model of Peano arithmetic in this way.  Lerman~\cite{LermanCo-r-Max} investigates the situation further and shows that if one restricts to \emph{cohesive} sets $C$ (i.e., cohesive for the collection of c.e.\ sets) that are co-c.e.\ and to computable functions $f$, then the first-order theory of the structure obtained is exactly determined by the many-one degree of $C$.  Additional results in this direction appear in~\cites{HirschfeldModels, HirschfeldWheelerBook}.

Dimitrov~\cite{DimitrovCohPow} generalizes the effective ultrapower construction to arbitrary computable structures.  These \emph{cohesive powers} of computable structures are studied in~\cites{DimitrovLattices, DimitrovHarizanov, DimitrovHarizanovMillerMourad} in relation to the lattice of c.e.\ subspaces, modulo finite dimension, of a fixed computable infinite dimensional vector space over the field $\Qb$.  In this work, we investigate a question dual to the question studied by Lerman.  Lerman fixes a computable presentation of a computable structure (indeed, all computable presentations of the standard model of arithmetic are computably isomorphic) and studies the effect that the choice of the cohesive set has on the resulting cohesive power.  Instead of fixing a computable presentation of a structure and varying the cohesive set, we fix a computably presentable structure and a cohesive set, and then we vary the structure's computable presentation.  We focus on linear orders, with special emphasis on computable presentations of $\omega$.  We choose to work with linear orders because they are a good source of non-computably categorical structures and because the setting is simple enough to be able to completely describe certain cohesive powers up to isomorphism.  This work is a greatly expanded version of the preliminary work of~\cite{CohPowCiE}.

Our main results are the following.  Below, $\omega$, $\zeta$, and $\eta$ denote the respective order-types of the natural numbers, the integers, and the rationals.  For each $k \geq 1$, $\bm{k}$ denotes the order-type of the $k$\nobreakdash-element linear order.
\begin{itemize}
\item If $C$ is cohesive and $\mc{L}$ is a computable copy of $\omega$ that is computably isomorphic to the usual presentation of $\omega$ (i.e., the immediate successor relation of $\mc{L}$ is computable), then the cohesive power $\prod_C \mc{L}$ has order-type $\omega + \zeta\eta$ (Theorem~\ref{thm-StdCohPow}).

\medskip

\item If $C$ is co-c.e.\ and cohesive and $\mc{L}$ is a computable copy of $\omega$, then the finite condensation of the cohesive power $\prod_C \mc{L}$ has order-type $\bm{1} + \eta$ (Theorem~\ref{thm-CoCeDenseCond}; see Definition~\ref{def-FinCond} for the definition of \emph{finite condensation}).

\medskip

\item If $C$ is co-c.e.\ and cohesive, then there is a computable copy $\mc{L}$ of $\omega$ where the cohesive power $\prod_C \mc{L}$ has order-type $\omega + \eta$ (Corollary~\ref{cor-DenseNonstd}).

\medskip

\item More generally, if $C$ is co-c.e.\ and cohesive and $X \subseteq \Nb \setminus \{0\}$ is a Boolean combination of $\Sigma_2$ sets, thought of as a set of finite order-types, then there is a computable copy $\mc{L}$ of $\omega$ where the cohesive power $\prod_C \mc{L}$ has order-type $\omega + \bm{\sigma}(X \cup \{\omega + \zeta\eta + \omega^*\})$.  Here $\omega^*$ denotes the reverse of $\omega$, and $\bm{\sigma}$ denotes the shuffle operation of Definition~\ref{def-Shuffle}.  Furthermore, if $X$ is finite and non-empty, then there is a computable copy $\mc{L}$ of $\omega$ where the cohesive power $\prod_C \mc{L}$ has order-type $\omega + \bm{\sigma}(X)$ (Theorem~\ref{thm-BCSigma2Shuffle}).
\end{itemize}

This work also serves to compare and contrast properties of cohesive powers with those of classical ultrapowers.  The key points are the following.  Recall that a computable structure is \emph{decidable} if its elementary diagram is computable and is \emph{$n$\nobreakdash-decidable} if its $\Sigma_n$\nobreakdash-elementary diagram is computable.  These definitions are discussed in more detail in Section~\ref{sec-ProdAndPow}.

\begin{itemize}
\item Classically, an ultrapower of a structure is elementarily equivalent to the base structure by {\L}o\'{s}'s theorem.  Effectively, {\L}o\'{s}'s theorem holds for cohesive powers of decidable structures (Corollary~\ref{cor-DecLos}).  For cohesive powers of $n$\nobreakdash-decidable structures, {\L}o\'{s}'s theorem need only hold up to $\Delta_{n+3}$\nobreakdash-expressible sentences.  In fact, every $\Sigma_{n+3}$ sentence true of an $n$\nobreakdash-decidable structure is also true of all of its cohesive powers (Theorem~\ref{thm-LosGeneral}), but this is optimal in general (Corollary~\ref{cor-LosSigma3Tight}).

\medskip

\item Classically, ultrapowers of isomorphic structures over a fixed ultrafilter are isomorphic.  Effectively, cohesive powers of \emph{computably isomorphic} computable structures over a fixed cohesive set are isomorphic (Theorem~\ref{thm-IsoCohPow}).  However, it is possible for isomorphic (but not computably isomorphic) computable structures to have non-elementarily equivalent (hence non-isomorphic) cohesive powers over a fixed cohesive set.  Together, Theorems~\ref{thm-StdCohPow} and~\ref{thm-NotNZQ} imply that for every cohesive set $C$, there are computable copies $\mc{L}_0$ and $\mc{L}_1$ of $\omega$ such that the cohesive powers $\prod_C \mc{L}_0$ and $\prod_C \mc{L}_1$ are not elementarily equivalent.  This sort of phenomenon can also be witnessed by computable structures whose cohesive powers are completely described.  Fix a co-c.e.\ cohesive set.  Example~\ref{ex-NonElemEquiv} shows that for every $k \geq 1$, there is a computable copy $\mc{L}$ of $\omega$ with $\prod_C \mc{L} \;\iso\; \omega + \bm{k}\eta$.  The order-types $\omega + \bm{k}\eta$ are pairwise non-elementarily equivalent for $k \geq 1$.  Theorem~\ref{thm-BCSigma2Shuffle} shows that many more order-types are achievable as cohesive powers of computable copies of $\omega$.

\medskip

\item Classically, the Keisler--Shelah theorem states that two structures are elementarily equivalent if and only if there is an ultrafilter (on a set of appropriate size) over which the corresponding ultrapowers are isomorphic.  Effectively, an analogous result holds for decidable structures:  decidable structures $\mc{A}$ and $\mc{B}$ are elementarily equivalent if and only if $\prod_C \mc{A} \;\iso\; \prod_C \mc{B}$ for every cohesive set $C$ (Theorem~\ref{thm-EffectiveKS}, which is essentially due to Nelson~\cite{Nelson} in a slightly different context).  If $\mc{A}$ and $\mc{B}$ are computable structures that are not necessarily decidable, then the effective version of the Keisler--Shelah theorem can fail in either direction.  As explained in the previous bullet, there are many examples of elementarily equivalent computable linear orders having non-isomorphic cohesive powers.  Example~\ref{ex-NonElemEquiv} also shows that it is possible for non-elementarily equivalent computable linear orders to have isomorphic cohesive powers.

\medskip

\item Classically, for a countable language, ultrapowers over countably incomplete ultrafilters (i.e., ultrafilters that are not closed under countable intersections) are always $\aleph_1$\nobreakdash-saturated.  Effectively, cohesive powers of decidable structures are recursively saturated (Theorem~\ref{thm-SatGen} item~\ref{it-DecRecSatStr}, which is essentially due to Nelson~\cite{Nelson} in a slightly different context).  Furthermore, for $n > 0$, cohesive powers of $n$\nobreakdash-decidable structures are $\Sigma_n$\nobreakdash-recursively saturated (Theorem~\ref{thm-SatGen} item~\ref{it-nDecRecSatStr}).  Most interestingly, if the cohesive set is assumed to be co-c.e., then we obtain the $n=0$ case as well as an additional level of saturation:  cohesive powers of $n$\nobreakdash-decidable structures over co-c.e.\ cohesive sets are $\Sigma_{n+1}$\nobreakdash-recursively saturated (Theorem~\ref{thm-SatCoCe}).
\end{itemize}

This work is organized as follows.  Section~\ref{sec-ProdAndPow} presents the basic theory of cohesive products and cohesive powers, focusing on analogs of {\L}o\'{s}'s theorem, substructures, saturation, and isomorphisms.  Section~\ref{sec-LOandPow} concerns the cohesive powers of computable linear orders in general.  Section~\ref{sec-PowOfOmega} concerns the cohesive powers of computable copies of $\omega$.  In Section~\ref{sec-DenseNonstd}, given a co-c.e.\ cohesive set $C$, we construct a computable copy $\mc{L}$ of $\omega$ whose cohesive power $\prod_C \mc{L}$ has order-type $\omega + \eta$.  Section~\ref{sec-Shuffle} leverages the construction of Section~\ref{sec-DenseNonstd} to shuffle various patterns of finite order-types into cohesive powers of computable copies of $\omega$.

\section{Cohesive products and powers of computable structures}\label{sec-ProdAndPow}

We assume familiarity with the basic concepts and notation from computability theory and computable structure theory.  Comprehensive references include~\cites{LermanBook, SoareBookRE, SoareBookTC} for computability theory and~\cites{AshKnightBook, MontalbanBook} for computable structure theory.  See also~\cite{FokinaHarizanovMelnikov} for a survey of computable structure theory.

Throughout, $\Nb$ denotes the natural numbers, and $\omega$ denotes its order-type when thought of as a linear order.  For each $n \geq 2$, we use $\la x_0, \dots, x_{n-1} \ra \colon \Nb^n \imp \Nb$ to denote the usual computable bijective $n$\nobreakdash-tupling function, which we may assume is increasing in all coordinates.  For each $i < n$, $\pi_i$ denotes the corresponding projection function onto coordinate $i$.  For $X \subseteq \Nb$ and $n \in \Nb$, $X \rst n$ denotes the set $X \cap \{0, 1, \dots, n-1\}$.  Often we consider expressions of the form $\lim_{n \in C} f(n)$, $\limsup_{n \in C} f(n)$, $\liminf_{n \in C} f(n)$, etc., where $f \colon \Nb \imp \Nb$ is some function and $C \subseteq \Nb$ is an infinite set.  For this, let $n_0 < n_1 < n_2 < \cdots$ be the elements of $C$ listed in increasing order.  Then $\lim_{n \in C} f(n)$ means $\lim_{i \imp \infty} f(n_i)$, and $\limsup_{n \in C} f(n)$ and $\liminf_{n \in C} f(n)$ are interpreted similarly.  Notice that for functions $f \colon \Nb \imp \Nb$, $\lim_{n \in C} f(n) = \infty$ if and only if $\liminf_{n \in C} f(n) = \infty$.

We denote partial computable functions by $\varphi$, $\psi$, etc.  For a partial computable function $\varphi$, $\varphi(n)\da$ means that $\varphi$ halts on input $n$, thus producing an output, and $\varphi(n)\ua$ means that $\varphi$ does not halt on input $n$.  The notation $\varphi \keq \psi$ means that $\varphi$ and $\psi$ are equal partial functions:  for every $n$, either $\varphi(n)\da = \psi(n)\da$ or both $\varphi(n)\ua$ and $\psi(n)\ua$.  We also use the $\keq$ notation to define one partial computable function in terms of another.  For example, `let $\varphi(n) \keq \psi(n) + 1$' means compute $\varphi(n)$ by running $\psi(n)$ and adding $1$ to the output if $\psi(n)$ halts.  As usual, $(\varphi_e)_{e \in \Nb}$ denotes the standard effective enumeration of all partial computable functions, and $\varphi_{e,s}(n)$ denotes the result (if any) of running $\varphi_e$ on input $n$ for $s$ computational steps.  Sometimes we also write $\varphi_0, \dots, \varphi_{n-1}$ to refer to an arbitrary list of partial computable functions.  The usage of subscripts will be clear from context.

Throughout, we consider only first-order languages $\mf{L}$ and finite first-order $\mf{L}$\nobreakdash-formulas.  For $k \in \Nb$, we sometimes use the abbreviation $\exists^{\geq k} x\, \Phi(x)$ to express that there are at least $k$ distinct $x$ for which $\Phi(x)$ holds:
\begin{align*}
\exists^{\geq k} x\, \Phi(x) \quad\equiv\quad \exists x_0, \dots, x_{k-1} \, \left[\left(\bigwedge_{i < j < k}x_i \neq x_j\right) \;\andd\; \left(\bigwedge_{i < k} \Phi(x_i)\right)\right].
\end{align*}
Similarly, we use the abbreviation $\exists^{= k} x\, \Phi(x)$ to express that there are exactly $k$ distinct $x$ for which $\Phi(x)$ holds:  $\exists^{= k} x\, \Phi(x) \;\equiv\; \exists^{\geq k} x\, \Phi(x) \;\andd\; \neg\exists^{\geq k+1} x\, \Phi(x)$.  We point out that, for example, if $\Phi(x)$ is a $\Sigma_1$ formula, then $\exists^{\geq k} x\, \Phi(x)$ is equivalent to a $\Sigma_1$ formula and $\exists^{= k} x\, \Phi(x)$ is equivalent to the conjunction of a $\Sigma_1$ formula and a $\Pi_1$ formula.  In a slight abuse of the terminology, we say that a formula is $\Delta_n$ if it is logically equivalent to both a $\Sigma_n$ formula and a $\Pi_n$ formula.  So if $\Phi(x)$ is $\Sigma_1$, then then $\exists^{= k} x\, \Phi(x)$ is $\Delta_2$.
 
Fix a computable language $\mf{L}$.  A computable $\mf{L}$\nobreakdash-structure $\mc{A}$ consists of a non-empty computable domain $A \subseteq \Nb$ and a uniformly computable interpretation of all relation, function, and constant symbols of $\mf{L}$.  We often denote the domain of a structure $\mc{A}$ by $|\mc{A}|$.  A computable $\mf{L}$\nobreakdash-structure $\mc{A}$ is \emph{decidable} if there is an algorithm which, given a formula $\Phi(x_0, \dots, x_{m-1})$, with all free variables displayed, and a sequence of parameters $\la a_0, \dots, a_{m-1} \ra$ each from $|\mc{A}|$, determines whether or not $\mc{A} \models \Phi(a_0, \dots, a_{m-1})$.  Likewise, a computable $\mf{L}$\nobreakdash-structure $\mc{A}$ is \emph{$n$\nobreakdash-decidable} if there is such an algorithm determining whether or not $\Sigma_n$ formulas with parameters from $|\mc{A}|$ hold in $\mc{A}$.  In other words, a computable structure is a structure having a computable atomic diagram (or, equivalently, a computable quantifier-free diagram); a decidable structure is a structure having a computable elementary diagram; and an $n$\nobreakdash-decidable structure is a structure having a computable $\Sigma_n$\nobreakdash-elementary diagram.  A $0$\nobreakdash-decidable structure is the same thing as a computable structure.  Similarly, a sequence $(\mc{A}_i : i \in \Nb)$ of $\mf{L}$\nobreakdash-structures is \emph{uniformly computable}, \emph{uniformly decidable}, or \emph{uniformly $n$\nobreakdash-decidable} if the respective sequence of atomic, elementary, or $\Sigma_n$\nobreakdash-elementary diagrams is uniformly computable.

We find it convenient to extend the decidability terminology to individual formulas and to computable sequences of formulas that are not necessarily all members of some fixed syntactic class.  Say that a computable sequence of formulas $(\Phi_i : i \in \Nb)$ is \emph{uniformly decidable} in a computable $\mf{L}$\nobreakdash-structure $\mc{A}$ if there is an algorithm that, given a subformula $\Psi(\vec{y})$ of $\Phi_i$ for some $i$ (including of course the possibility $\Psi = \Phi_i$) and an appropriate sequence of parameters $\vec{a}$ from $|\mc{A}|$, determines whether or not $\mc{A} \models \Psi(\vec{a})$.  The reason for including the decidability of subformulas is to permit inductive arguments and to ensure that we can effectively search for witnesses to existential quantifiers.  These properties are used in the subsections on analogs of {\L}o\'{s}'s theorem and on saturation below, for example.  Formally, $(\Phi_i : i \in \Nb)$ is uniformly decidable in $\mc{A}$ if the set
\begin{align*}
\{\la i, \Psi(\vec{y}), \vec{a} \ra : \text{$\Psi$ is a subformula of $\Phi_i$} \;\andd\; |\vec{a}| = |\vec{y}| \;\andd\; \mc{A} \models \Psi(\vec{a})\}
\end{align*}
is computable.  In the case of sequences of structures, say that a computable sequence of formulas $(\Phi_i : i \in \Nb)$ is \emph{uniformly decidable} in a computable sequence $(\mc{A}_n : n \in \Nb)$ of $\mf{L}$\nobreakdash-structures if the set
\begin{align*}
\{\la n, i, \Psi(\vec{y}), \vec{a} \ra : \text{$\Psi$ is a subformula of $\Phi_i$} \;\andd\; |\vec{a}| = |\vec{y}| \;\andd\; \mc{A}_n \models \Psi(\vec{a})\}
\end{align*}
is computable.  In the case of a single formula $\Phi$ (and its subformulas), we simply say that $\Phi$ is \emph{decidable} in a computable $\mf{L}$\nobreakdash-structure $\mc{A}$ and that $\Phi$ is \emph{uniformly decidable} in a uniformly computable sequence $(\mc{A}_n : n \in \Nb)$ of $\mf{L}$\nobreakdash-structures.  In a decidable structure, every computable sequence of formulas is uniformly decidable.  In an $n$\nobreakdash-decidable structure, every computable sequence of $\Sigma_n$ formulas (indeed, every computable sequence of Boolean combinations of $\Sigma_n$ formulas) is uniformly decidable.

\subsection*{Cohesive products and cohesive powers}

\begin{Definition}
An infinite set $C \subseteq \Nb$ is \emph{cohesive} if for every c.e.\ set $W$, either $C \subseteq^* W$ or $C \subseteq^* \ol{W}$.
\end{Definition}

More generally, a straightforward induction shows that if $C$ is a cohesive set and $X$ is a Boolean combination of c.e.\ sets, then either $C \subseteq^* X$ or $C \subseteq^* \ol{X}$.  Notice that if $C$ is cohesive and $X$ is a Boolean combination of c.e.\ sets, then $C \cap X$ being infinite implies that $C \subseteq^* X$.  We use quantifiers $\forae n$ and $\existsinf n$ as abbreviations for  `for almost every $n$' and `there are infinitely many $n$.'  So for example, $(\forae n \in C)(n \in X)$ means $C \subseteq^* X$.

\begin{Definition}\label{def-CohProd}
Let $\mf{L}$ be a computable language.  Let $(\mc{A}_n : n \in \Nb)$ be a uniformly computable sequence of $\mf{L}$\nobreakdash-structures with corresponding uniformly computable sequence of non-empty domains $(|\mc{A}_n| : n \in \Nb)$.  Let $C \subseteq \Nb$ be cohesive.  The \emph{cohesive product of $(\mc{A}_n : n \in \Nb)$ over $C$} is the $\mf{L}$\nobreakdash-structure $\prod_C \mc{A}_n$ defined as follows.

\begin{itemize}
\item Let $D$ be the set of partial computable functions $\varphi$ such that $\forall n \, (\varphi(n)\da \,\imp\, \varphi(n) \in |\mc{A}_n|)$ and $C \subseteq^* \dom(\varphi)$.

\medskip

\item For $\varphi, \psi \in D$, let $\varphi =_C \psi$ denote $C \subseteq^* \{n : \varphi(n)\da = \psi(n)\da\}$.  The relation $=_C$ is an equivalence relation on $D$.  Let $[\varphi]$ denote the equivalence class of $\varphi \in D$ with respect to $=_C$.

\medskip

\item The domain of $\prod_C \mc{A}_n$ is the set $|\prod_C \mc{A}_n| = \{[\varphi] : \varphi \in D\}$.

\medskip

\item 
Let $R$ be an $m$\nobreakdash-ary relation symbol of $\mf{L}$.  For $[\varphi_0], \dots, [\varphi_{m-1}] \in |\prod_C \mc{A}_n|$, define $R^{\prod_C \mc{A}_n}([\varphi_0], \dots, [\varphi_{m-1}])$ by
\begin{align*}
R^{\prod_C \mc{A}_n}([\varphi_0], \dots, [\varphi_{m-1}]) \;\Biimp\; C \subseteq^* \bigl\{n : R^{\mc{A}_n}(\varphi_0(n), \dots, \varphi_{m-1}(n))\bigr\}.
\end{align*}
Here we think of $R^{\mc{A}_n}(\varphi_0(n), \dots, \varphi_{m-1}(n))$ as including the condition that $\varphi_i(n)\da$ for each $i < m$.

\medskip

\item Let $f$ be an $m$\nobreakdash-ary function symbol of $\mf{L}$.  For $[\varphi_0], \dots, [\varphi_{m-1}] \in |\prod_C \mc{A}_n|$, let $\psi$ be the partial computable function defined by
\begin{align*}
\psi(n) \keq f^{\mc{A}_n}(\varphi_0(n), \dots, \varphi_{m-1}(n)),
\end{align*}
and notice that $C \subseteq^* \dom(\psi)$ because $C \subseteq^* \dom(\varphi_i)$ for each $i < m$.  Define $f^{\prod_C \mc{A}_n}$ by
\begin{align*}
f^{\prod_C \mc{A}_n}([\varphi_0], \dots, [\varphi_{m-1}]) = [\psi].
\end{align*}

\medskip

\item Let $c$ be a constant symbol of $\mf{L}$.  Let $\psi$ be the total computable function defined by $\psi(n) = c^{\mc{A}_n}$, and define $c^{\prod_C \mc{A}_n} = [\psi]$.
\end{itemize}

In the case where $\mc{A}_n$ is the same fixed computable structure $\mc{A}$ for every $n$, the cohesive product $\prod_C \mc{A}_n$ is called the \emph{cohesive power of $\mc{A}$ over $C$} and is denoted $\prod_C \mc{A}$.
\end{Definition}

In Definition~\ref{def-CohProd}, it is equivalent to relax the condition $\forall n \, (\varphi(n)\da \,\imp\, \varphi(n) \in |\mc{A}_n|)$ and $C \subseteq^* \dom(\varphi)$ of the first bullet to $(\forae n \in C)(\varphi(n)\da \,\andd\, \varphi(n) \in |\mc{A}_n|)$.  If $\varphi$ is a partial computable function satisfying the relaxed condition, let $\psi$ be the partial computable function given by
\begin{align*}
\psi(n) = 
\begin{cases}
\varphi(n) & \text{if $\varphi(n)\da$ and $\varphi(n) \in |\mc{A}_n|$}\\
\ua & \text{otherwise}.
\end{cases}
\end{align*}
Then $\psi$ satisfies the original condition, and $\psi =_C \varphi$.

We often consider cohesive powers of computable structures by co-c.e.\ cohesive sets.  The co-c.e.\ cohesive sets are exactly the complements of the \emph{maximal} sets, which are the co-atoms of the lattice of c.e.\ sets modulo finite difference (see~\cite{SoareBookRE}*{Section~X.3}).  Such sets exist by a well-known theorem of Friedberg (see~\cite{SoareBookRE}*{Theorem~X.3.3}).  Cohesive powers are intended to be effective analogs of ultrapowers, so in light of this analogy, it makes sense to impose effectivity on the cohesive set, which plays the role of the ultrafilter, as well as on the base structure itself.  Technically, it helps to be able to learn what numbers are not in the cohesive set $C$ when building a computable structure $\mc{A}$ so as to influence $\prod_C \mc{A}$ in a particular way.  Cohesive products by co-c.e.\ cohesive sets also have the helpful property that every member of the cohesive product has a total computable representative.  Let $(\mc{A}_n : n \in \Nb)$ be a uniformly computable sequence of structures with non-empty domain $|\mc{A}_n|$ for each $n$, and let $a_n$ be the first element of $|\mc{A}_n|$ for each $n$.  Suppose that $C$ is co-c.e.\ and cohesive, and let $\varphi \colon \Nb \imp \Nb$ be a partial computable function representing an element $[\varphi]$ of $\prod_C \mc{A}_n$.  Then $C \subseteq^* \dom(\varphi)$, so let $N$ be such that $(\forall n > N)(n \in C \imp \varphi(n)\da)$.  Define a total computable $f \colon \Nb \imp \Nb$ as follows.  If $n \leq N$, then output $f(n) = a_n$.  If $n > N$, then simultaneously run $\varphi(n)$ and enumerate the complement $\ol{C}$ of $C$.  Either $\varphi(n)\da$, $n \in \ol{C}$, or both.  If $\varphi(n)$ halts before $n$ is enumerated into $\ol{C}$, then output $f(n) = \varphi(n)$; and if $n$ is enumerated into $\ol{C}$ before $\varphi(n)$ halts, then output $f(n) = a_n$.  This $f$ is total and satisfies $f =_C \varphi$.

As with structures and their ultrapowers, a computable structure $\mc{A}$ always naturally embeds into its cohesive powers.  For $a \in |\mc{A}|$, let $f_a$ be the total computable function with constant value $a$.  Then for any cohesive set $C$, the map $a \mapsto [f_a]$ embeds $\mc{A}$ into $\prod_C \mc{A}$.  This map is called the \emph{canonical embedding} of $\mc{A}$ into $\prod_C \mc{A}$.  If $\mc{A}$ is finite and $C$ is cohesive, then every partial computable function $\varphi \colon \Nb \imp |\mc{A}|$ with $C \subseteq^* \dom(\varphi)$ is eventually constant on $C$.  In this case, every element of $\prod_C \mc{A}$ is in the range of the canonical embedding, and therefore $\mc{A} \iso \prod_C \mc{A}$.  If $\mc{A}$ is an infinite computable structure, then every cohesive power $\prod_C \mc{A}$ is countably infinite:  infinite because $\mc{A}$ embeds into $\prod_C \mc{A}$, and countable because the elements of $\prod_C \mc{A}$ are represented by partial computable functions.  See~\cite{DimitrovCohPow} for further details.

\subsection*{Analogs of {\L}o\'{s}'s theorem}

A restricted form of {\L}o\'{s}'s theorem holds for cohesive powers.  Let $\mf{L}$ be a computable language, let $\mc{A}$ be an $n$\nobreakdash-decidable $\mf{L}$\nobreakdash-structure, and let $C$ be cohesive.  We show that for every $\Sigma_{n+3}$ sentence $\Phi$, $\mc{A} \models \Phi$ implies that $\prod_C \mc{A} \models \Phi$.  It follows that $\mc{A}$ and $\prod_C \mc{A}$ agree on all $\Delta_{n+3}$ sentences.  It also follows that if $\mc{A}$ is decidable, then $\mc{A}$ and $\prod_C \mc{A}$ are elementarily equivalent.  Thus for decidable structures, we recover the full {\L}o\'{s} theorem (for first-order logic).  These results update Dimitrov's \emph{fundamental theorem of cohesive powers}~\cite{DimitrovCohPow}, which is the $0$\nobreakdash-decidable case:  if $\mc{A}$ is a computable $\mf{L}$\nobreakdash-structure, $C$ is cohesive, and $\Phi$ is a $\Sigma_3$ sentence, then $\mc{A} \models \Phi$ implies that $\prod_C \mc{A} \models \Phi$.  In general, the fundamental theorem of cohesive powers is the best possible analog of {\L}o\'{s}'s theorem.  In Sections~\ref{sec-PowOfOmega},~\ref{sec-DenseNonstd},~and~\ref{sec-Shuffle}, we see several examples of computable linear orders $\mc{L}$ where the $\Pi_3$ sentence ``every element has an immediate successor'' is true of $\mc{L}$ but false of some cohesive power of $\mc{L}$.  For sequences of structures, we show that if $(\mc{A}_i : i \in \Nb)$ is a sequence of uniformly $n$\nobreakdash-decidable $\mf{L}$\nobreakdash-structures, $C$ is cohesive, and $\Phi$ is a $\Pi_{n+2}$ sentence, then $C \subseteq^* \bigl\{i : \mc{A}_i \models \Phi \bigr\}$ implies that $\prod_C \mc{A}_i \models \Phi$.  This updates the \emph{fundamental theorem of cohesive products} from~\cite{DimitrovHarizanovSurvey}, which is the $0$\nobreakdash-decidable case.  In Section~\ref{sec-PowOfOmega}, we show that the fundamental theorem of cohesive products is best possible by uniformly computing a sequence of finite linear orders $(\mc{L}_i : i \in \Nb)$ such that for every cohesive set $C$, the cohesive product $\prod_C \mc{L}_i$ is a linear order with no maximum element.

\begin{Lemma}\label{lem-LosTerm}
Let $\mf{L}$ be a computable language, let $(\mc{A}_n : n \in \Nb)$ be a uniformly computable sequence of $\mf{L}$\nobreakdash-structures, and let $C$ be cohesive.  Let $t(v_0, \dots, v_{m-1})$ be a term, with all variables displayed.  Let $[\varphi_0], \dots, [\varphi_{m-1}] \in |\prod_C \mc{A}_n|$.  Let $\psi$ be the partial computable function given by
\begin{align*}
\psi(n) \keq t^{\mc{A}_n}(\varphi_0(n), \dots, \varphi_{m-1}(n)).
\end{align*}
Then 
\begin{align*}
t^{\prod_C \mc{A}_n}([\varphi_0], \dots, [\varphi_{m-1}]) = [\psi].
\end{align*}
\end{Lemma}

\begin{proof}
The proof is a straightforward induction on the construction of the term $t$, using Definition~\ref{def-CohProd}.
\end{proof}

\begin{Lemma}\label{lem-LosDec}
Let $\mf{L}$ be a computable language, let $(\mc{A}_n : n \in \Nb)$ be a uniformly computable sequence of $\mf{L}$\nobreakdash-structures, and let $C$ be cohesive.  Let $\Phi(v_0, \dots, v_{m-1})$ be a formula (with all free variables displayed) that is uniformly decidable in $(\mc{A}_n : n \in \Nb)$.  Then for any $[\varphi_0], \dots, [\varphi_{m-1}] \in |\prod_C \mc{A}_n|$,
\begin{align*}
\prod\nolimits_C \mc{A}_n \models \Phi([\varphi_0], \dots, [\varphi_{m-1}]) \quad\Biimp\quad C \subseteq^* \bigl\{n : \mc{A}_n \models \Phi(\varphi_0(n), \dots, \varphi_{m-1}(n))\bigr\}.
\end{align*}

\begin{proof}
Proceed by induction on the construction of the formula $\Phi$.  For the base case, assume that $\Phi(v_0, \dots, v_{m-1})$ is the atomic formula $R\bigl(t_0(v_0, \dots, v_{m-1}), \dots, t_{\ell-1}(v_0, \dots, v_{m-1})\bigr)$, where $R$ is either a relation symbol from $\mf{L}$ or equality, and $t_0, \dots, t_{\ell-1}$ are terms whose variables are among $v_0, \dots, v_{m-1}$.  For each $i < \ell$, define
\begin{align*}
\psi_i(n) \keq t_i^{\mc{A}_n}(\varphi_0(n), \dots, \varphi_{m-1}(n))
\end{align*}
as in Lemma~\ref{lem-LosTerm} so that $t_i^{\prod_C \mc{A}_n}([\varphi_0], \dots, [\varphi_{m-1}]) = [\psi_i]$.  Then
\begin{align*}
&& \prod\nolimits_C \mc{A}_n &\models R\bigl(t_0([\varphi_0], \dots, [\varphi_{m-1}]), \dots, t_{\ell-1}([\varphi_0], \dots, [\varphi_{m-1}])\bigr)\\
&\Biimp& \prod\nolimits_C \mc{A}_n &\models R([\psi_0], \dots, [\psi_{\ell-1}])\\
&\Biimp& C &\subseteq^* \bigl\{n : \mc{A}_n \models R(\psi_0(n), \dots, \psi_{\ell-1}(n))\bigr\}\\
&\Biimp& C &\subseteq^* \bigl\{n : \mc{A}_n \models R\bigl(t_0(\varphi_0(n), \dots, \varphi_{m-1}(n)), \dots, t_{\ell-1}(\varphi_0(n), \dots, \varphi_{m-1}(n))\bigr)\bigr\}.
\end{align*}

We consider the inductive cases for the connectives $\andd$ and $\neg$ and for the quantifier $\exists$.

Assume that $\Phi(v_0, \dots, v_{m-1})$ is the formula $\Phi_0(v_0, \dots, v_{m-1}) \andd \Phi_1(v_0, \dots, v_{m-1})$.  Then $\prod_C \mc{A}_n \models \Phi([\varphi_0], \dots, [\varphi_{m-1}])$ if and only if $(\forall i < 2) \bigl(\prod_C \mc{A}_n \models \Phi_i([\varphi_0], \dots, [\varphi_{m-1}])\bigr)$.  Formulas $\Phi_0$ and $\Phi_1$ are uniformly decidable in $(\mc{A}_n : n \in \Nb)$ because they are subformulas of $\Phi$.  Thus the induction hypothesis yields that
\begin{align*}
\prod\nolimits_C \mc{A}_n \models \Phi_i([\varphi_0], \dots, [\varphi_{m-1}]) \quad\Biimp\quad C \subseteq^* \bigl\{n : \mc{A}_n \models \Phi_i(\varphi_0(n), \dots, \varphi_{m-1}(n))\bigr\}
\end{align*}
for each $i < 2$.  Finally, $C \subseteq^* \bigl\{n : \mc{A}_n \models \Phi_i(\varphi_0(n), \dots, \varphi_{m-1}(n))\bigr\}$ holds for both $i=0$ and $i=1$ if and only if
\begin{align*}
C \subseteq^* \bigl\{n : \mc{A}_n \models \Phi_0(\varphi_0(n), \dots, \varphi_{m-1}(n)) \andd \Phi_1(\varphi_0(n), \dots, \varphi_{m-1}(n))\bigr\}.
\end{align*}
Putting this all together yields that
\begin{align*}
\prod\nolimits_C \mc{A}_n \models \Phi_0([\varphi_0], \dots, [\varphi_{m-1}]) \andd \Phi_1([\varphi_0], \dots, [\varphi_{m-1}])
\end{align*}
if and only if
\begin{align*}
C \subseteq^* \bigl\{n : \mc{A}_n \models \Phi_0(\varphi_0(n), \dots, \varphi_{m-1}(n)) \andd \Phi_1(\varphi_0(n), \dots, \varphi_{m-1}(n))\bigr\}.
\end{align*}

Now assume that $\Phi(v_0, \dots, v_{m-1})$ is the formula $\neg \Psi(v_0, \dots, v_{m-1})$.  Then $\prod_C \mc{A}_n \models \Phi([\varphi_0], \dots, [\varphi_{m-1}])$ if and only if $\prod_C \mc{A}_n \nmodels \Psi([\varphi_0], \dots, [\varphi_{m-1}])$.  The formula $\Psi$ is uniformly decidable in $(\mc{A}_n : n \in \Nb)$ because it is a subformula of $\Phi$.  Thus the induction hypothesis yields that
\begin{align*}
\prod\nolimits_C \mc{A}_n \nmodels \Psi([\varphi_0], \dots, [\varphi_{m-1}]) \quad\Biimp\quad C \nsubseteq^* \bigl\{n : \mc{A}_n \models \Psi(\varphi_0(n), \dots, \varphi_{m-1}(n))\bigr\}.
\end{align*}
We have that $C \nsubseteq^* \bigl\{n : \mc{A}_n \models \Psi(\varphi_0(n), \dots, \varphi_{m-1}(n))\bigr\}$ if and only if $C$ has infinite intersection with the set $\bigl\{n : \mc{A}_n \nmodels \Psi(\varphi_0(n), \dots, \varphi_{m-1}(n))\bigr\}$, which is co-c.e.\ because $\Psi(v_0, \dots, v_{m-1})$ is uniformly decidable in $(\mc{A}_n : n \in \Nb)$.  By cohesiveness and the fact that $C \subseteq^* \dom(\varphi_i)$ for each $i < m$, we therefore have that
\begin{align*}
C \nsubseteq^* \bigl\{n : \mc{A}_n \models \Psi(\varphi_0(n), \dots, \varphi_{m-1}(n))\bigr\} \quad&\Biimp\quad C \subseteq^* \bigl\{n : \mc{A}_n \nmodels \Psi(\varphi_0(n), \dots, \varphi_{m-1}(n))\bigr\}\\
\quad&\Biimp\quad C \subseteq^* \bigl\{n : \mc{A}_n \models \neg\Psi(\varphi_0(n), \dots, \varphi_{m-1}(n))\bigr\}.
\end{align*}
Putting this all together yields that
\begin{align*}
\prod\nolimits_C \mc{A}_n \models \neg\Psi([\varphi_0], \dots, [\varphi_{m-1}]) \quad\Biimp\quad C \subseteq^* \bigl\{n : \mc{A}_n \models \neg\Psi(\varphi_0(n), \dots, \varphi_{m-1}(n))\bigr\}.
\end{align*}

Finally, assume that $\Phi(v_0, \dots, v_{m-1})$ is the formula $\exists x\, \Psi(x, v_0, \dots, v_{m-1})$.  First suppose that $\prod_C \mc{A}_n \models \exists x\, \Psi(x, [\varphi_0], \dots, [\varphi_{m-1}])$.  Then there is a $[\theta] \in |\prod_C \mc{A}_n|$ such that $\prod_C \mc{A}_n \models \Psi([\theta], [\varphi_0], \dots, [\varphi_{m-1}])$.  The formula $\Psi$ is uniformly decidable in $(\mc{A}_n : n \in \Nb)$ because it is a subformula of $\Phi$.  Thus the induction hypothesis yields that
\begin{align*}
\prod\nolimits_C \mc{A}_n \models \Psi([\theta], [\varphi_0], \dots, [\varphi_{m-1}]) \quad\Biimp\quad C \subseteq^* \bigl\{n : \mc{A}_n \models \Psi(\theta(n), \varphi_0(n), \dots, \varphi_{m-1}(n))\bigr\}.
\end{align*}
Clearly
\begin{align*}
\bigl\{n : \mc{A}_n \models \Psi(\theta(n), \varphi_0(n), \dots, \varphi_{m-1}(n))\bigr\} \subseteq \bigl\{n : \mc{A}_n \models \exists x\, \Psi(x, \varphi_0(n), \dots, \varphi_{m-1}(n))\bigr\},
\end{align*}
so we have that $C \subseteq^* \bigl\{n : \mc{A}_n \models \exists x\, \Psi(x, \varphi_0(n), \dots, \varphi_{m-1}(n))\bigr\}$.

Conversely, suppose that $C \subseteq^* \bigl\{n : \mc{A}_n \models \exists x\, \Psi(x, \varphi_0(n), \dots, \varphi_{m-1}(n))\bigr\}$.  The formula $\Psi(x, v_0, \dots, v_{m-1})$ is uniformly decidable in $(\mc{A}_n : n \in \Nb)$, so we may define a partial computable function $\theta$ by
\begin{align*}
\theta(n) \keq \text{the first $a \in |\mc{A}_n|$ such that $\mc{A}_n \models \Psi(a, \varphi_0(n), \dots, \varphi_{m-1}(n))$}.
\end{align*}
The assumption $C \subseteq^* \bigl\{n : \mc{A}_n \models \exists x\, \Psi(x, \varphi_0(n), \dots, \varphi_{m-1}(n))\bigr\}$ implies that $C \subseteq^* \dom(\theta)$.  We therefore have that $[\theta] \in |\prod_C \mc{A}_n|$ and that
\begin{align*}
C \subseteq^* \bigl\{n : \mc{A}_n \models \Psi(\theta(n), \varphi_0(n), \dots, \varphi_{m-1}(n))\bigr\}.
\end{align*}
The induction hypothesis then yields that $\prod_C \mc{A}_n \models \Psi([\theta], [\varphi_0], \dots, [\varphi_{m-1}])$.  Therefore $\prod_C \mc{A}_n \models \exists x\, \Psi(x, [\varphi_0], \dots, [\varphi_{m-1}])$.  This completes the proof.
\end{proof}
\end{Lemma}

\begin{Lemma}\label{lem-LosProdParamHelper}
Let $\mf{L}$ be a computable language, let $(\mc{A}_n : n \in \Nb)$ be a uniformly computable sequence of $\mf{L}$\nobreakdash-structures, and let $C$ be cohesive.  Let $\Psi(x_0, \dots, x_{k-1}, y_0, \dots, y_{\ell-1}, v_0, \dots, v_{m-1})$ be a formula that is uniformly decidable in $(\mc{A}_n : n \in \Nb)$.

\begin{enumerate}[(1)]
\item\label{it-LosProdPramSig2Helper} For any $[\varphi_0], \dots, [\varphi_{m-1}] \in |\prod_C \mc{A}_n|$,
\begin{align*}
\prod\nolimits_C \mc{A}_n \models \exists \vec{x}\, \forall \vec{y}\, \Psi(\vec{x}, \vec{y}, [\varphi_0], \dots, [\varphi_{m-1}]) \quad\Imp\quad C \subseteq^* \bigl\{n : \mc{A}_n \models \exists \vec{x}\, \forall \vec{y}\, \Psi(\vec{x}, \vec{y}, \varphi_0(n), \dots, \varphi_{m-1}(n))\bigr\}.
\end{align*}

\medskip

\item\label{it-LosProdPramPi2Helper} For any $[\varphi_0], \dots, [\varphi_{m-1}] \in |\prod_C \mc{A}_n|$,
\begin{align*}
C \subseteq^* \bigl\{n : \mc{A}_n \models \forall \vec{x}\, \exists \vec{y}\, \Psi(\vec{x}, \vec{y}, \varphi_0(n), \dots, \varphi_{m-1}(n))\bigr\} \quad\Imp\quad \prod\nolimits_C \mc{A}_n \models \forall \vec{x}\, \exists \vec{y}\, \Psi(\vec{x}, \vec{y}, [\varphi_0], \dots, [\varphi_{m-1}]).
\end{align*}
\end{enumerate}
\end{Lemma}

\begin{proof}
For item~\ref{it-LosProdPramSig2Helper}, suppose that $\prod_C \mc{A}_n \models \exists \vec{x}\, \forall \vec{y}\, \Psi(\vec{x}, \vec{y}, [\varphi_0], \dots, [\varphi_{m-1}])$.  Let $[\psi_0], \dots, [\psi_{k-1}] \in |\prod_C \mc{A}_n|$ be such that
\begin{align*}
\prod\nolimits_C \mc{A}_n \models \forall \vec{y}\, \Psi([\psi_0], \dots, [\psi_{k-1}], \vec{y}, [\varphi_0], \dots, [\varphi_{m-1}]). \tag{$*$}\label{eq-ModelsForAll}
\end{align*}
The set
\begin{align*}
X = \bigl\{n : \mc{A}_n \models \forall \vec{y}\, \Psi(\psi_0(n), \dots, \psi_{k-1}(n), \vec{y}, \varphi_0(n), \dots, \varphi_{m-1}(n))\bigr\}
\end{align*}
is a Boolean combination of c.e.\ sets, so by cohesiveness, either $C \subseteq^* X$ or $C \subseteq^* \ol{X}$.  If $C \subseteq^* \ol{X}$, then because $C \subseteq^* \dom(\varphi_i)$ for each $i < m$ and $C \subseteq^* \dom(\psi_i)$ for each $i < k$, we would have that
\begin{align*}
C \subseteq^* \bigl\{n : \mc{A}_n \models \exists \vec{y}\, \neg \Psi(\psi_0(n), \dots, \psi_{k-1}(n), \vec{y}, \varphi_0(n), \dots, \varphi_{m-1}(n))\bigr\}.
\end{align*}
As $\Psi$ is uniformly decidable in $(\mc{A}_n : n \in \Nb)$, we could then argue as in the $\exists$ case of the proof of Lemma~\ref{lem-LosDec} and simultaneously define partial computable functions $\theta_i$ for $i < \ell$ as follows.  Given $n$, search for the first sequence $\la a_0, \dots, a_{\ell-1} \ra \in |\mc{A}_n|^\ell$ such that
\begin{align*}
\mc{A}_n \models \neg \Psi(\psi_0(n), \dots, \psi_{k-1}(n), a_0, \dots, a_{\ell-1}, \varphi_0(n), \dots, \varphi_{m-1}(n)),
\end{align*}
and set $\theta_i(n) = a_i$ for each $i < \ell$.  Then
\begin{align*}
C \subseteq^* \bigl\{n : \mc{A}_n \models \neg \Psi(\psi_0(n), \dots, \psi_{k-1}(n), \theta_0(n), \dots, \theta_{\ell-1}(n), \varphi_0(n), \dots, \varphi_{m-1}(n))\bigr\},
\end{align*}
so
\begin{align*}
\prod\nolimits_C \mc{A}_n \models \neg \Psi([\psi_0], \dots, [\psi_{k-1}], [\theta_0], \dots, [\theta_{\ell-1}], [\varphi_0], \dots, [\varphi_{m-1}])
\end{align*}
by Lemma~\ref{lem-LosDec}, which contradicts~\eqref{eq-ModelsForAll}.  Thus we cannot have $C \subseteq^* \ol{X}$, so it must be that $C \subseteq^* X$.  Therefore
\begin{align*}
C &\subseteq^* \bigl\{n : \mc{A}_n \models \forall \vec{y}\, \Psi(\psi_0(n), \dots, \psi_{k-1}(n), \vec{y}, \varphi_0(n), \dots, \varphi_{m-1}(n))\bigr\}\\
&\subseteq^{\phantom{*}} \bigl\{n : \mc{A}_n \models \exists \vec{x}\, \forall \vec{y}\, \Psi(\vec{x}, \vec{y}, \varphi_0(n), \dots, \varphi_{m-1}(n))\bigr\},
\end{align*}
as desired.

Item~\ref{it-LosProdPramPi2Helper} follows from item~\ref{it-LosProdPramSig2Helper}.  Suppose that $\prod_C \mc{A}_n \nmodels \forall \vec{x}\, \exists \vec{y}\, \Psi(\vec{x}, \vec{y}, [\varphi_0], \dots, [\varphi_{m-1}])$.  Then $\prod_C \mc{A}_n \models \exists \vec{x}\, \forall \vec{y}\, \neg\Psi(\vec{x}, \vec{y}, [\varphi_0], \dots, [\varphi_{m-1}])$.  The formula $\neg\Psi$ is uniformly decidable in $(\mc{A}_n : n \in \Nb)$ because $\Psi$ is, so applying item~\ref{it-LosProdPramSig2Helper} to $\neg\Psi$ yields that
\begin{align*}
C \subseteq^* \bigl\{n : \mc{A}_n \models \exists \vec{x}\, \forall \vec{y}\, \neg\Psi(\vec{x}, \vec{y}, \varphi_0(n), \dots, \varphi_{m-1}(n))\bigr\}.
\end{align*}
Therefore
\begin{align*}
C \nsubseteq^* \bigl\{n : \mc{A}_n \models \forall \vec{x}\, \exists \vec{y}\, \Psi(\vec{x}, \vec{y}, \varphi_0(n), \dots, \varphi_{m-1}(n))\bigr\}
\end{align*}
because the two sets are disjoint.
\end{proof}

\begin{Lemma}\label{lem-LosProdParamDelta2Helper}
Let $\mf{L}$ be a computable language, let $(\mc{A}_n : n \in \Nb)$ be a uniformly computable sequence of $\mf{L}$\nobreakdash-structures, and let $C$ be cohesive.  Let $\Phi(v_0, \dots, v_{m-1})$ be a formula that is logically equivalent to a formula of the form $\exists \vec{x}\, \forall \vec{y}\, \Psi_0(\vec{x}, \vec{y}, v_0, \dots, v_{m-1})$ and to a formula of the form $\forall \vec{x}\, \exists \vec{y}\, \Psi_1(\vec{x}, \vec{y}, v_0, \dots, v_{m-1})$, where $\Psi_0$ and $\Psi_1$ are uniformly decidable in $(\mc{A}_n : n \in \Nb)$.  Then for any $[\varphi_0], \dots, [\varphi_{m-1}] \in |\prod_C \mc{A}_n|$,
\begin{align*}
\prod\nolimits_C \mc{A}_n \models \Phi([\varphi_0], \dots, [\varphi_{m-1}]) \quad\Biimp\quad C \subseteq^* \bigl\{n : \mc{A}_n \models \Phi(\varphi_0(n), \dots, \varphi_{m-1}(n))\bigr\}.
\end{align*}
\end{Lemma}

\begin{proof}
The `$\Imp$' implication is by Lemma~\ref{lem-LosProdParamHelper} item~\ref{it-LosProdPramSig2Helper} and the equivalence of $\Phi(\vec{v})$ with $\exists \vec{x}\, \forall \vec{y}\, \Psi_0(\vec{x}, \vec{y}, \vec{v})$.  The `$\Leftarrow$' implication is by Lemma~\ref{lem-LosProdParamHelper} item~\ref{it-LosProdPramPi2Helper} and the equivalence of $\Phi(\vec{v})$ with $\forall \vec{x}\, \exists \vec{y}\, \Psi_1(\vec{x}, \vec{y}, \vec{v})$.
\end{proof}

The following theorem refines the \emph{fundamental theorem of cohesive products} from~\cite{DimitrovHarizanovSurvey}.

\begin{Theorem}\label{thm-LosProdParam}
Let $\mf{L}$ be a computable language, let $(\mc{A}_i : i \in \Nb)$ be a sequence of uniformly $n$\nobreakdash-decidable $\mf{L}$\nobreakdash-structures, and let $C$ be cohesive.
\begin{enumerate}[(1)]
\item\label{it-LosProdPramSig2} Let $\Phi(v_0, \dots, v_{m-1})$ be a $\Sigma_{n+2}$ formula.  Then for any $[\varphi_0], \dots, [\varphi_{m-1}] \in |\prod_C \mc{A}_i|$,
\begin{align*}
\prod\nolimits_C \mc{A}_i \models \Phi([\varphi_0], \dots, [\varphi_{m-1}]) \quad\Imp\quad C \subseteq^* \bigl\{i : \mc{A}_i \models \Phi(\varphi_0(i), \dots, \varphi_{m-1}(i))\bigr\}.
\end{align*}

\medskip

\item\label{it-LosProdPramPi2} Let $\Phi(v_0, \dots, v_{m-1})$ be a $\Pi_{n+2}$ formula.  Then for any $[\varphi_0], \dots, [\varphi_{m-1}] \in |\prod_C \mc{A}_i|$,
\begin{align*}
C \subseteq^* \bigl\{i : \mc{A}_i \models \Phi(\varphi_0(i), \dots, \varphi_{m-1}(i))\bigr\} \quad\Imp\quad \prod\nolimits_C \mc{A}_i \models \Phi([\varphi_0], \dots, [\varphi_{m-1}]).
\end{align*}

\medskip

\item\label{it-LosProdPramDelta2} Let $\Phi(v_0, \dots, v_{m-1})$ be a $\Delta_{n+2}$ formula.  Then for any $[\varphi_0], \dots, [\varphi_{m-1}] \in |\prod_C \mc{A}_i|$,
\begin{align*}
\prod\nolimits_C \mc{A}_i \models \Phi([\varphi_0], \dots, [\varphi_{m-1}]) \quad\Biimp\quad C \subseteq^* \bigl\{i : \mc{A}_i \models \Phi(\varphi_0(i), \dots, \varphi_{m-1}(i))\bigr\}.
\end{align*}
\end{enumerate}
\end{Theorem}

\begin{proof}
Item~\ref{it-LosProdPramSig2} follows from Lemma~\ref{lem-LosProdParamHelper} item~\ref{it-LosProdPramSig2Helper} because a $\Sigma_{n+2}$ formula $\Phi(v_0, \dots, v_{m-1})$ has the form $\exists \vec{x}\, \forall \vec{y}\, \Psi(\vec{x}, \vec{y}, v_0, \dots, v_{m-1})$, where $\Psi(\vec{x}, \vec{y}, v_0, \dots, v_{m-1})$ is $\Sigma_n$ and hence is uniformly decidable in the uniformly $n$\nobreakdash-decidable sequence $(\mc{A}_i : i \in \Nb)$.  Likewise, item~\ref{it-LosProdPramPi2} follows from Lemma~\ref{lem-LosProdParamHelper} item~\ref{it-LosProdPramPi2Helper}, and item~\ref{it-LosProdPramDelta2} follows from Lemma~\ref{lem-LosProdParamDelta2Helper}.
\end{proof}

Proposition~\ref{prop-LosProdParamTight} below shows that Theorem~\ref{thm-LosProdParam} is tight in general.  We now switch from cohesive products to cohesive powers.

\begin{Lemma}\label{lem-LosSigma3}
Let $\mf{L}$ be a computable language, let $\mc{A}$ be a computable $\mf{L}$\nobreakdash-structure, and let $C$ be cohesive.  Let $\Phi$ be a sentence of the form $\forall \vec{x}\, \exists \vec{y}\, \forall \vec{z} \, \Psi(\vec{x}, \vec{y}, \vec{z})$, where $\Psi(\vec{x}, \vec{y}, \vec{z})$ is decidable in $\mc{A}$.  Then
\begin{align*}
\prod\nolimits_C \mc{A} \models \Phi \quad\Imp\quad \mc{A} \models \Phi.
\end{align*}
\end{Lemma}

\begin{proof}
Suppose that $\prod_C \mc{A} \models \Phi$.  Write $\vec{x} = x_0, \dots, x_{m-1}$.  For each $i < m$, fix an $a_i \in |\mc{A}|$, and let $\varphi_i$ be the constant function with value $a_i$.  Then $[\varphi_0], \dots, [\varphi_{m-1}] \in |\prod_C \mc{A}|$, so
\begin{align*}
\prod\nolimits_C \mc{A} \models \exists \vec{y}\, \forall \vec{z}\, \Psi([\varphi_0], \dots, [\varphi_{m-1}], \vec{y}, \vec{z}).
\end{align*}
Therefore
\begin{align*}
C \subseteq^* \bigl\{n : \mc{A} \models \exists \vec{y}\, \forall \vec{z}\, \Psi(\varphi_0(n), \dots, \varphi_{m-1}(n), \vec{y}, \vec{z})\bigr\} = \bigl\{n : \mc{A} \models \exists \vec{y}\, \forall \vec{z}\, \Psi(a_0, \dots, a_{m-1}, \vec{y}, \vec{z})\bigr\}
\end{align*}
by Lemma~\ref{lem-LosProdParamHelper} item~\ref{it-LosProdPramSig2Helper} applied to the formula $\exists \vec{y}\, \forall \vec{z} \, \Psi(\vec{x}, \vec{y}, \vec{z})$ and the sequence of structures $(\mc{A}_n : n \in \Nb)$ where $\mc{A}_n$ is $\mc{A}$ for each $n$.  It must therefore be that $\mc{A} \models \exists \vec{y}\, \forall \vec{z}\, \Psi(a_0, \dots, a_{m-1}, \vec{y}, \vec{z})$.  The sequence $a_0, \dots, a_{m-1} \in |\mc{A}|$ was arbitrary, so we have shown that $\mc{A} \models \forall \vec{x}\, \exists \vec{y}\, \forall \vec{z} \, \Psi(\vec{x}, \vec{y}, \vec{z})$.  That is, $\mc{A} \models \Phi$.
\end{proof}

The next theorem is our version of {\L}o\'{s}'s theorem for cohesive powers of $n$\nobreakdash-decidable structures, which refines the \emph{fundamental theorem of cohesive powers} from~\cite{DimitrovCohPow}.

\begin{Theorem}\label{thm-LosGeneral}
Let $\mf{L}$ be a computable language, let $\mc{A}$ be an $n$\nobreakdash-decidable $\mf{L}$\nobreakdash-structure, and let $C$ be a cohesive set.

\begin{enumerate}[(1)]
\item\label{it-LosDelta2Param} Let $\Phi(v_0, \dots, v_{m-1})$ be a $\Delta_{n+2}$ formula.  Then for any $[\varphi_0], \dots, [\varphi_{m-1}] \in |\prod_C \mc{A}|$,
\begin{align*}
\prod\nolimits_C \mc{A} \models \Phi([\varphi_0], \dots, [\varphi_{m-1}]) \quad\Biimp\quad C \subseteq^* \bigl\{n : \mc{A} \models \Phi(\varphi_0(n), \dots, \varphi_{m-1}(n))\bigr\}.
\end{align*}

\medskip

\item\label{it-LosDelta3Sent} Let $\Phi$ be a $\Delta_{n+3}$ sentence.  Then $\mc{A} \models \Phi$ if and only if $\prod_C \mc{A} \models \Phi$.

\medskip

\item\label{it-LosSigma3Sent} Let $\Phi$ be a $\Sigma_{n+3}$ sentence.  If $\mc{A} \models \Phi$, then $\prod_C \mc{A} \models \Phi$.
\end{enumerate}
\end{Theorem}

\begin{proof}
Item~\ref{it-LosDelta2Param} is the special case of Theorem~\ref{thm-LosProdParam} item~\ref{it-LosProdPramDelta2} in which each structure $\mc{A}_i$ is $\mc{A}$.  Item~\ref{it-LosDelta3Sent} follows from item~\ref{it-LosSigma3Sent} because if $\Phi$ is a $\Delta_{n+3}$ sentence, then both $\Phi$ and $\neg\Phi$ are logically equivalent to $\Sigma_{n+3}$ sentences.  For item~\ref{it-LosSigma3Sent}, consider the sentence $\neg \Phi$, which is logically equivalent to a $\Pi_{n+3}$ sentence $\Theta$.  The sentence $\Theta$ thus has the form $\forall \vec{x}\, \exists \vec{y}\, \forall \vec{z} \, \Psi(\vec{x}, \vec{y}, \vec{z})$, where $\Psi(\vec{x}, \vec{y}, \vec{z})$ is $\Sigma_n$.   The formula $\Psi(\vec{x}, \vec{y}, \vec{z})$ is decidable in $\mc{A}$ because $\mc{A}$ is $n$\nobreakdash-decidable.  Therefore $\prod_C \mc{A} \models \Theta$ implies that $\mc{A} \models \Theta$ by Lemma~\ref{lem-LosSigma3}.  The contrapositive yields that $\mc{A} \models \Phi$ implies that $\prod_C \mc{A} \models \Phi$.
\end{proof}

Corollary~\ref{cor-LosSigma3Tight} below shows that Theorem~\ref{thm-LosGeneral} item~\ref{it-LosSigma3Sent} is tight in general.  As mentioned above, we recover {\L}o\'{s}'s theorem for all first-order sentences when we consider cohesive powers of decidable structures.  This is essentially the same as Nelson's~\cite{Nelson}*{Theorem~0.5}.  See also~\cite{DimitrovCohPow}.

\begin{Corollary}\label{cor-DecLos}
Let $\mf{L}$ be a computable language, and let $C$ be a cohesive set.
\begin{enumerate}[(1)]
\item\label{it-FullLosSeq} Let $(\mc{A}_i : i \in \Nb)$ be a sequence of uniformly decidable $\mf{L}$\nobreakdash-structures, and let $\Phi(v_0, \dots, v_{m-1})$ be a formula with all free variables displayed.  Then for any $[\varphi_0], \dots, [\varphi_{m-1}] \in |\prod_C \mc{A}_i|$,
\begin{align*}
\prod\nolimits_C \mc{A}_i \models \Phi([\varphi_0], \dots, [\varphi_{m-1}]) \quad\Biimp\quad C \subseteq^* \bigl\{i : \mc{A}_i \models \Phi(\varphi_0(i), \dots, \varphi_{m-1}(i))\bigr\}.
\end{align*}
In particular, if $\Phi$ is a sentence, then
\begin{align*}
\prod\nolimits_C \mc{A}_i \models \Phi \quad\Biimp\quad C \subseteq^* \{i : \mc{A}_i \models \Phi\}.
\end{align*}

\medskip

\item\label{it-FullLosStr} If $\mc{A}$ is a decidable $\mf{L}$\nobreakdash-structure, then $\mc{A}$ and $\prod_C \mc{A}$ are elementarily equivalent.
\end{enumerate}
\end{Corollary}

\begin{proof}
Item~\ref{it-FullLosSeq} follows from Theorem~\ref{thm-LosProdParam} because the structures $(\mc{A}_i : i \in \Nb)$ are uniformly $n$\nobreakdash-decidable for every $n$.  Item~\ref{it-FullLosStr} follows from the special case of item~\ref{it-FullLosSeq} in which $\mc{A}_i$ is $\mc{A}$ for each $i$.  Item~\ref{it-FullLosStr} also follows from Theorem~\ref{thm-LosGeneral} item~\ref{it-LosDelta3Sent} because $\mc{A}$ is $n$\nobreakdash-decidable for every $n$.
\end{proof}

We pause to point out that we can recover a version of Skolem's countable non-standard model of arithmetic by relativizing everything to $0^{(\omega)}$, the $\omega$\textsuperscript{th} Turing jump of $0$.  Let $\mc{N} = (\Nb, +, \times, <)$ denote the standard model of arithmetic.  Then $\mc{N}$ is a decidable structure relative to $0^{(\omega)}$.  Therefore $\mc{N} \equiv \prod_C^{0^{(\omega)}} \mc{N}$ for any $C$ that is cohesive relative to $0^{(\omega)}$ by the relativized version of Corollary~\ref{cor-DecLos}.  Thus $\prod_C^{0^{(\omega)}} \mc{N}$ is a countable non-standard model of arithmetic.  The superscript $0^{(\omega)}$ in $\prod_C^{0^{(\omega)}} \mc{N}$ indicates that we relativize the cohesive power construction to $0^{(\omega)}$ by requiring that $C$ be cohesive for the collection of sets that are c.e.\ relative to $0^{(\omega)}$ and by building the cohesive power from functions that are partial computable relative to $0^{(\omega)}$.

\subsection*{Reducts, substructures, and disjoint unions}

Cohesive products respect reducts of computable structures.  Let $\mf{L} \subseteq \mf{L}^+$ be two languages, and let $\mc{A}$ be an $\mf{L}^+$\nobreakdash-structure.  Then the \emph{reduct} $\mc{A} \rst \mf{L}$ of $\mc{A}$ to $\mf{L}$ is the $\mf{L}$\nobreakdash-structure obtained from $\mc{A}$ by forgetting about the symbols of $\mf{L}^+ \setminus \mf{L}$.  If $\mf{L} \subseteq \mf{L}^+$ are computable languages and $\mc{A}$ is a computable $\mf{L}^+$\nobreakdash-structure, then $\mc{A} \rst \mf{L}$ is a computable $\mf{L}$\nobreakdash-structure.

Many of our arguments make implicit use of the following proposition.
\begin{Proposition}\label{Prop-reduct}
Let $\mf{L} \subseteq \mf{L}^+$ be computable languages, let $(\mc{A}_n : n \in \Nb)$ be a uniformly computable sequence of $\mf{L}^+$\nobreakdash-structures, and let $C$ be a cohesive set.  Then
\begin{align*}
\prod\nolimits_C (\mc{A}_n \rst \mf{L}) \;\iso\; \Bigl( \prod\nolimits_C \mc{A}_n \Bigr) \rst \mf{L}.
\end{align*}
Thus in the case of a single computable $\mf{L}^+$\nobreakdash-structure $\mc{A}$,
\begin{align*}
\prod\nolimits_C (\mc{A} \rst \mf{L}) \;\iso\; \Bigl( \prod\nolimits_C \mc{A} \Bigr) \rst \mf{L}.
\end{align*}
\end{Proposition}

\begin{proof}
The $\mf{L}$\nobreakdash-structures $\prod_C (\mc{A}_n \rst \mf{L})$ and $\bigl( \prod_C \mc{A}_n \bigr) \rst \mf{L}$ share the same domain, which is the set of all $=_C$\nobreakdash-equivalence classes $[\varphi]$ of partial computable functions $\varphi$ such that $\forall n \, (\varphi(n)\da \,\imp\, \varphi(n) \in |\mc{A}_n|)$ and $C \subseteq^* \dom(\varphi)$.  One then checks that the identity map is the desired isomorphism.
\end{proof}

Cohesive powers respect computable substructures and finite disjoint unions of computable structures.

\begin{Proposition}\label{prop-substr}
Let $\mf{L}$ be a computable language with a unary relation symbol $U$.  Let $(\mc{A}_n : n \in \Nb)$ be a uniformly computable sequence of $\mf{L}$\nobreakdash-structures, and suppose that $\{a \in |\mc{A}_n| : \mc{A}_n \models U(a)\}$ forms the domain of a computable substructure $\mc{B}_n$ of $\mc{A}_n$ for every $n$.  Let $C$ be a cohesive set.  Then $\bigl\{[\varphi] \in |\prod_C \mc{A}_n| : \prod_C \mc{A}_n \models U([\varphi])\bigr\}$ forms the domain of a substructure $\mc{D}$ of $\prod_C \mc{A}_n$, and $\prod_C \mc{B}_n \iso \mc{D}$.
\end{Proposition}

\begin{proof}
Let $f$ be an $m$\nobreakdash-ary function symbol of $\mf{L}$, and let $[\varphi_0], \dots, [\varphi_{m-1}]$ be elements of $\{[\varphi] \in |\prod_C \mc{A}_n| : \prod_C \mc{A}_n \models U([\varphi])\}$.  Then $(\forae n \in C)\left(\mc{A}_n \models \bigwedge_{i < m} U(\varphi_i(n))\right)$, so $(\forae n \in C)\bigl[\mc{A}_n \models U(f(\varphi_0(n), \dots, \varphi_{m-1}(n)))\bigr]$ because $\{a \in |\mc{A}_n| : \mc{A}_n \models U(a)\}$ is closed under $f^{\mc{A}_n}$ as it is the domain of the substructure $\mc{B}_n$.  Therefore $\prod_C \mc{A}_n \models U(f([\varphi_0], \dots, [\varphi_{m-1}]))$ by Theorem~\ref{thm-LosProdParam} item~\ref{it-LosProdPramDelta2}.  Thus $\bigl\{[\varphi] \in |\prod_C \mc{A}_n| : \prod_C \mc{A}_n \models U([\varphi])\bigr\}$ is closed under $f^{\prod_C \mc{A}_n}$.  Similar reasoning shows that $\prod_C \mc{A}_n \models U(c)$ for every constant symbol $c$ of $\mf{L}$.  Thus $\bigl\{[\varphi] \in |\prod_C \mc{A}_n| : \prod_C \mc{A}_n \models U([\varphi])\bigr\}$ forms the domain of a substructure of $\prod_C \mc{A}_n$.

Recall that $\mc{B}_n$ is the substructure of $\mc{A}_n$ with domain $\{a \in |\mc{A}_n| : \mc{A}_n \models U(a)\}$ for each $n$, and let $\mc{D}$ be the substructure of $\prod_C \mc{A}_n$ with domain $\bigl\{[\varphi] \in |\prod_C \mc{A}_n| : \prod_C \mc{A}_n \models U([\varphi])\bigr\}$.  In view of the comment following Definition~\ref{def-CohProd}, the domains of $\prod_C \mc{B}_n$ and of $\mc{D}$ are in both cases the $=_C$\nobreakdash-equivalence classes of partial computable functions $\varphi$ such that $(\forae n \in C)\bigl(\varphi(n)\da \,\andd\, U^{\mc{A}_n}(\varphi(n))\bigr)$.  One may then check that the map $[\varphi] \mapsto [\varphi]$ from $\prod_C \mc{B}_n$ to $\mc{D}$ is an isomorphism.
\end{proof}

We usually apply Proposition~\ref{prop-substr} in the case of a single computable structure $\mc{A}$.  In this situation, the proposition says that if $\mc{A}$ is a computable structure, if $\{a \in |\mc{A}| : \mc{A} \models U(a)\}$ forms the domain of a computable substructure $\mc{B}$ of $\mc{A}$, and if $C$ is cohesive, then $\bigl\{[\varphi] \in |\prod_C \mc{A}| : \prod_C \mc{A} \models U([\varphi])\bigr\}$ forms the domain of a substructure $\mc{D}$ of $\prod_C \mc{A}$ with $\prod_C \mc{B} \iso \mc{D}$.

\begin{Definition}
Let $\mf{L}$ be a relational language, and let $\mc{A}_0, \dots, \mc{A}_{k-1}$ be $\mf{L}$\nobreakdash-structures for some $k > 0$.  The \emph{disjoint union} of $\mc{A}_0, \dots, \mc{A}_{k-1}$ is the $\mf{L}$\nobreakdash-structure $\bigsqcup_{i < k} \mc{A}_i$ with domain $\bigcup_{i < k}(\{i\} \times |\mc{A}_i|)$ defined as follows.  For every $m$\nobreakdash-ary relation symbol $R$ and every $(i_0, x_0), \dots, (i_{m-1}, x_{m-1}) \in \bigl|\bigsqcup_{i < k} \mc{A}_i\bigr|$, the relation $R^{\bigsqcup_{i < k} \mc{A}_i}((i_0, x_0), \dots, (i_{m-1}, x_{m-1}))$ holds if and only if $i_0 = \cdots = i_{m-1} = i$ for some $i < k$ and $R^{\mc{A}_i}(x_0, \dots, x_{m-1})$ holds.
\end{Definition}
In the case of computable $\mf{L}$\nobreakdash-structures $\mc{A}_0, \dots, \mc{A}_{k-1}$ for a computable relational language $\mf{L}$, one may use the pairing function to compute a copy of $\bigsqcup_{i < k} \mc{A}_i$.  Thus if $\mc{A}_0, \dots, \mc{A}_{k-1}$ are computable structures, then so is $\bigsqcup_{i < k} \mc{A}_i$.

\begin{Proposition}\label{prop-DU}
Let $\mf{L}$ be a computable relational language, let $\mc{A}_0, \dots, \mc{A}_{k-1}$ be computable $\mf{L}$\nobreakdash-structures for some $k > 0$, and let $C$ be cohesive.  Then
\begin{align*}
\prod\nolimits_C\left(\bigsqcup_{i < k} \mc{A}_i\right) \;\iso\; \bigsqcup_{i < k}\left(\prod\nolimits_C \mc{A}_i\right).
\end{align*}
\end{Proposition}

\begin{proof}
Expand $\mf{L}$ to $\mf{L}^+ = \mf{L} \cup \{U_0, \dots, U_{k-1}\}$, where $U_0, \dots, U_{k-1}$ are $k$ fresh unary relation symbols.  Expand $\bigsqcup_{i < k} \mc{A}_i$ to a computable $\mf{L}^+$\nobreakdash-structure by interpreting $U_i$ as the domain of $\mc{A}_i$ for each $i < k$:  $U_i^{\bigsqcup_{i < k} \mc{A}_i}(x)$ holds if and only if $\pi_0(x) = i$.  Then for each $x \in \bigl|\bigsqcup_{i < k} \mc{A}_i\bigr|$ there is a unique $i < k$ for which $U_i^{\bigsqcup_{i < k} \mc{A}_i}(x)$ holds:
\begin{align*}
\bigsqcup_{i < k} \mc{A}_i \models \forall x\, \left[ \left( \bigvee_{i < k}U_i(x) \right) \;\andd\; \left( \bigwedge_{\substack{i,j < k \\ i \neq j}}(U_i(x) \imp \neg U_j(x)) \right) \right].\tag{$*$}\label{eq-OneU}
\end{align*}
Furthermore, for each $m$\nobreakdash-ary relation symbol $R \in \mf{L}$, if $R^{\bigsqcup_{i < k} \mc{A}_i}(x_0, \dots, x_{m-1})$ holds for some $x_0, \dots, x_{m-1} \in \bigl|\bigsqcup_{i < k} \mc{A}_i\bigr|$, then there is an $i < k$ such that $U_i^{\bigsqcup_{i < k} \mc{A}_i}(x_j)$ holds for all $j < m$:
\begin{align*}
\bigsqcup_{i < k} \mc{A}_i \models \forall x_0, \dots, x_{m-1}\, \left[ R(x_0, \dots, x_{m-1}) \;\imp\; \bigvee_{i < k}\bigwedge_{j < m} U_i(x_j) \right].\tag{$\star$}\label{eq-SameU}
\end{align*}
The $\mf{L}^+$\nobreakdash-sentences in~\eqref{eq-OneU} and~\eqref{eq-SameU} are $\Pi_1$, therefore $\prod_C (\bigsqcup_{i < k} \mc{A}_i)$ also satisfies all of these sentences by Theorem~\ref{thm-LosGeneral} item~\ref{it-LosDelta3Sent}.  For each $i < k$, let $\mc{D}_i$ be the $\mf{L}$\nobreakdash-substructure of $\prod_C (\bigsqcup_{i < k} \mc{A}_i)$ whose domain consists of the $[\varphi]$ for which $\prod_C (\bigsqcup_{i < k} \mc{A}_i) \models U_i([\varphi])$.  Then $\prod_C (\bigsqcup_{i < k} \mc{A}_i) \iso \bigsqcup_{i < k} \mc{D}_i$ as $\mf{L}$\nobreakdash-structures.  This is because each $[\varphi] \in \bigl|\prod_C (\bigsqcup_{i < k} \mc{A}_i)\bigr|$ is in $|\mc{D}_i|$ for exactly one $i < k$; and if $R^{\prod_C (\bigsqcup_{i < k} \mc{A}_i)}([\varphi_0], \dots, [\varphi_{m-1}])$ holds for an $m$\nobreakdash-ary relation symbol $R \in \mf{L}$ and $[\varphi_0], \dots, [\varphi_{m-1}] \in \bigl|\prod_C (\bigsqcup_{i < k} \mc{A}_i)\bigr|$, then $[\varphi_0], \dots, [\varphi_{m-1}]$ must all be in $|\mc{D}_i|$ for the same $i < k$.  We have that $\mc{D}_i \iso \prod_C \mc{A}_i$ as $\mf{L}$\nobreakdash-structures for each $i < k$ by Proposition~\ref{prop-substr}, so $\prod_C\Bigl(\bigsqcup_{i < k} \mc{A}_i\Bigr) \;\iso\; \bigsqcup_{i < k}(\prod_C \mc{A}_i)$ as $\mf{L}$\nobreakdash-structures.
\end{proof}

\subsection*{Saturation}

There are many classical results concerning the saturation of ultraproducts.  See, for example,~\cite{ChangKeislerBook}*{Section~6.1}.  One well-known result is that, for a countable language, ultraproducts over countably incomplete ultrafilters are always $\aleph_1$\nobreakdash-saturated (see~\cite{ChangKeislerBook}*{Theorem~6.1.1}).  Here we show that cohesive products of uniformly decidable structures are always recursively saturated (which is essentially due to Nelson~\cite{Nelson}) and that, for $n > 0$, cohesive products of uniformly $n$\nobreakdash-decidable structures are always $\Sigma_n$\nobreakdash-recursively saturated.  Furthermore, we show that if the cohesive set is assumed to be co-c.e., then we obtain the $n=0$ case and can also squeeze one more level of saturation out of the cohesive product:  cohesive products of uniformly $n$\nobreakdash-decidable structures over co-c.e.\ cohesive sets are always $\Sigma_{n+1}$\nobreakdash-recursively saturated.

We follow the terminology of~\cite{KayeBook}*{Section~11.2} regarding types and saturation.  Beware that what we call a \emph{type over a structure $\mc{A}$}, other authors may call a \emph{type over a finite set $\{c_0, \dots, c_{\ell-1}\}$ of parameters from $\mc{A}$}.  Let $\mf{L}$ be a language, and let $\mc{A}$ be an $\mf{L}$\nobreakdash-structure.  Consider a set $p(\vec{x}) = p(x_0, \dots, x_{m-1})$ of formulas of the form $\Phi(\vec{x}; \vec{c}) = \Phi(x_0, \dots, x_{m-1}; c_0, \dots, c_{\ell-1})$ in the language $\mf{L} \cup \{c_0, \dots, c_{\ell-1}\}$, where $x_0, \dots, x_{m-1}$ are $m$ fixed variables and $c_0, \dots, c_{\ell-1}$ are $\ell$ fixed parameters from $|\mc{A}|$ that are identified with fresh constant symbols.  Such a set $p(\vec{x})$ is called a \emph{type over $\mc{A}$} if it is finitely satisfied in $\mc{A}$:  for every $\Phi_0(\vec{x}; \vec{c}), \dots, \Phi_{k-1}(\vec{x}; \vec{c}) \in p(\vec{x})$, $\mc{A} \models \exists \vec{x} \, \bigwedge_{i < k} \Phi_i(\vec{x}; \vec{c})$.  A type $p(\vec{x})$ over $\mc{A}$ is \emph{realized} if there are $a_0, \dots, a_{m-1} \in |\mc{A}|$ such that for all $\Phi(\vec{x}; \vec{c}) \in p(\vec{x})$, $\mc{A} \models \Phi(\vec{a}; \vec{c})$.  A type $p(\vec{x})$ over $\mc{A}$ is a $\Sigma_n$\nobreakdash-type if every formula in $p(\vec{x})$ is $\Sigma_n$.  An $\mf{L}$\nobreakdash-structure $\mc{A}$ is \emph{recursively saturated} if it realizes every computable type over $\mc{A}$.  Similarly, $\mc{A}$ is \emph{$\Sigma_n$\nobreakdash-recursively saturated} if it realizes every computable $\Sigma_n$\nobreakdash-type over $\mc{A}$. 

When discussing a formula $\Phi(\vec{x}; \vec{c})$ of a type $p(\vec{x})$, we write $\Phi(\vec{x}; \vec{y})$ for the corresponding $\mf{L}$\nobreakdash-formula, with fresh variables $\vec{y}$ in place of the constants $\vec{c}$.  We sometimes write $p(\vec{x}; \vec{c})$ or $p(\vec{x}; \vec{y})$ in place of $p(\vec{x})$ when we want to emphasize the type's parameters or the corresponding variables.

\begin{Lemma}\label{lem-SatGen}
Let $\mf{L}$ be a computable language, let $(\mc{A}_n : n \in \Nb)$ be a uniformly computable sequence of $\mf{L}$\nobreakdash-structures, and let $C$ be cohesive.  Let $p(\vec{x}; \vec{c})$ be a computable type over $\prod_C \mc{A}_n$ with computable enumeration $(\Phi_i : i \in \Nb)$.  Assume that the formulas
\begin{align*}
\left(\exists \vec{x}\; \bigwedge_{i < k} \Phi_i(\vec{x}; \vec{y}) : k \in \Nb\right)
\end{align*}
are uniformly decidable in the structures $(\mc{A}_n : n \in \Nb)$.  Then $\prod_C \mc{A}_n$ realizes $p(\vec{x}; \vec{c})$.
\end{Lemma}

\begin{proof}
Let $\vec{c} = [\psi_0], \dots, [\psi_{\ell-1}]$ be the parameters of the type $p(\vec{x}; \vec{c})$, so that
\begin{align*}
\prod\nolimits_C \mc{A}_n \models \exists \vec{x}\; \bigwedge_{i < k} \Phi_i(\vec{x}; [\psi_0], \dots, [\psi_{\ell-1}]) 
\end{align*}
for each $k$.  To ease notation, pack $\psi_0, \dots, \psi_{\ell-1}$ into a single partial computable function $\psi \colon \Nb \imp \Nb^\ell$ given by $\psi(n) \keq \la \psi_0(n), \dots, \psi_{\ell-1}(n) \ra$.  Notice that $C \subseteq^* \dom(\psi)$.  Write $\ora{[\psi]}$ as an abbreviation for $[\psi_0], \dots, [\psi_{\ell-1}]$.

Define a partial computable function $\varphi \colon \Nb \imp \Nb^m$ as follows.  Given $n$, first search for the greatest $k \leq n$ such that
\begin{align*}
\mc{A}_n \models \exists \vec{x}\; \bigwedge_{i < k} \Phi_i(\vec{x}; \psi(n)).
\end{align*}
This search is effective on account of the uniform decidability assumption.  If such a $k$ is found, search for the first tuple $\vec{a} = \la a_0, \dots, a_{m-1} \ra$ such that $\mc{A}_n \models \bigwedge_{i < k} \Phi_i(\vec{a}; \psi(n))$, and set $\varphi(n) = \vec{a}$.  If there is no such $k$, then $\varphi(n)\ua$.

Consider a fixed $k$.  By assumption, $\prod_C \mc{A}_n \models \exists \vec{x}\; \bigwedge_{i < k} \Phi_i\left(\vec{x}; \ora{[\psi]}\right)$.  Thus $C \subseteq^* \bigl\{n : \mc{A}_n \models \exists \vec{x}\, \bigwedge_{i < k} \Phi_i(\vec{x}; \psi(n))\bigr\}$ by Lemma~\ref{lem-LosDec}.  This means that for almost every $n \in C$ with $n \geq k$, the initial search in the computation of $\varphi(n)$ succeeds and finds a $\wh{k}$ with $k \leq \wh{k} \leq n$.  Thus $C \subseteq^* \dom(\varphi)$ and 
\begin{align*}
C \subseteq^* \left\{n : \mc{A}_n \models \bigwedge_{i < k} \Phi_i(\varphi(n); \psi(n))\right\}.
\end{align*}
Let $\varphi_i = \pi_i \circ \varphi$ for each $i < m$.  Then $[\varphi_0], \dots, [\varphi_{m-1}] \in |\prod_C \mc{A}_n|$.  Let $\ora{[\varphi]}$ abbreviate the sequence $[\varphi_0], \dots, [\varphi_{m-1}]$.  Then
$\prod_C \mc{A}_n \models \bigwedge_{i < k} \Phi_i \left(\ora{[\varphi]}; \ora{[\psi]}\right)$ by Lemma~\ref{lem-LosDec}.  This implies that $\prod_C \mc{A}_n \models \Phi_i \left(\ora{[\varphi]}; \ora{[\psi]}\right)$ for every $i$.  Thus $[\varphi_0], \dots, [\varphi_{m-1}]$ realize $p(\vec{x}; \ora{[\psi]})$ in $\prod_C \mc{A}_n$.
\end{proof}

Items~\ref{it-DecRecSatSeq} and~\ref{it-DecRecSatStr} of the next theorem are essentially~\cite{Nelson}*{Theorem~2.2}.

\begin{Theorem}\label{thm-SatGen}
Let $\mf{L}$ be a computable language, and let $C$ be a cohesive set.
\begin{enumerate}[(1)]
\item\label{it-DecRecSatSeq} Let $(\mc{A}_i : i \in \Nb)$ be a sequence of uniformly decidable $\mf{L}$\nobreakdash-structures.  Then $\prod_C \mc{A}_i$ is recursively saturated.

\medskip

\item\label{it-nDecRecSatSeq} Let $(\mc{A}_i : i \in \Nb)$ be a sequence of uniformly $n$\nobreakdash-decidable $\mf{L}$\nobreakdash-structures for an $n > 0$.  Then $\prod_C \mc{A}_i$ is $\Sigma_n$\nobreakdash-recursively saturated.

\medskip

\item\label{it-DecRecSatStr} Let $\mc{A}$ be a decidable $\mf{L}$\nobreakdash-structure.  Then $\prod_C \mc{A}$ is recursively saturated.

\medskip

\item\label{it-nDecRecSatStr} Let $\mc{A}$ be an $n$\nobreakdash-decidable $\mf{L}$\nobreakdash-structure for an $n > 0$.  Then $\prod_C \mc{A}$ is $\Sigma_n$\nobreakdash-recursively saturated.
\end{enumerate}
\end{Theorem}

\begin{proof}
Item~\ref{it-DecRecSatSeq} follows directly from Lemma~\ref{lem-SatGen} because every computably enumerable sequence of formulas is uniformly decidable in $(\mc{A}_i : i \in \Nb)$.  Item~\ref{it-nDecRecSatSeq} also follows from Lemma~\ref{lem-SatGen}.  If $(\Phi_i : i \in \Nb)$ is a computable enumeration of $\Sigma_n$ formulas and $n > 0$, then
\begin{align*}
\left(\exists \vec{x}\; \bigwedge_{i < k} \Phi_i(\vec{x}; \vec{y}) : k \in \Nb\right)
\end{align*}
is a computable enumeration of (formulas that are logically equivalent to) $\Sigma_n$ formulas and hence is uniformly decidable in $(\mc{A}_i : i \in \Nb)$.  Items~\ref{it-DecRecSatStr} and~\ref{it-nDecRecSatStr} are the special cases of items~\ref{it-DecRecSatSeq} and~\ref{it-nDecRecSatSeq} in which $\mc{A}_i$ is $\mc{A}$ for each $i$.
\end{proof}

If we restrict to co-c.e.\ cohesive sets, then we can include $n=0$ and improve $\Sigma_n$\nobreakdash-recursive saturation to $\Sigma_{n+1}$\nobreakdash-recursive saturation in Theorem~\ref{thm-SatGen} items~\ref{it-nDecRecSatSeq} and~\ref{it-nDecRecSatStr}.

\begin{Lemma}\label{lem-SatCoCe}
Let $\mf{L}$ be a computable language, let $(\mc{A}_n : n \in \Nb)$ be a uniformly computable sequence of $\mf{L}$\nobreakdash-structures, and let $C$ be co-c.e.\ and cohesive.  Let $p(\vec{x}; \vec{c})$ be a computable type over $\prod_C \mc{A}_n$ consisting of formulas of the form $\exists \vec{z}\, \Phi(\vec{x}, \vec{z}; \vec{c})$, with computable enumeration $(\exists \vec{z}_i\, \Phi_i(\vec{x}, \vec{z}_i; \vec{c}) : i \in \Nb)$.  Further assume that the formulas $(\Phi_i(\vec{x}, \vec{z}_i; \vec{y}) : i \in \Nb)$ are uniformly decidable in the structures $(\mc{A}_n : n \in \Nb)$.  Then $\prod_C \mc{A}_n$ realizes $p(\vec{x}; \vec{c})$.
\end{Lemma}

\begin{proof}
Let $\vec{c} = [\psi_0], \dots, [\psi_{\ell-1}]$ be the parameters of the type $p(\vec{x}; \vec{c})$, so that
\begin{align*}
\prod\nolimits_C \mc{A}_n \models \exists \vec{x}\; \bigwedge_{i < k} \exists \vec{z}_i\, \Phi_i(\vec{x}, \vec{z}_i; [\psi_0], \dots, [\psi_{\ell-1}]) 
\end{align*}
for each $k$.  As in the proof of Lemma~\ref{lem-SatGen}, let $\psi \colon \Nb \imp \Nb^\ell$ be the partial computable function given by $\psi(n) \keq \la \psi_0(n), \dots, \psi_{\ell-1}(n) \ra$.  Notice that $C \subseteq^* \dom(\psi)$, and write $\ora{[\psi]}$ as an abbreviation for $[\psi_0], \dots, [\psi_{\ell-1}]$.

The goal is to partially compute a function $\theta \colon \Nb \imp \Nb^m$ so that $C \subseteq^* \dom(\theta)$ and
\begin{align*}
(\existsinf n \in C) \Bigl( \mc{A}_n \models \exists \vec{z}_i\, \Phi_i(\theta(n), \vec{z}_i; \psi(n)) \Bigr) \tag{$*$}\label{eq-InfSat}
\end{align*}
for each $i$.  The set
\begin{align*}
\Bigl\{n : \mc{A}_n \models \exists \vec{z}_i\, \Phi_i(\theta(n), \vec{z}_i; \psi(n))\Bigr\}
\end{align*}
is c.e.\ for each $i$ because $\Phi_i$ is uniformly decidable in $(\mc{A}_n : n \in \Nb)$.  Thus \eqref{eq-InfSat} implies that 
\begin{align*}
(\forae n \in C) \Bigl( \mc{A}_n \models \exists \vec{z}_i\, \Phi_i(\theta(n), \vec{z}_i; \psi(n)) \Bigr)
\end{align*}
for each $i$ by cohesiveness.   Letting $\varphi_j = \pi_j \circ \theta$ for each $j < m$, we therefore have that $[\varphi_0], \dots, [\varphi_{m-1}] \in |\prod_C \mc{A}_n|$ and that
\begin{align*}
\prod\nolimits_C \mc{A}_n \models \exists \vec{z}_i\, \Phi_i\left(\ora{[\varphi]}, \vec{z}_i; \ora{[\psi]} \right)
\end{align*}
for each $i$ by Lemma~\ref{lem-LosProdParamDelta2Helper}.  Thus $[\varphi_0], \dots, [\varphi_{m-1}]$ realize $p(\vec{x}; \ora{[\psi]})$ in $\prod_C \mc{A}_n$.

The strategy for partially computing $\theta$ is to keep track of the numbers $k$ that are \emph{covered}, meaning that it looks like there is an $n \in C$ with $n > k$ such that $\mc{A}_n \models \bigwedge_{i < k} \exists \vec{z}_i\, \Phi_i(\theta(n), \vec{z}_i; \psi(n))$.  As the computation progresses, a $k$ that is covered may become uncovered because the $n$ that covers it is enumerated in the complement of $C$.  When this happens, we note the least $k$ that becomes uncovered, we search for the first $n > k$ where $\theta(n)$ is not yet defined, it looks like $n \in C$, and there looks to be an $\vec{a} \in \Nb^m$ such that $\mc{A}_n \models \bigwedge_{i < k} \exists \vec{z}_i\, \Phi_i(\vec{a}, \vec{z}_i; \psi(n))$, and we attempt to cover $k$ again by setting $\theta(n) = \vec{a}$.  This strategy eventually succeeds because if $n_0 \in C$ is sufficiently large and we never choose a smaller member of $C$ to cover $k$, then we eventually choose $n_0$ to cover either $k$ or an even bigger number.

Formally, let $W$ denote the c.e.\ set $\ol{C}$, and let $(W_s)_{s \in \Nb}$ be a computable $\subseteq$\nobreakdash-increasing enumeration of $W$.  Let $(U_k : k \in \Nb)$ be the uniformly c.e.\ sequence of sets given by
\begin{align*}
U_k = \left\{\la \vec{a}, n \ra \in \Nb^m \times \Nb: \mc{A}_n \models \bigwedge_{i < k} \exists \vec{z}_i\, \Phi_i(\vec{a}, \vec{z}_i; \psi(n))\right\}
\end{align*}
with uniformly computable $\subseteq$\nobreakdash-increasing enumerations $(U_{k,s})_{s \in \Nb}$ for each $k$.  The sequence $(U_k : k \in \Nb)$ is uniformly c.e.\ because the formulas $(\Phi_i : i \in \Nb)$ are uniformly decidable in $(\mc{A}_n : n \in \Nb)$.  Observe that if $k_0 \leq k_1$, then $U_{k_1} \subseteq U_{k_0}$.

To partially compute $\theta$, we compute an increasing sequence $\theta_0 \subseteq \theta_1 \subseteq \theta_2 \subseteq \cdots$ of finite approximations to $\theta$.  Start at stage $0$ with $\theta_0 = \emptyset$.  At stage $s$, we have $\theta_s$ and we define $\theta_{s+1}$.

Say that \emph{$n$ covers $k$ at stage $s$} if
\begin{itemize}
\item $n > k$,

\smallskip

\item $n \notin W_s$,

\smallskip

\item $\theta_s(n)\da$, and

\smallskip

\item $\la \theta_s(n), n \ra \in U_{k,s}$.
\end{itemize}
If there is an $n$ that covers $k$ at stage $s$, then also say that $k$ is \emph{covered} at stage $s$.  Let $k^0_s$ be the least number that is not covered at stage $s$.  If $s > 0$, let $X_s = W_s \setminus W_{s-1}$.  Let $k^1_s$ be the least number, if it exists, for which some $n \in X_s$ covered $k^1_s$ at stage $s-1$ but no $m < n$ covers $k^1_s$ at stage $s$.  If $k^1_s$ is defined, let $k_s = \min\{k^0_s, k^1_s\}$.  Otherwise, let $k_s = k^0_s$.  Now search for the least $n > k_s$ such that $n \notin W_s$, such that $\theta_s(n)\ua$, and such that $\la \vec{a}, n \ra \in U_{k_s, s}$ for some $\vec{a}$.  If there is such an $n$, let $\vec{a}$ be the first corresponding $\vec{a}$, and extend $\theta_s$ to $\theta_{s+1}$ by setting $\theta_{s+1}(n) = \vec{a}$.  If there is no such $n$, then set $\theta_{s+1} = \theta_s$.  Go to stage $s+1$.  This completes the partial computation of $\theta$.

If $n$ covers $k$ at some stage $s$, there could be a later stage $t > s$ at which $n$ does not cover $k$ because $n \in W_t$.  However, if $n \in C$, then $n \notin W_t$ for every $t$, so $k$ stays covered by $n$ forever.

\begin{Claim*}
Every $k$ is eventually covered by an $n \in C$.
\end{Claim*}
\begin{proof}[Proof of Claim]
Proceed by induction on $k$.  Let $s_0$ be a stage by which all $\wh{k} < k$ have been covered by members of $C$.  Let $c$ be the greatest member of $C$ covering a $\wh{k} < k$ at stage $s_0$, and let $s_1 > s_0$ be a stage such that $W_{s_1} \rst c = W \rst c$.  Then $k_s \geq k$ at all stages $s > s_1$.  By assumption,
\begin{align*}
\prod\nolimits_C \mc{A}_n \models \exists \vec{x}\, \bigwedge_{i < k} \exists \vec{z}_i\, \Phi_i\left(\vec{x}, \vec{z}_i; \ora{[\psi]} \right),
\end{align*}
and therefore
\begin{align*}
C \subseteq^* \left\{n : \mc{A}_n \models \exists \vec{x}\, \bigwedge_{i < k} \exists \vec{z}_i\, \Phi_i\left(\vec{x}, \vec{z}_i; \psi(n) \right)\right\}
\end{align*}
by Lemma~\ref{lem-LosProdParamDelta2Helper}.  To see that Lemma~\ref{lem-LosProdParamDelta2Helper} applies here, pull the $\exists \vec{z}_i$ quantifiers out in front of the conjunction.  The resulting formula has the form $\exists \vec{w}\, \Psi(\vec{w}; \vec{y})$, where $\Psi$ is uniformly decidable in $(\mc{A}_n : n \in \Nb)$.  Let $n_0$ be least such that $n_0 > k$, such that $n_0 \in C$, such that $\theta_{s_1}(n_0)\ua$, such that $\psi(n_0)\da$, and such that
\begin{align*}
\mc{A}_{n_0} \models \exists \vec{x}\, \bigwedge_{i < k} \exists \vec{z_i}\, \Phi_i(\vec{x}, \vec{z}_i; \psi(n_0)).
\end{align*}
If $\theta_s(n_0)\da$ for the first time at some stage $s > s_1$, it is to cover some $j \geq k$.  As $n_0 \in C$, we have that $n_0$ covers $j$ and therefore covers $k$ at all later stages.

Let $s_2 > s_1$ be large enough so that $W_{s_2} \rst n_0 = W \rst n_0$ and so that there is an $\vec{a}$ with $\la \vec{a}, n_0 \ra \in U_{k, s_2}$.  Consider stage $s_2$.  If $k$ is not covered at stage $s_2$, then it must be that $\theta_{s_2}(n_0)\ua$.  In this case, $k_{s_2} = k$, and $n_0$ is least such that $n_0 > k_{s_2}$, $n_0 \notin W_{s_2}$, $\theta_{s_2}(n_0)\ua$, and $\la \vec{a}, n_0 \ra \in U_{k_{s_2}, s_2}$ for some $\vec{a}$.  So $\theta_{s_2 + 1}(n_0)$ is defined to cover $k$ at stage $s_2$.

Suppose instead that $k$ is covered at stage $s_2$.  In this case, let $n_1$ be least such that there is a stage $s_3 \geq s_2$ at which $n_1$ covers $k$.  If $n_1 \in C$, then this is as desired.  Otherwise, $n_1 \in W$, in which case there is a least $s > s_3$ with $n_1 \in W_s$.  The number $n_1$ covers $k$ at stage $s-1$, but by choice of $n_1$, no $n < n_1$ covers $k$ at stage $s$.  Thus $k^1_s = k$, so $k_s = k$.  If $\theta_s(n_0)\da$, then $n_0$ must already cover $k$, as noted above.  If $\theta_s(n_0)\ua$, then $n_0$ is least such that $n_0 > k_s$, $n_0 \notin W_s$, $\theta_s(n_0)\ua$, and $\la \vec{a}, n_0 \ra \in U_{k_s, s}$ for some $\vec{a}$.  So $\theta_{s+1}(n_0)$ is defined to cover $k$ at stage $s$.  This completes the proof of the claim.
\end{proof}

To finish the proof, consider the formula $\exists \vec{z}_i\, \Phi_i$.  By the claim, every $k$ is eventually covered by an $n \in C$.  Thus for every $k > i$, there is an $n > k$ with $n \in C$, $\theta(n)\da$, and $\la \theta(n), n \ra \in U_k$.  Thus $C \subseteq^* \dom(\theta)$ by cohesiveness, and
\begin{align*}
(\existsinf n \in C) \Bigl( \mc{A}_n \models \exists \vec{z}_i\, \Phi_i(\theta(n), \vec{z}_i; \psi(n)) \Bigr)
\end{align*}
as desired.
\end{proof}

\begin{Theorem}\label{thm-SatCoCe}
Let $\mf{L}$ be a computable language, and let $C$ be a co-c.e.\ cohesive set.
\begin{enumerate}[(1)]
\item\label{it-nDecRecSatSeqCoCe} Let $(\mc{A}_i : i \in \Nb)$ be a sequence of uniformly $n$\nobreakdash-decidable $\mf{L}$\nobreakdash-structures.  Then $\prod_C \mc{A}_i$ is $\Sigma_{n+1}$\nobreakdash-recursively saturated.

\medskip

\item\label{it-nDecRecSatStrCoCe} Let $\mc{A}$ be an $n$\nobreakdash-decidable $\mf{L}$\nobreakdash-structure.  Then $\prod_C \mc{A}$ is $\Sigma_{n+1}$\nobreakdash-recursively saturated.
\end{enumerate}
\end{Theorem}

\begin{proof}
Item~\ref{it-nDecRecSatSeqCoCe} follows from Lemma~\ref{lem-SatCoCe}.  A computable $\Sigma_{n+1}$\nobreakdash-type can be computably enumerated as $(\exists \vec{z}_j\, \Phi_j : j \in \Nb)$, where $\Phi_j$ is $\Pi_n$ for every $j$.  The formulas $(\Phi_j : j \in \Nb)$ are then uniformly decidable in $(\mc{A}_i : i \in \Nb)$ because $(\mc{A}_i : i \in \Nb)$ is a uniformly $n$\nobreakdash-decidable sequence of structures.  Item~\ref{it-nDecRecSatStrCoCe} is the special case of item~\ref{it-nDecRecSatSeqCoCe} in which $\mc{A}_i$ is $\mc{A}$ for each $i$.
\end{proof}

The $n=0$ case of Theorem~\ref{thm-SatCoCe} is particularly noteworthy.

\begin{Corollary}\label{cor-SatCoCe}
Let $\mf{L}$ be a computable language, and let $C$ be a co-c.e.\ cohesive set.
\begin{enumerate}[(1)]
\item\label{it-Sigma1RecSatSeqCoCe} Let $(\mc{A}_i : i \in \Nb)$ be a uniformly computable sequence of $\mf{L}$\nobreakdash-structures.  Then $\prod_C \mc{A}_i$ is $\Sigma_1$\nobreakdash-recursively saturated.

\medskip

\item\label{it-Sigma1RecSatStrCoCe} Let $\mc{A}$ be a computable $\mf{L}$\nobreakdash-structure.  Then $\prod_C \mc{A}$ is $\Sigma_1$\nobreakdash-recursively saturated.
\end{enumerate}
\end{Corollary}

\subsection*{Isomorphisms}

Classically, an isomorphism between two structures induces an isomorphism between the corresponding ultrapowers over a fixed ultrafilter.  In the effective case, a computable isomorphism between two computable structures induces an isomorphism between the corresponding cohesive powers over a fixed cohesive set.  This fact essentially appears in~\cite{DimitrovCohPow}, but we include a proof here for completeness.

\begin{Theorem}\label{thm-IsoCohPow}
Let $\mf{L}$ be a computable language, let $\mc{A}_0$ and $\mc{A}_1$ be computable $\mf{L}$\nobreakdash-structures that are computably isomorphic, and let $C$ be cohesive.  Then $\prod_C \mc{A}_0 \iso \prod_C \mc{A}_1$.
\end{Theorem}

\begin{proof}
We first prove the theorem under the assumption that $\mf{L}$ is a relational language.  Let $\mc{A}_0$ and $\mc{A}_1$ be computable $\mf{L}$\nobreakdash-structures, and let $f \colon |\mc{A}_0| \imp |\mc{A}_1|$ be a computable isomorphism.  Expand the language to $\mf{L}^+ = \mf{L} \cup \{U_0, U_1, R_f\}$, where $U_0$ and $U_1$ are fresh unary relation symbols and $R_f$ is a fresh binary relation symbol.  Expand $\mc{A}_0 \sqcup \mc{A}_1$ to a computable $\mf{L}^+$\nobreakdash-structure by interpreting $U_0$ and $U_1$ as the domains of $\mc{A}_0$ and $\mc{A}_1$ and by interpreting $R_f$ as the graph of $f$.
\begin{itemize}
\item For each $i < 2$, $U_i^{\mc{A}_0 \sqcup \mc{A}_1}(x)$ holds if and only if $\pi_0(x) = i$.

\medskip

\item $R_f^{\mc{A}_0 \sqcup \mc{A}_1}(x,y)$ holds if and only if $\pi_0(x) = 0$, $\pi_0(y) = 1$, and $f(\pi_1(x)) = \pi_1(y)$.
\end{itemize}

The function $f$ is an isomorphism, so $R_f^{\mc{A}_0 \sqcup \mc{A}_1}$ is the graph of an isomorphism between $\mc{A}_0$ and $\mc{A}_1$ as $\mc{L}$\nobreakdash-structures in the $\mf{L}^+$\nobreakdash-structure $\mc{A}_0 \sqcup \mc{A}_1$.  That is, $R_f$ has the following properties in $\mc{A}_0 \sqcup \mc{A}_1$.
\begin{itemize}
\item The domain of $R_f$ corresponds to $|\mc{A}_0|$:  $\forall x \, (\exists y \, R_f(x, y) \;\biimp\; U_0(x))$.

\medskip

\item The image of $R_f$ corresponds to $|\mc{A}_1|$:  $\forall y \, (\exists x \, R_f(x, y) \;\biimp\; U_1(y))$.

\medskip

\item $R_f$ is single-valued on its domain:  $\forall x \forall y_0 \forall y_1 \, (R_f(x, y_0) \andd R_f(x, y_1) \;\imp\; y_0 = y_1)$.

\medskip

\item $R_f$ is injective on its domain:  $\forall x_0 \forall x_1 \forall y \, (R_f(x_0, y) \andd R_f(x_1, y) \;\imp\; x_0 = x_1)$.

\medskip

\item $R_f$ respects the relations of $\mf{L}$:  for every $m$\nobreakdash-ary relation symbol $S \in \mf{L}$,
\begin{align*}
\forall x_0 \cdots \forall x_{m-1} \forall y_0 \cdots \forall y_{m-1} \, \left(\bigwedge_{i < m}R_f(x_i, y_i) \;\imp\; (S(x_0, \dots, x_{m-1}) \biimp S(y_0, \dots, y_{m-1}))\right).
\end{align*}
\end{itemize}

The above properties constitute a collection of $\Pi_2$ $\mf{L}^+$\nobreakdash-sentences that hold in $\mc{A}_0 \sqcup \mc{A}_1$, so they also hold in the cohesive power $\prod_C(\mc{A}_0 \sqcup \mc{A}_1)$ as an $\mf{L}^+$\nobreakdash-structure by Theorem~\ref{thm-LosGeneral} item~\ref{it-LosDelta3Sent}.  For each $i < 2$, let $\mc{D}_i$ denote the substructure of $\prod_C(\mc{A}_0 \sqcup \mc{A}_1)$ with domain given by $U_i$:  $|\mc{D}_i| = \bigl\{[\varphi] : \prod_C(\mc{A}_0 \sqcup \mc{A}_1) \models U_i([\varphi])\bigr\}$.  Then $\prod_C \mc{A}_i \iso \mc{D}_i \rst \mf{L}$ as an $\mf{L}$\nobreakdash-structure for each $i < 2$ by Proposition~\ref{prop-substr}.  In $\prod_C(\mc{A}_0 \sqcup \mc{A}_1)$, $R_f^{\prod_C(\mc{A}_0 \sqcup \mc{A}_1)}$ is the graph of an isomorphism between the reducts $\mc{D}_0 \rst \mf{L}$ and $\mc{D}_1 \rst \mf{L}$.  Therefore $\prod_C \mc{A}_0 \iso \prod_C \mc{A}_1$ as $\mf{L}$\nobreakdash-structures.

Now suppose that $\mf{L}$ contains constant and function symbols in addition to relation symbols.  For uniformity of argument, treat constant symbols as $0$\nobreakdash-ary function symbols.  Let $\mf{L}^{\msf{rel}}$ be the relational language obtained from $\mf{L}$ by replacing each $m$\nobreakdash-ary function symbol $f$ by a fresh $(m+1)$\nobreakdash-ary relation symbol $G_f$ whose intended interpretation is the graph of $f$.  We may translate any $\mf{L}$\nobreakdash-structure $\mc{A}$ into an $\mf{L}^{\msf{rel}}$\nobreakdash-structure $\mc{A}^{\msf{rel}}$ with the same domain by defining
\begin{align*}
G_f^{\mc{A}^{\msf{rel}}}(x_0, \dots, x_{m-1}, y) \;\Biimp\; f^{\mc{A}}(x_0, \dots, x_{m-1}) = y
\end{align*}
for every $m$\nobreakdash-ary function symbol $f \in \mf{L}$ and every $x_0, \dots, x_{m-1}, y \in |\mc{A}|$.  Conversely, suppose that $\mc{A}$ is an $\mf{L}^{\msf{rel}}$\nobreakdash-structure such that for every $m$\nobreakdash-ary function symbol $f \in \mf{L}$, $G_f$ is the graph of a function:  $\forall x_0 \cdots \forall x_{m-1} \exists! y \, G_f^{\mc{A}}(x_0, \dots, x_{m-1}, y)$.  Then we may translate $\mc{A}$ into an $\mf{L}$\nobreakdash-structure $\mc{A}^{\msf{fun}}$ with the same domain by defining $f^{\mc{A}^{\msf{fun}}}(x_0, \dots, x_{m-1})$ to be the unique $y$ such that $G_f^{\mc{A}}(x_0, \dots, x_{m-1}, y)$ for every $m$\nobreakdash-ary function symbol $f \in \mf{L}$ and every $x_0, \dots, x_{m-1} \in |\mc{A}|$.  If $\mc{A}$ and $\mc{B}$ are isomorphic $\mf{L}$\nobreakdash-structures, then the isomorphism is also an isomorphism between $\mc{A}^{\msf{rel}}$ and $\mc{B}^{\msf{rel}}$ as $\mf{L}^{\msf{rel}}$\nobreakdash-structures.  Conversely, if $\mc{A}$ and $\mc{B}$ are isomorphic $\mf{L}^{\msf{rel}}$\nobreakdash-structures such that $G_f$ is the graph of a function in both structures for every function symbol $f \in \mf{L}$, then the isomorphism is also an isomorphism between $\mc{A}^{\msf{fun}}$ and $\mc{B}^{\msf{fun}}$ as $\mf{L}$\nobreakdash-structures.

If $\mf{L}$ is a computable language and $\mc{A}$ is a computable $\mf{L}$\nobreakdash-structure, then $\mf{L}^{\msf{rel}}$ is a computable language and $\mc{A}^{\msf{rel}}$ is a computable $\mf{L}^{\msf{rel}}$\nobreakdash-structure.  Let $\mf{L}$ be a computable language, let $\mc{A}$ and $\mc{B}$ be computably isomorphic $\mf{L}$\nobreakdash-structures, and let $C$ be a cohesive set.  Then $\mc{A}^{\msf{rel}}$ and $\mc{B}^{\msf{rel}}$ are computably isomorphic computable $\mc{L}^{\msf{rel}}$\nobreakdash-structures, so $\prod_C \mc{A}^{\msf{rel}} \iso \prod_C \mc{B}^{\msf{rel}}$ as $\mf{L}^{\msf{rel}}$\nobreakdash-structures.  For each function symbol $f \in \mf{L}$, $G_f$ is the graph of a function in $\mc{A}^{\msf{rel}}$ and in $\mc{B}^{\msf{rel}}$.  Therefore $G_f$ is the graph of a function in $\prod_C \mc{A}^{\msf{rel}}$ and in $\prod_C \mc{B}^{\msf{rel}}$ by Theorem~\ref{thm-LosGeneral} item~\ref{it-LosDelta3Sent} because the statement ``$G_f$ is the graph of a function'' is expressible by a $\Pi_2$ $\mf{L}^{\msf{rel}}$\nobreakdash-sentence.  Therefore $(\prod_C \mc{A}^{\msf{rel}})^{\msf{fun}} \,\iso\, (\prod_C \mc{B}^{\msf{rel}})^{\msf{fun}}$ as $\mf{L}$\nobreakdash-structures.  It is straightforward to check that $\prod_C \mc{A}^{\msf{rel}} = (\prod_C \mc{A})^{\msf{rel}}$ and therefore that $(\prod_C \mc{A}^{\msf{rel}})^{\msf{fun}} = \prod_C \mc{A}$; and similarly for $\mc{B}$.  Thus $\prod_C \mc{A} \iso \prod_C \mc{B}$ as $\mf{L}$\nobreakdash-structures, as desired.
\end{proof}

Recall that a computable structure $\mc{A}$ is called \emph{computably categorical} if every computable structure that is isomorphic to $\mc{A}$ is isomorphic to $\mc{A}$ via a computable isomorphism.  It follows from Theorem~\ref{thm-IsoCohPow} that if $\mc{A}$ is a computably categorical computable structure and $C$ is cohesive, then $\prod_C \mc{A} \iso \prod_C \mc{B}$ whenever $\mc{B}$ is a computable structure isomorphic to $\mc{A}$.

\begin{Corollary}
Let $\mf{L}$ be a computable language, let $\mc{A}$ be a computably categorical computable $\mf{L}$\nobreakdash-structure, let $\mc{B}$ be a computable $\mf{L}$\nobreakdash-structure that is isomorphic to $\mc{A}$, and let $C$ be cohesive.  Then $\prod_C \mc{A} \iso \prod_C \mc{B}$.
\end{Corollary}

In Theorem~\ref{thm-IsoCohPow}, it is essential that the two structures are isomorphic via a computable isomorphism.  In Sections~\ref{sec-PowOfOmega},~\ref{sec-DenseNonstd},~and~\ref{sec-Shuffle}, we see many examples of pairs of computable linear orders that are isomorphic (but not computably isomorphic) to $\omega$ but have non-elementarily equivalent cohesive powers.

The next theorem says that decidable structures $\mc{A}$ and $\mc{B}$ are elementarily equivalent if and only if $\prod_C \mc{A}$ and $\prod_C \mc{B}$ are isomorphic for every cohesive set $C$.  It is essentially~\cite{Nelson}*{Theorem~2.1}, though we give a slightly different proof.  Compare this to the Keisler--Shelah theorem, which states that, in general, two structures are elementarily equivalent if and only if there is an ultrafilter (on a set of appropriate size) over which the corresponding ultrapowers are isomorphic (see~\cite{ChangKeislerBook}*{Theorem~6.1.15}).

\begin{Theorem}\label{thm-EffectiveKS}
Let $\mf{L}$ be a computable language, and let $\mc{A}$ and $\mc{B}$ be decidable $\mf{L}$\nobreakdash-structures.  Then $\mc{A} \equiv \mc{B}$ if and only if for every cohesive set $C$, $\prod_C \mc{A} \iso \prod_C \mc{B}$.
\end{Theorem}

\begin{proof}
In general, say that two structures $\mc{M}$ and $\mc{N}$ have the same types (without parameters) if for every sequence $\vec{a} = a_0, \dots, a_{m-1}$ of elements of $|\mc{M}|$, there is a corresponding sequence $\vec{b} = b_0, \dots, b_{m-1}$ of elements of $|\mc{N}|$ such that for every formula $\Phi(x_0, \dots, x_{m-1})$ with $m$ free variables,
\begin{align*}
\mc{M} \models \Phi(a_0, \dots, a_{m-1}) \;\Biimp\; \mc{N} \models \Phi(b_0, \dots, b_{m-1}),
\end{align*}
and similarly with the roles of $\mc{M}$ and $\mc{N}$ reversed.  Now recall~\cite{KayeBook}*{Corollary~15.15}, which states that if $\mc{M}$ and $\mc{N}$ are countable recursively saturated $\mf{L}$\nobreakdash-structures, then $\mc{M} \iso \mc{N}$ if and only if $\mc{M}$ and $\mc{N}$ are elementarily equivalent and have the same types.

For the forward direction, let $\mc{A}$ and $\mc{B}$ be decidable $\mf{L}$\nobreakdash-structures that are elementarily equivalent, and let $C$ be cohesive.  Then $\prod_C \mc{A}$ and $\prod_C \mc{B}$ are countable, are elementarily equivalent by Corollary~\ref{cor-DecLos} (which yields that $\prod_C \mc{A} \equiv \mc{A} \equiv \mc{B} \equiv \prod_C \mc{B}$), and are recursively saturated by Theorem~\ref{thm-SatGen}.  Thus to conclude that $\prod_C \mc{A} \iso \prod_C \mc{B}$, it suffices to show that $\prod_C \mc{A}$ and $\prod_C \mc{B}$ have the same types.

We show that for every $[\varphi_0], \dots, [\varphi_{m-1}] \in |\prod_C \mc{A}|$, there are $[\psi_0], \dots, [\psi_{m-1}] \in |\prod_C \mc{B}|$ with the same type.  A symmetric argument shows that the same holds with the roles of $\mc{A}$ and $\mc{B}$ reversed.  As in the proofs of Lemmas~\ref{lem-SatGen} and~\ref{lem-SatCoCe}, let $\varphi \colon \Nb \imp \Nb^m$ be the partial computable function $\varphi(n) \keq \la \varphi_0(n), \dots, \varphi_{m-1}(n) \ra$, and let $\ora{[\varphi]}$ denote $[\varphi_0], \dots, [\varphi_{m-1}]$.  Note that $C \subseteq^* \dom(\varphi)$.  Let $(\Phi_i(\vec{x}) : i \in \Nb)$ be a computable enumeration of all formulas with $m$ free variables.

Define a partial computable function $\psi \colon \Nb \imp \Nb^m$ as follows.   If $\varphi(n)\da$, then for each $i \leq n$, use the decidability of $\mc{A}$ to determine whether $\mc{A} \models \Phi_i(\varphi(n))$.  If $\mc{A} \models \Phi_i(\varphi(n))$, let $\Theta_i = \Phi_i$; and if $\mc{A} \nmodels \Phi_i(\varphi(n))$, let $\Theta_i = \neg\Phi_i$.  Then $\varphi(n)$ witnesses that $\mc{A} \models \exists \vec{x}\, \bigwedge_{i \leq n} \Theta_i (\vec{x})$, so $\mc{B} \models \exists \vec{x}\, \bigwedge_{i \leq n} \Theta_i (\vec{x})$ because $\mc{B} \equiv \mc{A}$.  By the decidability of $\mc{B}$, search for the first $\vec{b}$ such that $\mc{B} \models \bigwedge_{i \leq n} \Theta_i (\vec{b})$, and define $\psi(n) = \vec{b}$.  On the other hand, if $\varphi(n)\ua$, then $\psi(n)\ua$.

Consider the formula $\Phi_i$, and suppose that $\prod_C \mc{A} \models \Phi_i \left(\ora{[\varphi]}\right)$.  Then $C \subseteq^* \{n : \mc{A} \models \Phi_i(\varphi(n))\}$ by Corollary~\ref{cor-DecLos}.  Thus for sufficiently large $n \geq i$ with $n \in C$, we have that $\Theta_i = \Phi_i$ in the computation of $\psi(n)$, and therefore $\psi(n)$ is defined so that $\mc{B} \models \Phi_i(\psi(n))$.  Thus $C \subseteq^* \{n : \mc{B} \models \Phi_i(\psi(n))\}$.  Letting $\psi_j = \pi_j \circ \psi$ for each $j < m$ yields that $\prod_C \mc{B} \models \Phi_i \left(\ora{[\psi]}\right)$ by Corollary~\ref{cor-DecLos}.  If instead $\prod_C \mc{A} \models \neg\Phi_i \left(\ora{[\varphi]}\right)$, then $\prod_C \mc{B} \models \neg\Phi_i \left(\ora{[\psi]}\right)$ by a similar argument.  Thus $[\psi_0], \dots, [\psi_{m-1}]$ has in $\prod_C \mc{B}$ the same type that $[\varphi_0], \dots, [\varphi_{m-1}]$ has in $\prod_C \mc{A}$, as desired.

For the converse, let $\mc{A}$ and $\mc{B}$ be decidable $\mf{L}$\nobreakdash-structures, and suppose that $\prod_C \mc{A} \iso \prod_C \mc{B}$ for every cohesive set $C$.  Fix any cohesive set $C$.  Then $\mc{A} \equiv \prod_C \mc{A} \equiv \prod_C \mc{B} \equiv \mc{B}$ by Corollary~\ref{cor-DecLos}.
\end{proof}

Again, the decidability assumption in Theorem~\ref{thm-EffectiveKS} is essential, as we shall see examples of isomorphic (and hence elementarily equivalent) computable linear orders having non-elementarily equivalent (and hence non-isomorphic) cohesive powers.  We shall also see examples of non-elementarily equivalent computable linear orders having isomorphic cohesive powers.

\section{Linear orders and their cohesive powers}\label{sec-LOandPow}

We investigate the cohesive powers of computable linear orders, with special attention to computable linear orders of type $\omega$.  A \emph{linear order} $\mc{L} = (L, \prec)$ consists of a non-empty set $L$ equipped with a binary relation $\prec$ satisfying the following axioms.
\begin{itemize}
\item $\forall x \, (x \nprec x)$.
\item $\forall x \forall y \forall z \, [(x \prec y \andd y \prec z) \;\imp\; x \prec z]$.
\item $\forall x \forall y \, (x \prec y \,\orr\, x = y \,\orr\, y \prec x)$.
\end{itemize}
Additionally, a linear order $\mc{L}$ is \emph{dense} if $\forall x \forall y \exists z \, (x \prec y \,\imp\, x \prec z \prec y)$ and \emph{has no endpoints} if $\forall x \exists y \exists z \, (y \prec x \prec z)$.  Rosenstein's book~\cite{RosBook} is an excellent reference for linear orders.

For a linear order $\mc{L} = (L, \prec)$, we use the usual interval notation $(a,b)_\mc{L} = \{x \in L : a \prec x \prec b\}$ and $[a,b]_\mc{L} = \{x \in L : a \preceq x \preceq b\}$ to denote open and closed intervals of $\mc{L}$.  Sometimes it is convenient to allow $b \preceq a$ in this notation, in which case, for example, $(a,b)_\mc{L} = \emptyset$.  The notation $|(a,b)_\mc{L}|$ denotes the cardinality of the interval $(a,b)_\mc{L}$.  The notations $\min_\prec \{a, b\}$ and $\max_\prec \{a, b\}$ denote the minimum and maximum of $a$ and $b$ with respect to $\prec$.  

As is customary, $\omega$ denotes the order-type of $(\Nb, <)$, $\zeta$ denotes the order-type of $(\Zb, <)$, and $\eta$ denotes the order-type of $(\Qb, <)$.  That is, $\omega$, $\zeta$, and $\eta$ denote the respective order-types of the natural numbers, the integers, and the rationals, each with their usual order.  We refer to $(\Nb, <)$, $(\Zb, <)$, and $(\Qb, <)$ as the \emph{usual presentations} of $\omega$, $\zeta$, and $\eta$, respectively.  Recall that every countable dense linear order without endpoints has order-type $\eta$ (see~\cite{RosBook}*{Theorem~2.8}).  Furthermore, every computable countable dense linear order without endpoints is computably isomorphic to $(\Qb, <)$ (see~\cite{RosBook}*{Exercise~16.4}).

To help reason about order-types, we use the \emph{sum}, \emph{product}, and \emph{reverse} of linear orders as well as \emph{condensations} of linear orders.

\begin{Definition}\label{def-SumProdRev}
Let $\mc{L}_0 = (L_0, \prec_{\mc{L}_0})$ and $\mc{L}_1 = (L_0, \prec_{\mc{L}_1})$ be linear orders.
\begin{itemize}
\item The \emph{sum} $\mc{L}_0 + \mc{L}_1$ of $\mc{L}_0$ and $\mc{L}_1$ is the linear order $\mc{S} = (S, \prec_{\mc{S}})$, where $S = (\{0\} \times L_0) \cup (\{1\} \times L_1)$ and
\begin{align*}
(i, x) \prec_{\mc{S}} (j, y) \quad\text{if and only if}\quad (i < j) \;\orr\; (i = j \,\andd\, x \prec_{\mc{L}_i} y).
\end{align*}

\medskip

\item The \emph{product} $\mc{L}_0 \mc{L}_1$ of $\mc{L}_0$ and $\mc{L}_1$ is the linear order $\mc{P} = (P, \prec_{\mc{P}})$, where $P = L_1 \times L_0$ and
\begin{align*}
(x, a) \prec_{\mc{P}} (y, b) \quad\text{if and only if}\quad (x \prec_{\mc{L}_1} y) \;\orr\; (x = y \,\andd\, a \prec_{\mc{L}_0} b).
\end{align*}
Note that, by (fairly entrenched) convention, $\mc{L}_0 \mc{L}_1$ is given by the product order on $L_1 \times L_0$, not on $L_0 \times L_1$.

\medskip

\item The \emph{reverse} $\mc{L}_0^*$ of $\mc{L}_0$ is the linear order $\mc{R} = (R, \prec_{\mc{R}})$, where $R = L_0$ and $x \prec_{\mc{R}} y$ if and only if $y \prec_{\mc{L}_0} x$.  (We warn the reader that the $*$ in the notation $\mc{L}_0^*$ is unrelated to the $*$ in the notation $X \subseteq^* Y$.)
\end{itemize}
\end{Definition}

If $\mc{L}_0$ and $\mc{L}_1$ are computable linear orders, then one may use the pairing function to compute copies of $\mc{L}_0 + \mc{L}_1$ and $\mc{L}_0 \mc{L}_1$.  Clearly, if $\mc{L}$ is a computable linear order, then so is $\mc{L}^*$. 

\begin{Definition}
Let $\mc{L} = (L, \prec_\mc{L})$ be a linear order.  A \emph{condensation} of $\mc{L}$ is any linear order $\mc{M} = (M, \prec_\mc{M})$ obtained by partitioning $L$ into a collection $M$ of non-empty intervals and, for intervals $I, J \in M$, defining $I \prec_\mc{M} J$ if and only if $(\forall a \in I)(\forall b \in J)(a \prec_\mc{L} b)$.
\end{Definition}

The most important condensation is the \emph{finite condensation}.

\begin{Definition}\label{def-FinCond}
Let $\mc{L} = (L, \prec_\mc{L})$ be a linear order.  For $x \in L$, let $\condF(x)$ denote the set of $y \in L$ for which there are only finitely many elements between $x$ and $y$:
\begin{align*}
\condF(x) = \Bigl\{y \in L : \text{the interval $\bigl[\min\nolimits_{\prec_\mc{L}}\{x,y\}, \max\nolimits_{\prec_\mc{L}}\{x,y\}\bigr]_{\mc{L}}$ in $\mc{L}$ is finite} \Bigl\}.
\end{align*}
The set $\condF(x)$ is always a non-empty interval because $x \in \condF(x)$.  The \emph{finite condensation} $\condF(\mc{L})$ of $\mc{L}$ is the condensation obtained from the partition $\{\condF(x) : x \in L\}$.
\end{Definition}

For example, $\condF(\omega) \iso \bm{1}$, $\condF(\zeta) \iso \bm{1}$, $\condF(\eta) \iso \eta$, and $\condF(\omega + \zeta\eta) \iso \bm{1} + \eta$.  Notice that for an element $x$ of a linear order $\mc{L}$, the order-type of $\condF(x)$ is always either finite, $\omega$, $\omega^*$, or $\zeta$.

We often refer to the intervals that comprise a condensation of a linear order as \emph{blocks}.  For the finite condensation of a linear order $\mc{L}$, a block is a maximal interval $I$ such that for any two elements of $I$, there are only finitely many elements of $\mc{L}$ between them.  For elements $a$ and $b$ of $\mc{L}$, we write $a \pprec_\mc{L} b$ if the interval $(a, b)_\mc{L}$ (equivalently, the interval $[a, b]_\mc{L}$) in $\mc{L}$ is infinite.  For $a \prec_\mc{L} b$, we have that $a \pprec_\mc{L} b$ if and only if $a$ and $b$ are in different blocks.  See~\cite{RosBook}*{Chapter~4} for more on condensations.

Let $C$ be a cohesive set.  It follows from Theorem~\ref{thm-LosProdParam} that if $(\mc{L}_n : n \in \Nb)$ is a uniformly computable sequence of linear orders, then $\prod_C \mc{L}_n$ is again a linear order because linear orders are axiomatized by $\Pi_1$ sentences.  Likewise, if $(\mc{L}_n : n \in \Nb)$ is a uniformly computable sequence of dense linear orders without endpoints, then $\prod_C \mc{L}_n$ is again a dense linear order without endpoints because dense linear orders without endpoints are axiomatized by $\Pi_2$ sentences.  In particular, if $\mc{L}$ is a computable linear order, then $\prod_C \mc{L}$ is a linear order; and if $\mc{L}$ is a computable dense linear order without endpoints, then $\prod_C \mc{L}$ is a dense linear order without endpoints.

The case of $\Qb = (\Qb, <)$ is curious and deserves a digression.  We have seen that if $\mc{A}$ is a finite structure, then $\mc{A} \iso \prod_C \mc{A}$ for every cohesive set $C$.  For $\Qb$, $\prod_C \Qb$ is a countable dense linear order without endpoints, and hence isomorphic to $\Qb$, for every cohesive set $C$.  Thus $\Qb$ is an example of an infinite computable structure with $\Qb \iso \prod_C \Qb$ for every cohesive set $C$.  That $\Qb$ is isomorphic to all of its cohesive powers is no accident.  By combining Theorem~\ref{thm-LosGeneral} with the theory of \emph{Fra\"{i}ss\'{e} limits} (see~\cite{HodgesBook}*{Chapter~6}, for example), we see that a uniformly locally finite ultrahomogeneous computable structure for a finite language is always isomorphic to all of its cohesive powers.  Recall that a structure is \emph{locally finite} if every finitely-generated substructure is finite and is \emph{uniformly locally finite} if there is a function $f \colon \Nb \imp \Nb$ such that every substructure generated by at most $n$ elements has cardinality at most $f(n)$.  Notice that every structure for a finite relational language is uniformly locally finite.  Also recall that a structure is \emph{ultrahomogeneous} if every isomorphism between two finitely-generated substructures extends to an automorphism of the whole structure.

\begin{Proposition}\label{prop-UltrahomIso}
Let $\mc{A}$ be an infinite uniformly locally finite ultrahomogeneous computable structure for a finite language, and let $C$ be cohesive.  Then $\mc{A} \iso \prod_C \mc{A}$.
\end{Proposition}

\begin{proof}
The structure $\mc{A}$ is ultrahomogeneous, so it is the Fra\"{i}ss\'{e} limit of its \emph{age} (i.e., the class of all finitely-generated structures embeddable into $\mc{A}$).  By~\cite{HodgesBook}*{Theorem~6.4.1} and its proof, the first-order theory of $\mc{A}$ is $\aleph_0$\nobreakdash-categorical and is axiomatized by a set $T$ of $\Pi_2$ sentences.  Thus if $\mc{B}$ is any countable model of $T$, then $\mc{A} \iso \mc{B}$.  We have that $\prod_C \mc{A} \models T$ by Theorem~\ref{thm-LosGeneral} item~\ref{it-LosDelta3Sent}, so $\mc{A} \iso \prod_C \mc{A}$.
\end{proof}

Proposition~\ref{prop-UltrahomIso} implies that if a uniformly locally finite computable structure for a finite language is a Fra\"{i}ss\'{e} limit, then it is isomorphic to all of its cohesive powers.  Thus computable presentations of the Rado graph and the countable atomless Boolean algebra are additional examples of computable structures that are isomorphic to all of their cohesive powers.  Examples of this phenomenon that cannot be attributed to ultrahomogeneity appear in Sections~\ref{sec-PowOfOmega} and~\ref{sec-DenseNonstd}.
 
Returning to linear orders, it is helpful to recall the following well-known lemma stating that a strictly order-preserving surjection from one linear order onto another is necessarily an isomorphism.
\begin{Lemma}\label{lem-LOIso}
Let $\mc{L} = (L, \prec_\mc{L})$ and $\mc{M} = (M, \prec_\mc{M})$ be linear orders.  If $f \colon L \imp M$ is surjective and satisfies $(\forall x, y \in L)(x \prec_\mc{L} y \;\imp\; f(x) \prec_\mc{M} f(y))$, then $f$ is an isomorphism.\qed
\end{Lemma}

Cohesive powers commute with sums, products, and reverses.

\begin{Theorem}\label{thm-CohPres}
Let $\mc{L}_0$ and $\mc{L}_1$ be computable linear orders, and let $C$ be cohesive.  Then
\begin{enumerate}[(1)]
\item\label{it-SumIso} $\prod_C(\mc{L}_0 + \mc{L}_1) \iso \prod_C\mc{L}_0 + \prod_C\mc{L}_1$,

\smallskip

\item\label{it-ProdIso} $\prod_C(\mc{L}_0\mc{L}_1) \iso \bigl( \prod_C\mc{L}_0 \bigr) \bigl( \prod_C\mc{L}_1\bigr)$, and

\smallskip

\item\label{it-RevIso} $\prod_C(\mc{L}_0^*) \iso \bigl( \prod_C\mc{L}_0 \bigr)^*$.
\end{enumerate}
\end{Theorem}

\begin{proof}
We give proofs in the style of that of Theorem~\ref{thm-IsoCohPow}.  For alternate proofs, see~\cite{CohPowCiE}*{Theorem~6}.

For~\ref{it-SumIso}, let $\mc{M} = \mc{L}_0 + \mc{L}_1$.  Expand the disjoint union $\mc{L}_0 \sqcup \mc{L}_1 \sqcup \mc{M}$ by adding unary relation symbols $L_0$, $L_1$, and $M$ that are interpreted as $|\mc{L}_0|$, $|\mc{L}_1|$, and $|\mc{M}|$; and by adding a $2$\nobreakdash-ary relation symbol $R_f$ that is interpreted as the graph of the function $f \colon |\mc{L}_0 \sqcup \mc{L}_1| \imp |\mc{M}|$ inside $\mc{L}_0 \sqcup \mc{L}_1 \sqcup \mc{M}$ given by
\begin{align*}
f(\la i, x \ra) = \la 2, \la i, x \ra \ra
\end{align*}
for each $i < 2$.
In $\mc{L}_0 \sqcup \mc{L}_1 \sqcup \mc{M}$, the $\prec$ relation is a linear order when restricted to $|\mc{L}_0|$, $|\mc{L}_1|$, or $|\mc{M}|$.  Furthermore, $R_f$ is the graph of a function $f$ from $|\mc{L}_0 \sqcup \mc{L}_1|$ onto $|\mc{M}|$ with the following properties, which together witness that $\mc{M} \iso \mc{L}_0 + \mc{L}_1$ as linear orders.
\begin{itemize}
\item $\forall x \forall y\, [(L_0(x) \,\andd\, L_0(y) \,\andd\, x \prec y) \;\imp\; f(x) \prec f(y)]$.
\item $\forall x \forall y\, [(L_1(x) \,\andd\, L_1(y) \,\andd\, x \prec y) \;\imp\; f(x) \prec f(y)]$.
\item $\forall x \forall y\, [(L_0(x) \,\andd\, L_1(y)) \;\imp\; f(x) \prec f(y)]$.
\end{itemize}
All of the above is expressible by a $\Pi_2$ sentence.  By Proposition~\ref{prop-substr}, the substructures of $\prod_C(\mc{L}_0 \sqcup \mc{L}_1 \sqcup \mc{M})$ corresponding to $L_0$, $L_1$, and $M$ are isomorphic to $\prod_C \mc{L}_0$, $\prod_C \mc{L}_1$, and $\prod_C \mc{M}$.  By Theorem~\ref{thm-LosGeneral} item~\ref{it-LosDelta3Sent}, $\prod_C \mc{L}_0$, $\prod_C \mc{L}_1$, and $\prod_C \mc{M}$ are linear orders as $\{\prec\}$\nobreakdash-structures, and $R_f^{\prod_C(\mc{L}_0 \sqcup \mc{L}_1 \sqcup \mc{M})}$ yields a function from $|\prod_C \mc{L}_0 \sqcup \prod_C \mc{L}_1|$ onto $|\prod_C \mc{M}|$ witnessing that $\prod_C \mc{M} \iso \prod_C \mc{L}_0 + \prod_C \mc{L}_1$ as linear orders.  Thus $\prod_C(\mc{L}_0 + \mc{L}_1) \iso \prod_C\mc{L}_0 + \prod_C\mc{L}_1$.

For~\ref{it-ProdIso}, let $\mc{M} = \mc{L}_0 \mc{L}_1$.  Expand the disjoint union $\mc{L}_0 \sqcup \mc{L}_1 \sqcup \mc{M}$ by adding unary relation symbols $L_0$, $L_1$, and $M$ that are interpreted as $|\mc{L}_0|$, $|\mc{L}_1|$, and $|\mc{M}|$; and by adding a $3$\nobreakdash-ary relation symbol $R_f$ that is interpreted as the graph of the function $f \colon |\mc{L}_1| \times |\mc{L}_0| \imp |\mc{M}|$ inside $\mc{L}_0 \sqcup \mc{L}_1 \sqcup \mc{M}$ given by 
\begin{align*}
f(\la 1, x \ra, \la 0, a \ra) = \la 2, \la x, a \ra \ra.
\end{align*}
In $\mc{L}_0 \sqcup \mc{L}_1 \sqcup \mc{M}$, the $\prec$ relation is a linear order when restricted to $|\mc{L}_0|$, $|\mc{L}_1|$, or $|\mc{M}|$.  Furthermore, $R_f$ is the graph of a function $f$ from $|\mc{L}_1| \times |\mc{L}_0|$ onto $|\mc{M}|$ with the following property, which witnesses that $\mc{M} \iso \mc{L}_0 \mc{L}_1$ as linear orders.
\begin{align*}
\forall a \forall b \forall x \forall y \, \biggl[\Bigl(L_0(a) &\,\andd\, L_0(b) \,\andd\, L_1(x) \,\andd\, L_1(y)\Bigr) \;\imp\\
&\Bigl( f(x,a) \prec f(y,b) \;\;\biimp\;\; \bigl( x \prec y \;\orr\; (x = y \,\andd\, a \prec b) \bigr) \Bigr)\biggr].
\end{align*}
All of the above is expressible by a $\Pi_2$ sentence.  By Proposition~\ref{prop-substr}, the substructures of $\prod_C(\mc{L}_0 \sqcup \mc{L}_1 \sqcup \mc{M})$ corresponding to $L_0$, $L_1$, and $M$ are isomorphic to $\prod_C \mc{L}_0$, $\prod_C \mc{L}_1$, and $\prod_C \mc{M}$.  By Theorem~\ref{thm-LosGeneral} item~\ref{it-LosDelta3Sent}, $\prod_C \mc{L}_0$, $\prod_C \mc{L}_1$, and $\prod_C \mc{M}$ are linear orders as $\{\prec\}$\nobreakdash-structures, and $R_f^{\prod_C(\mc{L}_0 \sqcup \mc{L}_1 \sqcup \mc{M})}$ yields a function from $|\prod_C \mc{L}_1| \times |\prod_C \mc{L}_0|$ onto $|\prod_C \mc{M}|$ witnessing that $\prod_C \mc{M} \iso \bigl( \prod_C \mc{L}_0 \bigr) \bigl( \prod_C \mc{L}_1 \bigr)$ as linear orders.  Thus $\prod_C(\mc{L}_0 \mc{L}_1) \iso \bigl( \prod_C \mc{L}_0 \bigr) \bigl( \prod_C \mc{L}_1 \bigr)$.

For~\ref{it-RevIso}, let $\mc{M} = \mc{L}_0^*$.  Expand the disjoint union $\mc{L}_0 \sqcup \mc{M}$ by adding unary relation symbols $L_0$ and $M$ that are interpreted as $|\mc{L}_0|$ and $|\mc{M}|$; and by adding a $2$\nobreakdash-ary relation symbol $R_f$ that is interpreted as the graph of the function $f \colon |\mc{L}_0| \imp |\mc{M}|$ inside $\mc{L}_0 \sqcup \mc{M}$ given by $f(\la 0, x \ra) = \la 1, x \ra$.  In $\mc{L}_0 \sqcup \mc{M}$, the $\prec$ relation is a linear order when restricted to $|\mc{L}_0|$ or $|\mc{M}|$, and $R_f$ is the graph of a function $f$ from $|\mc{L}_0|$ onto $|\mc{M}|$ such that
\begin{align*}
\forall x \forall y \, [(L_0(x) \,\andd\, L_0(y) \,\andd\, x \prec y) \;\imp\; f(y) \prec f(x)],
\end{align*}
which witnesses that $\mc{M} \iso \mc{L}_0^*$ as linear orders.  All of the above is expressible by a $\Pi_2$ sentence.  By Proposition~\ref{prop-substr}, the substructures of $\prod_C(\mc{L}_0 \sqcup \mc{M})$ corresponding to $L_0$ and $M$ are isomorphic to $\prod_C \mc{L}_0$ and $\prod_C \mc{M}$.  By Theorem~\ref{thm-LosGeneral} item~\ref{it-LosDelta3Sent}, $\prod_C \mc{L}_0$ and $\prod_C \mc{M}$ are linear orders as $\{\prec\}$\nobreakdash-structures, and $R_f^{\prod_C(\mc{L}_0 \sqcup \mc{M})}$ yields a function from $|\prod_C \mc{L}_0|$ onto $|\prod_C \mc{M}|$ witnessing that $\prod_C \mc{M} \iso \bigl( \prod_C \mc{L}_0 \bigr)^*$ as linear orders.  Thus $\prod_C(\mc{L}_0^*) \iso \bigl( \prod_C \mc{L}_0 \bigr)^*$.
\end{proof}

Sections~\ref{sec-PowOfOmega},~\ref{sec-DenseNonstd}, and~\ref{sec-Shuffle} concern calculating the order-types of cohesive powers of computable copies of $\omega$.  To do this, we must be able to determine when one element of a cohesive power is an immediate successor or immediate predecessor of another, and we must be able to determine when two elements of a cohesive power are in different blocks of its finite condensation.

In a cohesive power $\prod_C \mc{L}$ of a computable linear order $\mc{L}$, $[\varphi]$ is the immediate successor of $[\psi]$ if and only if $\varphi(n)$ is the immediate successor of $\psi(n)$ for almost every $n \in C$.  Therefore also $[\psi]$ is the immediate predecessor of $[\varphi]$ if and only if $\psi(n)$ is the immediate predecessor of $\varphi(n)$ for almost every $n \in C$.

\begin{Lemma}\label{lem-ImmedSucc}
Let $(\mc{L}_n : n \in \Nb)$ be a uniformly computable sequence of linear orders, let $C$ be cohesive, and let $[\psi]$ and $[\varphi]$ be elements of $\prod_C \mc{L}_n$.  Then the following are equivalent.
\begin{enumerate}[(1)]
\item\label{it-ImmedSuccInPow} $[\varphi]$ is the $\prec_{\prod_C \mc{L}_n}$\nobreakdash-immediate successor of $[\psi]$.

\smallskip

\item\label{it-aeImmedSucc} $(\forae n \in C)(\text{$\varphi(n)$ is the $\prec_{\mc{L}_n}$\nobreakdash-immediate successor of $\psi(n)$})$.

\smallskip

\item\label{it-existsinfImmedSucc} $(\existsinf n \in C)(\text{$\varphi(n)$ is the $\prec_{\mc{L}_n}$\nobreakdash-immediate successor of $\psi(n)$})$.
\end{enumerate}
\end{Lemma}

\begin{proof}
That $x$ is the $\prec$\nobreakdash-immediate successor of $y$ is a $\Pi_1$ property of $x$ and $y$.  Therefore items~\ref{it-ImmedSuccInPow} and~\ref{it-aeImmedSucc} are equivalent by Theorem~\ref{thm-LosProdParam} item~\ref{it-LosProdPramDelta2}.  The set
\begin{align*}
\{n : \text{$\varphi(n)$ is the $\prec_{\mc{L}_n}$\nobreakdash-immediate successor of $\psi(n)$}\}
\end{align*}
is the intersection of a c.e.\ set and a co-c.e.\ set, so items~\ref{it-aeImmedSucc} and~\ref{it-existsinfImmedSucc} are equivalent by cohesiveness.
\end{proof}

\begin{Lemma}\label{lem-DifferentBlocks}
Let $(\mc{L}_n : n \in \Nb)$ be a uniformly computable sequence of linear orders, let $C$ be cohesive, and let $[\psi]$ and $[\varphi]$ be elements of $\prod_C \mc{L}_n$.  Then the following are equivalent.
\begin{enumerate}[(1)]
\item\label{it-FarBelow} $[\psi] \pprec_{\prod_C \mc{L}_n} [\varphi]$.

\smallskip

\item\label{it-BelowLim} $\lim_{n \in C}|(\psi(n), \varphi(n))_{\mc{L}_n}| = \infty$.

\smallskip

\item\label{it-BelowLimSup} $\limsup_{n \in C}|(\psi(n), \varphi(n))_{\mc{L}_n}| = \infty$.
\end{enumerate}
\end{Lemma}

\begin{proof}
We show that the following are equivalent for each fixed $k \in \Nb$.
\begin{enumerate}[(i)]
\item\label{it-AtLeastKInPow} $|([\psi], [\varphi])_{\prod_C \mc{L}_n}| \geq k$.

\smallskip

\item\label{it-aeAtLeastK} $(\forae n \in C)(|(\psi(n), \varphi(n))_{\mc{L}_n}| \geq k)$.

\smallskip

\item\label{it-existsinfAtLeastK} $(\existsinf n \in C)(|(\psi(n), \varphi(n))_{\mc{L}_n}| \geq k)$.
\end{enumerate}

That an interval $(x,y)$ in a linear order contains at least $k$ distinct elements for a fixed $k$ is a $\Sigma_1$ property of $x$ and $y$.  Therefore items~\ref{it-AtLeastKInPow} and~\ref{it-aeAtLeastK} are equivalent by Theorem~\ref{thm-LosProdParam} item~\ref{it-LosProdPramDelta2}.  The set $\{n : |(\psi(n), \varphi(n))_{\mc{L}_n}| \geq k\}$ is c.e., so items~\ref{it-aeAtLeastK} and~\ref{it-existsinfAtLeastK} are equivalent by cohesiveness.

It now follows that items~\ref{it-FarBelow}--\ref{it-BelowLimSup} are equivalent.  Item~\ref{it-FarBelow} holds if and only if item~\ref{it-AtLeastKInPow} holds for every $k$; item~\ref{it-BelowLim} holds if and only if item~\ref{it-aeAtLeastK} holds for every $k$; and item~\ref{it-BelowLimSup} holds if and only if item~\ref{it-existsinfAtLeastK} holds for every $k$.
\end{proof}

The finite condensation of a cohesive product of computable linear orders by a co-c.e.\ cohesive set is always dense.

\begin{Theorem}\label{thm-DenseBlocks}
Let $(\mc{L}_n : n \in \Nb)$ be a uniformly computable sequence of linear orders, and let $C$ be a cohesive set.  If either $(\mc{L}_n : n \in \Nb)$ is uniformly $1$\nobreakdash-decidable or $C$ is co-c.e., then $\condF(\prod_C \mc{L}_n)$ is dense.
\end{Theorem}

\begin{proof}
The cohesive product $\prod_C \mc{L}_n$ is $\Sigma_1$\nobreakdash-recursively saturated by Theorem~\ref{thm-SatGen} item~\ref{it-nDecRecSatSeq} in the uniformly $1$\nobreakdash-decidable case and by Corollary~\ref{cor-SatCoCe} item~\ref{it-Sigma1RecSatSeqCoCe} in the co-c.e.\ case.  Thus it suffices to show that if $\mc{M} = (M, \prec_\mc{M})$ is a $\Sigma_1$\nobreakdash-recursively saturated linear order, then $\condF(\mc{M})$ is dense.  To see this, let $a, b \in M$ be such that $a \pprec_\mc{M} b$.  For each $k \in \Nb$, let $\Phi_k(x; a, b)$ be the following formula (with parameters $a$ and $b$) expressing that there are at least $k$ elements between $a$ and $x$ and at least $k$ elements between $x$ and $b$:
\begin{align*}
\Phi_k(x;a,b) \quad\equiv\quad \exists^{\geq k} z \, (a \prec_\mc{M} z \prec_\mc{M} x) \;\andd\; \exists^{\geq k} z \, (x \prec_\mc{M} z \prec_\mc{M} b).
\end{align*}
Let $p(x) = \{\Phi_k(x; a, b) : k \in \Nb\}$.  Then $p(x)$ is a computable set of $\Sigma_1$ formulas.  Furthermore, $p(x)$ is a type over $\mc{M}$ because the interval $(a,b)_\mc{M}$ is infinite.  Therefore $p(x)$ is realized by some $c \in M$ because $\mc{M}$ is $\Sigma_1$\nobreakdash-recursively saturated.  Thus the intervals $(a, c)_\mc{M}$ and $(c, b)_\mc{M}$ are both infinite, so $a \pprec_\mc{M} c \pprec_\mc{M} b$.  It follows that $\condF(\mc{M})$ is dense.
\end{proof}

\section{Cohesive powers of computable copies of \texorpdfstring{$\omega$}{omega}}\label{sec-PowOfOmega}

We investigate the cohesive powers of computable linear orders of type $\omega$.  Observe that an infinite linear order has type $\omega$ if and only if every element has only finitely many predecessors.  We rely on this characterization throughout.  Though not part of the language of linear orders, every linear order $\mc{L}$ has an associated immediate successor relation $S^{\mc{L}} \subseteq |\mc{L}| \times |\mc{L}|$, where $S^{\mc{L}}(a, b)$ holds for $a, b \in |\mc{L}|$ if and only if $b$ is the $\prec_\mc{L}$\nobreakdash-immediate successor of $a$.  As explained in~\cite{Moses}*{Section~3}, a computable linear order $\mc{L}$ is $1$\nobreakdash-decidable if and only if the immediate successor relation $S^{\mc{L}}$ is computable.  It is straightforward to check that a computable copy $\mc{L}$ of $\omega$ is computably isomorphic to the usual presentation $(\Nb, <)$ if and only if $S^\mc{L}$ is computable.  Thus the computable copies of $\omega$ that are computably isomorphic to the usual presentation are exactly the $1$\nobreakdash-decidable copies of $\omega$.

We show that if $\mc{L}$ is a computable copy of $\omega$ that is computably isomorphic to the usual presentation, then every cohesive power of $\mc{L}$ has order-type $\omega + \zeta\eta$ (Theorem~\ref{thm-StdCohPow}).  This is to be expected because $\omega + \zeta\eta$ is familiar as the order-type of every countable non-standard model of Peano arithmetic (see~\cite{KayeBook}*{Theorem~6.4}).  However, being computably isomorphic to the usual presentation is not a characterization of the computable copies of $\omega$ having cohesive powers of order-type $\omega + \zeta\eta$.  We show that there is a computable copy of $\omega$ that is not computably isomorphic to the usual presentation, yet still has every cohesive power isomorphic to $\omega + \zeta\eta$ (Theorem~\ref{thm-CohPowNZQ}).  Thus to compute a copy of $\omega$ having a cohesive power not of type $\omega + \zeta\eta$, one must do more than simply arrange for the immediate successor relation to be non-computable.  We show that for every cohesive set $C$, there is a computable copy $\mc{L}$ of $\omega$ such that the cohesive power $\prod_C \mc{L}$ does not have order-type $\omega + \zeta\eta$ (Theorem~\ref{thm-NotNZQ}).  However, we also show that whenever $\mc{L}$ is a computable copy of $\omega$ and $C$ is a co-c.e.\ cohesive set, the finite condensation $\condF(\prod_C \mc{L})$ of the cohesive power $\prod_C \mc{L}$ always has order-type $\bm{1} + \eta$ (Theorem~\ref{thm-CoCeDenseCond}).

First, a cohesive power of a computable copy of $\omega$ always has an initial segment of order-type $\omega$.

\begin{Lemma}\label{lem-StdInitSeg}
Let $\mc{L}$ be a computable copy of $\omega$, and let $C$ be cohesive.  Then the image of the canonical embedding of $\mc{L}$ into $\prod_C \mc{L}$ is an initial segment of $\prod_C \mc{L}$ of order-type $\omega$.
\end{Lemma}

\begin{proof}
Let $a_0 \prec_\mc{L} a_1 \prec_\mc{L} a_2 \prec_\mc{L} \cdots$ be the (not necessarily computable) listing of $|\mc{L}|$ in $\prec_\mc{L}$\nobreakdash-increasing order.  The image of the canonical embedding consists of those elements of the form $[\psi_{a_k}]$, where $\psi_{a_k}$ is the total computable function with constant value $a_k$.

For each $k \in \Nb$, we have that $\mc{L} \models \exists^{= k} x\, (x \prec_\mc{L} a_k)$ and therefore that $(\forae n \in C)(\exists^{= k} x)(x \prec_\mc{L} \psi_{a_k}(n))$.  Thus also $\prod_C \mc{L} \models (\exists^{= k} x)(x \prec_{\prod_C \mc{L}} [\psi_{a_k}])$ by Theorem~\ref{thm-LosGeneral} item~\ref{it-LosDelta2Param}.  That is, for each $k \in \Nb$, there are exactly $k$ many elements of $\prod_C \mc{L}$ that are $\prec_{\prod_C \mc{L}}$\nobreakdash-below $[\psi_{a_k}]$.  Thus $[\psi_{a_0}] \prec_{\prod_C \mc{L}} [\psi_{a_1}] \prec_{\prod_C \mc{L}} \cdots$ is an initial segment of $\prod_C \mc{L}$ of order-type $\omega$.
\end{proof}

Let $\mc{L}$ be a computable copy of $\omega$, let $C$ be cohesive, and let $\varphi \colon \Nb \imp |\mc{L}|$ be any total computable bijection.  Then $[\varphi]$ is not in the image of the canonical embedding of $\mc{L}$ into $\prod_C \mc{L}$, so it must be $\prec_{\prod_C \mc{L}}$\nobreakdash-above every element in the image of the canonical embedding.  Thus $\prod_C \mc{L}$ is of the form $\omega + \mc{M}$ for some non-empty linear order $\mc{M}$.  By analogy with the terminology for models of arithmetic, we call the elements of the $\omega$\nobreakdash-part of $\prod_C \mc{L}$ (i.e., the image of the canonical embedding) \emph{standard} and the elements of the $\mc{M}$\nobreakdash-part of $\prod_C \mc{L}$ \emph{non-standard}.  In terms of the finite condensation, we have that $\condF(\prod_C \mc{L}) \iso \bm{1} + \mc{N}$ for some linear order $\mc{N}$.  Call the block corresponding to $\bm{1}$ the $\emph{standard block}$ and the blocks corresponding to $\mc{N}$ the \emph{non-standard blocks}.  Lemma~\ref{lem-NonStdNoEnd} below implies that $\mc{N}$ is always infinite and therefore that $\mc{M}$ is always infinite as well.

\begin{Lemma}\label{lem-NonstdUnbdd}
Let $\mc{L} = (L, \prec_\mc{L})$ be a computable copy of $\omega$, let $C$ be cohesive, and let $[\varphi]$ be an element of $\prod_C \mc{L}$.  Then $[\varphi]$ is non-standard if and only if $\liminf_{n \in C} \varphi(n) = \infty$.
\end{Lemma}

\begin{proof}
If $[\varphi]$ is standard, then $\varphi$ is eventually constant on $C$, so $\liminf_{n \in C} \varphi(n)$ is finite.  Conversely, suppose that $\liminf_{n \in C} \varphi(n) = k$ is finite.  Then $(\existsinf n \in C)(\varphi(n) = k)$.  By cohesiveness, it must therefore be that $(\forae n \in C)(\varphi(n) = k)$.  That is, $\varphi$ is eventually constant on $C$, so $[\varphi]$ is standard.  
\end{proof}

In Lemma~\ref{lem-NonstdUnbdd}, the condition $\liminf_{n \in C} \varphi(n) = \infty$ may be replaced by either $\lim_{n \in C} \varphi(n) = \infty$ or $\limsup_{n \in C} \varphi(n) = \infty$ because if $C$ is cohesive, $\varphi$ is partial computable, and $C \subseteq^* \dom(\varphi)$, then $\liminf_{n \in C} \varphi(n) = \infty$, $\lim_{n \in C} \varphi(n) = \infty$, and $\limsup_{n \in C} \varphi(n) = \infty$ are all equivalent conditions.

The following Lemma~\ref{lem-NonStdNoEnd} says that if $\mc{L}$ is a computable copy of $\omega$ and $C$ is cohesive, then $\prod_C \mc{L}$ has neither a least nor a greatest non-standard block.  If $\mc{L}$ is $1$\nobreakdash-decidable or $C$ is co-c.e., then this can be proved by a saturation argument similar to that given in the proof of Theorem~\ref{thm-DenseBlocks}.  Instead, we give a hands-on proof that works for all computable $\mc{L} \iso \omega$ and all cohesive $C$.

\begin{Lemma}\label{lem-NonStdNoEnd}
Let $\mc{L} = (L, \prec_\mc{L})$ be a computable copy of $\omega$, let $C$ be cohesive, and let $[\varphi]$ be a non-standard element of $\prod_C \mc{L}$.  Then there are non-standard elements $[\psi^-]$ and $[\psi^+]$ of $\prod_C \mc{L}$ with $[\psi^-] \pprec_{\prod_C \mc{L}} [\varphi] \pprec_{\prod_C \mc{L}} [\psi^+]$.
\end{Lemma}

\begin{proof}
Let $(\ell_i : i \in \Nb)$ be a computable enumeration of $L$.  Compute a sequence $x_0 \prec_\mc{L} x_1 \prec_\mc{L} x_2 \prec_\mc{L} \cdots$ that is cofinal in $\mc{L}$ by letting $x_0 = \ell_0$ and by letting each $x_{i+1}$ be the $<$\nobreakdash-least number with $\max_{\prec_\mc{L}}\{x_i, \ell_i\} \prec_\mc{L} x_{i+1}$.  Such an $x_{i+1}$ always exists because $\mc{L}$ has no $\prec_\mc{L}$\nobreakdash-maximum element.

Consider a non-standard $[\varphi] \in |\prod_C \mc{L}|$.  Define partial computable functions $\psi^-, \psi^+ \colon \Nb \imp L$ by
\begin{align*}
\psi^-(n) &\keq
\begin{cases}
x_i & \text{if $x_{2i} \preceq_\mc{L} \varphi(n) \prec_\mc{L} x_{2i+2}$}\\
\ua & \text{if $\varphi(n)\ua$}
\end{cases}\\ \\
\psi^+(n) &\keq
\begin{cases}
x_{2i} & \text{if $x_i \preceq_\mc{L} \varphi(n) \prec_\mc{L} x_{i+1}$}\\
\ua & \text{if $\varphi(n)\ua$}.
\end{cases}
\end{align*}

The element $[\varphi]$ is non-standard, so $(\forall i)(\forae n \in C)(x_{2i} \preceq_\mc{L} \varphi(n))$. Thus $(\forall i)(\forae n \in C)(x_i \preceq_\mc{L} \psi^-(n))$, so $[\psi^-]$ is non-standard as well.  Moreover, if $x_{2i} \preceq_\mc{L} \varphi(n) \prec_\mc{L} x_{2i+2}$, then $\psi^-(n) = x_i$, and therefore $|(\psi^-(n), \varphi(n))_\mc{L}| \geq i-1$ because $x_{i+1}, \dots, x_{2i-1} \in (\psi^-(n), \varphi(n))_\mc{L}$.  Therefore $\limsup_{n \in C}|(\psi^-(n), \varphi(n))_\mc{L}| = \infty$, so $[\psi^-] \pprec_{\prod_C \mc{L}} [\varphi]$ by Lemma~\ref{lem-DifferentBlocks}.  Similar reasoning shows that $[\varphi] \pprec_{\prod_C \mc{L}} [\psi^+]$.  Thus $[\psi^-] \pprec_{\prod_C \mc{L}} [\varphi] \pprec_{\prod_C \mc{L}} [\psi^+]$.
\end{proof}

If the computable linear order $\mc{L}$ is 1-decidable or the cohesive set $C$ is co-c.e., then between any two blocks there is a third.  We therefore have the following theorem.

\begin{Theorem}\label{thm-CoCeDenseCond}
Let $\mc{L}$ be a computable copy of $\omega$, and let $C$ be a cohesive set.  If either $\mc{L}$ is $1$\nobreakdash-decidable or $C$ is co-c.e., then $\condF(\prod_C \mc{L})$ has order-type $\bm{1} + \eta$.
\end{Theorem}

\begin{proof}
By Lemma~\ref{lem-StdInitSeg}, the standard elements of $\prod_C \mc{L}$ form an initial block.  By Theorem~\ref{thm-DenseBlocks} and Lemma~\ref{lem-NonStdNoEnd}, the non-standard blocks of $\prod_C \mc{L}$ form a countable dense linear order without endpoints.  Thus $\condF(\prod_C \mc{L}) \iso \bm{1} + \eta$.
\end{proof}

Thinking in terms of blocks, showing that a linear order $\mc{M}$ has type $\omega + \zeta\eta$ amounts to showing that $\mc{M}$ consists of an initial block of order-type $\omega$ followed by densely (without endpoints) ordered blocks of type $\zeta$.

\begin{Theorem}\label{thm-StdCohPow}
Let $\mc{L}$ be a computable copy of $\omega$ that is computably isomorphic to the usual presentation, and let $C$ be cohesive.  Then $\prod_C \mc{L}$ has order-type $\omega + \zeta\eta$.
\end{Theorem}

\begin{proof}
As explained above, it follows from~\cite{Moses}*{Section~3} that a computable copy $\mc{L}$ of $\omega$ is computably isomorphic to the usual presentation if and only if $\mc{L}$ is $1$\nobreakdash-decidable.  Thus we show that if $\mc{L}$ is a $1$\nobreakdash-decidable copy of $\omega$ and $C$ is cohesive, then $\prod_C \mc{L} \iso \omega + \zeta\eta$.

Let $\mc{L}$ be a $1$\nobreakdash-decidable copy of $\omega$, and let $C$ be cohesive.  By Lemma~\ref{lem-StdInitSeg} and Theorem~\ref{thm-CoCeDenseCond}, $\prod_C \mc{L}$ consists of an initial standard block of order-type $\omega$ followed by a densely (without endpoints) ordered collection of non-standard blocks.  It remains to show that each non-standard block has order-type $\zeta$.  To do this, it suffices to show that every non-standard element of $\prod_C \mc{L}$ has an $\prec_{\prod_C \mc{L}}$\nobreakdash-immediate successor and an $\prec_{\prod_C \mc{L}}$\nobreakdash-immediate predecessor.

Let $[\varphi] \in |\prod_C \mc{L}|$.  As $\mc{L} \iso \omega$, we have that
\begin{align*}
(\forae n \in C)\bigl(\mc{L} \models \exists x\, (\text{$x$ is the $\prec_\mc{L}$\nobreakdash-immediate successor of $\varphi(n)$})\bigr).
\end{align*}
The linear order $\mc{L}$ is $1$\nobreakdash-decidable and the relevant formula is $\Sigma_2$, so $[\varphi]$ has a $\prec_{\prod_C \mc{L}}$\nobreakdash-immediate successor in $\prod_C \mc{L}$ by the $n=1$ case of Theorem~\ref{thm-LosGeneral} item~\ref{it-LosDelta2Param}.  Now additionally suppose that $[\varphi]$ is not the $\prec_{\prod_C \mc{L}}$\nobreakdash-least element of $\prod_C \mc{L}$.  Then for almost every $n \in C$, $\varphi(n)$ is not the $\prec_\mc{L}$\nobreakdash-least element of $\mc{L}$.  Therefore
\begin{align*}
(\forae n \in C)\bigl(\mc{L} \models \exists x\, (\text{$x$ is the $\prec_\mc{L}$\nobreakdash-immediate predecessor of $\varphi(n)$})\bigr),
\end{align*}
so $[\varphi]$ has an $\prec_{\prod_C \mc{L}}$\nobreakdash-immediate predecessor in $\prod_C \mc{L}$ by Theorem~\ref{thm-LosGeneral} item~\ref{it-LosDelta2Param}.  It follows that if $[\varphi]$ is non-standard, then it has both an $\prec_{\prod_C \mc{L}}$\nobreakdash-immediate successor and an $\prec_{\prod_C \mc{L}}$\nobreakdash-immediate predecessor in $\prod_C \mc{L}$, which completes the proof.
\end{proof}

We can calculate the order-types of the cohesive powers of many other computable presentations of linear orders by combining Theorems~\ref{thm-IsoCohPow},~\ref{thm-CohPres},~\ref{thm-StdCohPow}, and the fact that $\prod_C \Qb \iso \eta$.
\begin{Example}\label{ex-CohPowCalc}
Let $C$ be a cohesive set.  Let $\Nb$, $\Zb$, and $\Qb$ denote the usual presentations of $\omega$, $\zeta$, and $\eta$.
\begin{enumerate}[(1)]
\item $\prod_C \Nb^* \iso \zeta\eta + \omega^*$:  This is because
\begin{align*}
\prod\nolimits_C \Nb^* \;\iso\; \left(\prod\nolimits_C \Nb\right)^* \;\iso\; (\omega + \zeta\eta)^* \;\iso\; \zeta\eta + \omega^*.
\end{align*}

\bigskip

\item $\prod_C \Zb \iso \zeta\eta$.  This is because $\Zb$ is computably isomorphic to $\Nb^* + \Nb$, so
\begin{align*}
\prod\nolimits_C \Zb \;\iso\; \prod\nolimits_C(\Nb^* + \Nb) \;\iso\; \prod\nolimits_C \Nb^* + \prod\nolimits_C \Nb \;\iso\; (\zeta\eta + \omega^*) + (\omega + \zeta\eta) \;\iso\; \zeta\eta + \zeta + \zeta\eta \;\iso\; \zeta\eta.
\end{align*}

\bigskip

\item\label{it-StdPowZQ} $\prod_C(\Zb\Qb) \iso \zeta\eta$.  This is because
\begin{align*}
\prod\nolimits_C(\Zb\Qb) \;\iso\; \left( \prod\nolimits_C \Zb \right) \left( \prod\nolimits_C\Qb \right) \;\iso\; (\zeta\eta)\eta \;\iso\; \zeta\eta.
\end{align*}

\bigskip

\item\label{it-StdPowNZQ} $\prod_C(\Nb + \Zb\Qb) \iso \omega + \zeta\eta$.  This is because
\begin{align*}
\prod\nolimits_C(\Nb + \Zb\Qb) \;\iso\; \prod\nolimits_C \Nb + \prod\nolimits_C(\Zb\Qb) \;\iso\; (\omega + \zeta\eta) + \zeta\eta \;\iso\; \omega + \zeta\eta.
\end{align*}
\end{enumerate}

Recall that, by Proposition~\ref{prop-UltrahomIso}, an ultrahomogeneous computable structure for a finite relational language, like the computable linear order $\Qb$, is isomorphic to each of its cohesive powers.  Notice, however, that the computable linear orders $\Zb\Qb$ and $\Nb + \Zb\Qb$ are not ultrahomogeneous, yet nevertheless are isomorphic to each of their respective cohesive powers.  Thus it is also possible for a non-ultrahomogeneous computable structure to be isomorphic to each of its cohesive powers.

Notice also that $\prod_C \Nb$ and $\prod_C(\Nb + \Zb\Qb)$ both have order-type $\omega + \zeta\eta$.  Similarly, $\prod_C \Zb$ and $\prod_C(\Zb\Qb)$ both have order-type $\zeta\eta$.  Thus it is possible for non-isomorphic linear orders to have isomorphic cohesive powers.  In Section~\ref{sec-DenseNonstd}, we give an example of a pair of non-elementarily equivalent linear orders with isomorphic cohesive powers.
\end{Example}

Now we give an example of a computable copy of $\omega$ that is not computably isomorphic to the usual presentation, yet still has all its cohesive powers isomorphic to $\omega + \zeta\eta$.

\begin{Theorem}\label{thm-CohPowNZQ}
There is a computable copy $\mc{L}$ of $\omega$ such that
\begin{itemize}
\item $\mc{L}$ is not computably isomorphic to the usual presentation of $\omega$, yet

\smallskip

\item for every cohesive set $C$, the cohesive power $\prod_C \mc{L}$ has order-type $\omega + \zeta\eta$.
\end{itemize}
\end{Theorem}

\begin{proof}
We use a classic example of a computable copy of $\omega$ with a non-computable immediate successor relation.  Fix any non-computable c.e.\ set $A$, and let $f \colon \Nb \imp A$ be a computable bijection.  Let $\mc{L} =(\Nb ,\prec_\mc{L})$ be the linear order obtained by ordering the even numbers according to their usual order and by setting $2a \prec_\mc{L} 2k + 1 \prec_\mc{L} 2a + 2$ if and only if $f(k) = a$. Specifically, define
\begin{align*}
2c &\prec_\mc{L} 2d & &\Biimp & 2c &< 2d \\
2c &\prec_\mc{L} 2k + 1 & &\Biimp & c &\leq f(k) \\
2k + 1 &\prec_\mc{L} 2c & &\Biimp & f(k) &< c \\
2k + 1 &\prec_\mc{L} 2\ell + 1 & &\Biimp & f(k) &< f(\ell).
\end{align*}
Then $\mc{L}$ is a computable linear order of type $\omega$.  Let $S^\mc{L}$ denote the immediate successor relation of $\mc{L}$.  Then $A \leqT S^\mc{L}$ (in fact, $A \equivT S^\mc{L}$) because $a \in A$ if and only if $\neg S^\mc{L}(2a, 2a+2)$.  Thus $S^\mc{L}$ is not computable, so $\mc{L}$ is not computably isomorphic to the usual presentation of $\omega$.

Let $C$ be cohesive.  We show that $\prod_C \mc{L} \iso \omega + \zeta\eta$.  To do this, expand the language to $\{\prec, E, R\}$, where $E$ is a unary relation and $R$ is a binary relation.  Expand $\mc{L}$ by interpreting $E$ as the evens and by interpreting $R$ as the immediate successor relation among the evens:
\begin{align*}
E^\mc{L}(a) \quad&\Biimp\quad \text{$a = 2n$ for some $n$}\\
R^\mc{L}(a, b) \quad&\Biimp\quad \text{$a = 2n$ and $b = 2n+2$ for some $n$}.
\end{align*}
Now consider the substructure $\mc{B}$ of $\mc{L}$ with domain $\{a \in |\mc{L}| : E^\mc{L}(a)\}$ and the substructure $\mc{D}$ of $\prod_C \mc{L}$ with domain $\bigl\{[\varphi] \in |\prod_C \mc{L}| : E^{\prod_C \mc{L}}([\varphi])\bigr\}$.  As a linear order, $\mc{B}$ is computably isomorphic to the usual presentation of $\omega$.  Therefore $\mc{D} \iso \prod_C \mc{B} \iso \omega + \zeta\eta$ as linear orders by Proposition~\ref{prop-substr} and Theorem~\ref{thm-StdCohPow}.

In $\mc{L}$, if $z$ is not in the substructure $\mc{B}$, then $z$ is the unique element $\prec_\mc{L}$\nobreakdash-between two consecutive elements of $\mc{B}$.  That is, $\mc{L}$ satisfies the following $\Pi_2$ sentences.
\begin{align*}
\mc{L} &\models \forall z \, [\neg E(z) \;\imp\; \exists x \exists y \, (R(x,y) \,\andd\, x \prec z \prec y)]\\
\mc{L} &\models \forall x \forall y \forall z_0 \forall z_1 \, [(R(x,y) \,\andd\, x \prec z_0 \prec y \,\andd\, x \prec z_1 \prec y) \;\imp\; z_0 = z_1].
\end{align*}
By Theorem~\ref{thm-LosGeneral} item~\ref{it-LosDelta3Sent}, $\prod_C \mc{L}$ also satisfies these sentences.  Thus in $\prod_C \mc{L}$, if some $[\varphi]$ is not in the substructure $\mc{D}$, then $[\varphi]$ is the unique element $\prec_{\prod_C \mc{L}}$\nobreakdash-between two consecutive elements of $\mc{D}$.  As $\mc{D} \iso \omega + \zeta\eta$, we may conclude that $\prod_C \mc{L} \iso \omega + \zeta\eta$ as well.
\end{proof}

Now we show that for every cohesive set $C$, there is a computable copy $\mc{L}$ of $\omega$ such that $\prod_C \mc{L}$ is not isomorphic, indeed, not elementarily equivalent, to $\omega + \zeta\eta$.  The strategy is to arrange for the element $[\id]$ of $\prod_C \mc{L}$ represented by the identity function $\id \colon \Nb \imp \Nb$ to have no $\prec_{\prod_C \mc{L}}$\nobreakdash-immediate successor.  This exhibits an elementary difference between $\prod_C \mc{L}$ and $\omega + \zeta\eta$ because every element of $\omega + \zeta\eta$ has an immediate successor.  This also shows that Theorem~\ref{thm-LosGeneral} item~\ref{it-LosSigma3Sent} is tight:  ``every element has an immediate successor'' is a $\Pi_3$ sentence that is true of $\mc{L}$ but not of $\prod_C \mc{L}$.

\begin{Theorem}\label{thm-NotNZQ}
Let $C$ be any cohesive set.  Then there is a computable copy $\mc{L}$ of $\omega$ for which $\prod_C \mc{L}$ is not elementarily equivalent (and hence not isomorphic) to $\omega + \zeta\eta$.
\end{Theorem}

\begin{proof}
Let $(\varphi_e)_{e \in \Nb}$ denote the usual effective list of all partial computable functions, and recall that $\varphi_{e,s}(n)$ denotes the result of running $\varphi_e$ on input $n$ for $s$ computational steps.  We compute a linear order $\mc{L} = (\Nb, \prec_\mc{L})$ of type $\omega$ such that for every $\varphi_e$:
\begin{align*}
(\forae n \in C)\bigl[\varphi_e(n)\da \;\Imp\; (\text{$\varphi_e(n)$ is not the $\prec_\mc{L}$\nobreakdash-immediate successor of $n$})\bigr].  \tag{$*$}\label{eq-AvoidId}
\end{align*}

By Lemma~\ref{lem-ImmedSucc}, achieving~\eqref{eq-AvoidId} for $\varphi_e$ ensures that $[\varphi_e]$ is not the $\prec_{\prod_C \mc{L}}$\nobreakdash-immediate successor of $[\id]$ in $\prod_C \mc{L}$.  Therefore, achieving~\eqref{eq-AvoidId} for every $\varphi_e$ ensures that $[\id]$ has no $\prec_{\prod_C \mc{L}}$\nobreakdash-immediate successor in $\prod_C \mc{L}$.  Thus $\prod_C \mc{L}$ is not elementarily equivalent to $\omega + \zeta\eta$ because every element of $\omega + \zeta\eta$ has an immediate successor, which is a $\Pi_3$ property.

Fix an infinite computable set $R \subseteq \ol{C}$.  Such an $R$ may be obtained, for example, by partitioning $\Nb$ into the even numbers $R_0$ and the odd numbers $R_1$.  By cohesiveness, $C \subseteq^* R_i$ for either $i = 0$ or $i = 1$, in which case $R_{1-i} \subseteq^* \ol{C}$.  Thus we may take $R$ to be an appropriate tail of $R_{1-i}$.

Define $\prec_\mc{L}$ in stages.  By the end of stage $s$, $\prec_\mc{L}$ will have been defined on $X_s \times X_s$ for some finite $X_s \supseteq \{0,1,\dots, s\}$.  As the construction progresses, at some stage we may notice an $e$ and $n$ such that $\varphi_e(n) = a$ looks like the $\prec_\mc{L}$\nobreakdash-immediate successor of $n$, where $n$ may or may not be in $C$.  In this case, we want to add an $m$ to the order and set $n \prec_\mc{L} m \prec_\mc{L} a$ to help achieve~\eqref{eq-AvoidId} for $\varphi_e$.  If we choose an $m$ that may be in $C$, then at some later stage there may be an $i$ for which $\varphi_i(m) = a$ looks like the $\prec_\mc{L}$\nobreakdash-immediate successor of $m$, and then we would want to add another element $\prec_\mc{L}$\nobreakdash-below $a$.  If this happens infinitely often, then we would add infinitely many elements $\prec_\mc{L}$\nobreakdash-below $a$, in which case $\mc{L}$ would not be a copy of $\omega$.  We avoid this problem by choosing $m$ from $R$, which is safe because $R \subseteq \ol{C}$.  Since we know that $m \notin C$, we do not need to worry about it when trying to achieve~\eqref{eq-AvoidId}.

At stage $0$, set $X_0 = \{0\}$ and define $0 \nprec_\mc{L} 0$.  At stage $s > 0$, start with $X_s = X_{s-1}$, and update $X_s$ and $\prec_\mc{L}$ according to the following procedure.

\begin{enumerate}[(1)]
\item If $\prec_\mc{L}$ has not yet been defined on $s$ (i.e., if $s \notin X_s$), then update $X_s$ to $X_s \cup \{s\}$ and extend $\prec_\mc{L}$ to make $s$ the $\prec_\mc{L}$\nobreakdash-greatest element of $X_s$.

\medskip

\item\label{it:BreakSucc} Consider each $\la e, n \ra < s$ in order.  For each $\la e,n \ra < s$, if
\smallskip
\begin{enumerate}[(a)]
\item $\varphi_{e,s}(n)\da \in X_s$,

\smallskip

\item $\varphi_e(n)$ is currently the $\prec_\mc{L}$\nobreakdash-immediate successor of $n$ in $X_s$,

\smallskip

\item\label{it:notR} $n \notin R$, and

\smallskip

\item\label{it:restraint} $n$ is not $\preceq_\mc{L}$\nobreakdash-below any of $0, 1, \dots, e$,
\end{enumerate}
\smallskip
\noindent
then let $m$ be the $<$\nobreakdash-least element of $R \setminus X_s$, update $X_s$ to $X_s \cup \{m\}$, and extend $\prec_\mc{L}$ so that $n \prec_\mc{L} m \prec_\mc{L} \varphi_e(n)$.
\end{enumerate}
This completes the construction.

We claim that for every $k$, there are only finitely many elements $\prec_\mc{L}$\nobreakdash-below $k$.  It follows that $\mc{L}$ has order-type $\omega$.  Say that $\varphi_e$ \emph{acts for $n$ and adds $m$} when $\prec_\mc{L}$ is defined on an $m \in R$ to make $n \prec_\mc{L} m \prec_\mc{L} \varphi_e(n)$ as in~\ref{it:BreakSucc}.  Let $s_0$ be a stage with $k \in X_{s_0}$.  Suppose at some stage $s > s_0$, an $m$ is added to $X_s$ and $m \prec_\mc{L} k$ is defined.  This can only be due to a $\varphi_e$ acting for an $n \notin R$ and adding $m$ at stage $s$.  Thus at stage $s$, it must be that $n \prec_\mc{L} k$ because $n \prec_\mc{L} m \prec_\mc{L} k$.  Therefore it must also be that $e < k$, for otherwise $k$ would be among $0,1,\dots, e$, and condition~\ref{it:restraint} would prevent the action of $\varphi_e$.  Furthermore, $m$ is chosen from $R$, so only elements of $R$ are added $\prec_\mc{L}$\nobreakdash-below $k$ after stage $s_0$.  All together, this means that an $m$ can only be added $\prec_\mc{L}$\nobreakdash-below $k$ after stage $s_0$ when a $\varphi_e$ with $e < k$ acts for an $n \prec_\mc{L} k$ with $n \notin R$.  Each $\varphi_e$ acts at most once for each $n$, and no new $n \notin R$ appears $\prec_\mc{L}$\nobreakdash-below $k$ after stage $s_0$.   Thus after stage $s_0$, only finitely many $m$ are ever added $\prec_\mc{L}$\nobreakdash-below $k$.

Finally, we claim that~\eqref{eq-AvoidId} is satisfied for every $\varphi_e$.  Given $e$, let $\ell$ be the $\prec_\mc{L}$\nobreakdash-maximum element of $\{0, 1, \dots, e\}$.  Observe that almost every $n \in \Nb$ satisfies $n \succ_\mc{L} \ell$ because $\mc{L} \iso \omega$.  So suppose that $n \succ_\mc{L} \ell$ and $n \in C$.  If $\varphi_e(n)\da$, let $s$ be large enough so that $\la e, n \ra < s$, $\varphi_{e,s}(n)\da$, $n \in X_s$, and $\varphi_e(n) \in X_s$.  Then either $\varphi_e(n)$ is already not the $\prec_\mc{L}$\nobreakdash-immediate successor of $n$ at stage $s+1$, or at stage $s+1$ the conditions of~\ref{it:BreakSucc} are satisfied for $\la e, n \ra$, and an $m$ is added such that $n \prec_\mc{L} m \prec_\mc{L} \varphi_e(n)$.  This completes the proof.
\end{proof}

\begin{Corollary}\label{cor-LosSigma3Tight}
Theorem~\ref{thm-LosGeneral} item~\ref{it-LosSigma3Sent} cannot be improved in general:  there is a computable linear order $\mc{L}$, a cohesive set $C$, and a $\Pi_3$ sentence $\Phi$ such that $\mc{L} \models \Phi$, but $\prod_C \mc{L} \nmodels \Phi$.
\end{Corollary}

\begin{proof}
Let $C$ be any cohesive set, and let $\mc{L}$ be a computable copy of $\omega$ as in Theorem~\ref{thm-NotNZQ} for $C$.  Let $\Phi$ be a $\Pi_3$ sentence in the language of linear orders expressing that every element has an immediate successor.  Then $\mc{L} \models \Phi$, but $\prod_C \mc{L} \nmodels \Phi$.
\end{proof}

Corollary~\ref{cor-LosSigma3Tight} may also be deduced from Lerman's proof of Feferman, Scott, and Tennenbaum's theorem that no cohesive power of the standard model of arithmetic is a model of Peano arithmetic (see~\cite{LermanCo-r-Max}*{Theorem~2.1}).  Lerman uses Kleene's $T$ predicate to give a somewhat technical example of a $\Pi_3$ sentence that holds in the standard model of arithmetic but fails in every cohesive power.  Our proof of Corollary~\ref{cor-LosSigma3Tight} is more satisfying because it witnesses the optimality of Theorem~\ref{thm-LosGeneral} item~\ref{it-LosSigma3Sent} with a natural $\Pi_3$ sentence in the simple language of linear orders.  In the next section, we enhance the construction of Theorem~\ref{thm-NotNZQ} to compute a copy $\mc{L}$ of $\omega$ with $\prod_C \mc{L} \iso \omega + \eta$ under the additional assumption that the given cohesive set $C$ is co-c.e.

Finally, we show that Theorem~\ref{thm-LosProdParam} concerning cohesive products is also tight by uniformly computing a sequence of finite linear orders $(\mc{L}_n : n \in \Nb)$ such that for every cohesive set $C$, the cohesive product $\prod_C \mc{L}_n$ is a linear order with no maximum element.  Thus the $\Sigma_2$ sentence ``there is a maximum element'' is true in $\mc{L}_n$ for each $n$ (because each $\mc{L}_n$ is finite), whereas the $\Pi_2$ sentence ``there is no maximum element'' is true in $\prod_C \mc{L}_n$ for every cohesive set $C$.  Although each linear order $\mc{L}_n$ has a maximum element, the sequence of maximum elements is not computable.

\begin{Proposition}\label{prop-LosProdParamTight}
There is a uniformly computable sequence of finite linear orders $(\mc{L}_n : n \in \Nb)$ such that for every cohesive set $C$, the cohesive product $\prod_C \mc{L}_n$ is a linear order with no maximum element.  Therefore Theorem~\ref{thm-LosProdParam} cannot be improved in general.
\end{Proposition}

\begin{proof}
For each $n$, let $\mc{L}_n = (L_n, \prec)$ be the linear order with domain
\begin{align*}
L_n = \{0\} \cup \{t : (\exists e < n)(\exists s < t)(\text{$t = s + \varphi_{e,s}(n) + 1$, where $s$ is least such that $\varphi_{e,s}(n)\da$})\}
\end{align*}
that is ordered by the usual order by taking $\prec$ equal to $<$ on $L_n$.  To compute whether a given $t > 0$ is in $L_n$, first run $\varphi_e(n)$ for $t$ steps for each $e < n$.  Then for each such $\varphi_e(n)$ that halts within $t$ steps, find the $s$ such that $\varphi_e(n)$ halts in exactly $s$ steps, and compute the number $s + \varphi_e(n) + 1$.  If any of these numbers is $t$, then $t \in L_n$.  Otherwise $t \notin L_n$.  Notice that each $e < n$ contributes at most one element $t$ to $L_n$, so $\mc{L}_n$ has at most $1 + n$ elements.

Let $C$ be cohesive, and consider an element $[\varphi_e]$ of $\prod_C \mc{L}_n$.  We show that $[\varphi_e]$ is not the $\prec_{\prod_C \mc{L}_n}$\nobreakdash-greatest element of $\prod_C \mc{L}_n$ and therefore that $\prod_C \mc{L}_n$ has no $\prec_{\prod_C \mc{L}_n}$\nobreakdash-greatest element.  If $n > e$ is a sufficiently large member of $C$, then $\varphi_e(n)\da \in L_n$.  This means that there is a least $s$ such that $\varphi_{e,s}(n)\da$ and therefore that there is a $t \in L_n$ with $\varphi_e(n) \prec t$.  Thus $(\forae n \in C)(\mc{L}_n \models \exists x\, (\varphi_e(n) \prec x))$, so $\prod_C \mc{L}_n \models \exists x ([\varphi_e] \prec x)$ by Theorem~\ref{thm-LosProdParam} item~\ref{it-LosProdPramDelta2}.  That is, $[\varphi_e]$ is not the $\prec_{\prod_C \mc{L}_n}$\nobreakdash-greatest element of $\prod_C \mc{L}_n$.
\end{proof}

\section{A computable copy of \texorpdfstring{$\omega$}{omega} with a cohesive power of order-type \texorpdfstring{$\omega + \eta$}{omega + eta}}\label{sec-DenseNonstd}

Given a co-c.e.\ cohesive set, we compute a copy $\mc{L}$ of $\omega$ for which $\prod_C \mc{L}$ has order-type $\omega + \eta$.  In order to help shuffle various linear orders into cohesive powers in Section~\ref{sec-Shuffle}, we in fact compute a linear order $\mc{L} = (L, \prec_\mc{L})$ along with a coloring function $F \colon L \imp \Nb$ that colors the elements of $\mc{L}$ with countably many colors so as to induce a coloring with a certain density property on $\prod_C \mc{L}$.

\begin{Definition}
A \emph{colored linear order} is a structure $\mc{O} = (L, \Nb, \prec_\mc{L}, F)$, where $\mc{L} = (L, \prec_\mc{L})$ is a linear order and $F$ is (the graph of) a function $F \colon L \imp \Nb$, thought of as a coloring of $L$.  Here the language includes unary relation symbols for $L$ and $\Nb$ and a binary relation symbol for $F$ in addition to the binary relation symbol $\prec$.
\end{Definition}

Let $\mc{O} = (L, \Nb, \prec_\mc{L}, F)$ be a colored linear order, and let $\mc{L} = (L, \prec_\mc{L})$.  We may think of $\mc{O}$ as the disjoint union $\mc{L} \sqcup \Nb$ expanded to include $F$.  Thus if $\mc{O}$ is a computable colored linear order and $C$ is a cohesive set, then the cohesive power $\prod_C \mc{O}$ consists of a linear order $\prod_C \mc{L}$, a set $\prod_C \Nb$ thought of as a collection of colors, and a (graph of a) function $F^{\prod_C \mc{O}} \colon |\prod_C \mc{L}| \imp |\prod_C \Nb|$ thought of as a coloring of $\prod_C \mc{L}$.  This is by Proposition~\ref{prop-DU} and by Theorem~\ref{thm-LosGeneral} item~\ref{it-LosDelta3Sent}, as $F$ being the graph of a function from the substructure given by $L$ into the substructure given by $\Nb$ can be expressed by a $\Pi_2$ sentence.  In $\prod_C \mc{O}$, we denote elements of $\prod_C \mc{L}$ by $[\varphi]$ and elements of $\prod_C \Nb$ by $\llb \delta \rrb$.  Call a color $\llb \delta \rrb \in |\prod_C \Nb|$ a \emph{solid color} if $\delta$ is eventually constant on $C$ (i.e., if $\llb \delta \rrb$ is in the range of the canonical embedding of $\Nb$ into $\prod_C \Nb$).  Otherwise, call $\llb \delta \rrb$ a \emph{striped color}.  Finally, call a colored linear order $\mc{O} = (L, \Nb, \prec_\mc{L}, F)$ a \emph{colored copy of $\omega$} if $\mc{L} \iso \omega$.

\begin{Definition}\label{def-colorful}
Let $\mc{O} = (L, \Nb, \prec_\mc{L}, F)$ be a computable colored copy of $\omega$, and let $\mc{L}$ denote $(L, \prec_\mc{L})$.  Let $C$ be a cohesive set.  Call the cohesive power $\prod_C \mc{O}$ \emph{colorful} if the following items hold.
\begin{itemize}
\item For every pair of non-standard elements $[\varphi], [\psi] \in |\prod_C \mc{L}|$ with $[\psi] \prec_{\prod_C \mc{L}} [\varphi]$ and every solid color $\llb \delta \rrb \in |\prod_C \Nb|$, there is a $[\theta] \in |\prod_C \mc{L}|$ with $[\psi] \prec_{\prod_C \mc{L}} [\theta] \prec_{\prod_C \mc{L}} [\varphi]$ and $F^{\prod_C \mc{O}}([\theta]) = \llb \delta \rrb$.

\medskip

\item For every pair of non-standard elements $[\varphi], [\psi] \in |\prod_C \mc{L}|$ with $[\psi] \prec_{\prod_C \mc{L}} [\varphi]$, there is a $[\theta] \in |\prod_C \mc{L}|$ with $[\psi] \prec_{\prod_C \mc{L}} [\theta] \prec_{\prod_C \mc{L}} [\varphi]$ where $F^{\prod_C \mc{O}}([\theta])$ is a striped color.
\end{itemize}
\end{Definition}

Thus if $\mc{O} = (L, \Nb, \prec_\mc{L}, F)$ is a computable colored copy of $\omega$ and $C$ is cohesive, then $\prod_C \mc{O}$ being colorful means that the solid colors occur densely in the non-standard part of $\prod_C \mc{L}$ and also that between any two elements of the non-standard part of $\prod_C \mc{L}$ there is an element with a striped color.  Notice that we do not require any individual striped color to occur densely in the non-standard part of $\prod_C \mc{L}$.  If $C$ is a co-c.e.\ cohesive set, then the first bullet of Definition~\ref{def-colorful} implies the second.  This can be seen by a saturation argument, if one generalizes Lemma~\ref{lem-SatCoCe} to allow types over $\prod_C \mc{O}$ with an infinite sequence of parameters $([\theta_i] : i \in \Nb)$ represented by a uniformly partial computable sequence $(\theta_i : i \in \Nb)$.  Here the relevant parameters would be the non-standard elements $[\varphi]$ and $[\psi]$ and the sequence of solid colors.  The type would then describe an element between $[\varphi]$ and $[\psi]$ whose color is not among the solid colors.

In Section~\ref{sec-Shuffle}, we show that replacing each point of $\mc{L}$ by some finite linear order depending on its color has the effect of shuffling these finite orders into the non-standard part of $\prod_C \mc{L}$.

\begin{Theorem}\label{thm-ColorsDenseNonstd}
Let $C$ be a co-c.e.\ cohesive set.  Then there is a computable colored copy $\mc{O}$ of $\omega$ such that $\prod_C \mc{O}$ is colorful.
\end{Theorem}

\begin{proof}
We construct a computable copy $\mc{L} = (L, \prec_\mc{L})$ of $\omega$ with $L = \Nb$ and a function $F \colon L \imp \Nb$ so that $\mc{O} = (L, \Nb, \prec_\mc{L}, F)$ is a computable colored copy of $\omega$ for which $\prod_C \mc{O}$ is colorful.  We are working with a co-c.e.\ cohesive set, so recall that in this situation every element $[\varphi]$ of $\prod_C \mc{L}$ has a total computable representative by the discussion following Definition~\ref{def-CohProd}.  Recall also that an element $[\varphi]$ of $\prod_C \mc{L}$ is non-standard if and only if $\lim_{n \in C}\varphi(n) = \infty$ by Lemma~\ref{lem-NonstdUnbdd}.

The goal of the construction is to arrange, for every pair of total computable functions $\varphi$ and $\psi$ with $\lim_{n \in C}\varphi(n) = \lim_{n \in C}\psi(n) = \infty$, that
\begin{align*}
(\forae n \in C)\Bigl( &\psi(n)\da \prec_\mc{L} \varphi(n)\da\\
&\Imp\; \bigl( \forall d \leq \max\nolimits_<\{\varphi(n), \psi(n)\} \bigr) \bigl( \exists k \bigr) \bigl[(\psi(n) \prec_\mc{L} k \prec_\mc{L} \varphi(n)) \,\andd\, (F(k) = d) \bigr]\Bigr).\tag{$*$}\label{eq-MakeDense}
\end{align*}
Suppose we achieve~\eqref{eq-MakeDense} for $\varphi$ and $\psi$, where $\lim_{n \in C}\varphi(n) = \lim_{n \in C}\psi(n) = \infty$ and $(\forae n \in C)(\varphi(n)\da \prec_\mc{L} \psi(n)\da)$.  Fix any color $d$, and let $\delta$ be the constant function with value $d$.  Partially compute a function $\theta(n)$ by searching for a $k$ with $\psi(n) \prec_\mc{L} k \prec_\mc{L} \varphi(n)$ and $F(k) = d$.  If there is such a $k$, let $\theta(n)$ be the first such $k$.  Property~\eqref{eq-MakeDense} and the assumption $\lim_{n \in C}\varphi(n) = \lim_{n \in C}\psi(n) = \infty$ ensure that there is such a $k$ for almost every $n \in C$.  Therefore $C \subseteq^* \dom(\theta)$, $[\psi] \prec_{\prod_C \mc{L}} [\theta] \prec_{\prod_C \mc{L}} [\varphi]$, and $F^{\prod_C \mc{O}}([\theta]) = \llb \delta \rrb$.  Likewise, we could instead define $\theta(n)$ to search for a $k$ with $\psi(n) \prec_\mc{L} k \prec_\mc{L} \varphi(n)$ and $F(k) = \varphi(n)$ and let $\theta(n)$ be the first (if any) such $k$ found.  In this case we would have $[\psi] \prec_{\prod_C \mc{L}} [\theta] \prec_{\prod_C \mc{L}} [\varphi]$ and $F^{\prod_C \mc{O}}([\theta]) = \llb \varphi \rrb$, which is a striped color because $\lim_{n \in C}\varphi(n) = \infty$.  Notice that $\varphi$ represents both an element $[\varphi]$ of the linear order $\prod_C \mc{L}$ and a color $\llb \varphi \rrb$ from $\prod_C \Nb$.  As a color, $F^{\prod_C \mc{O}}([\theta]) = \llb \varphi \rrb$ because $(\forae n \in C)(F(\theta(n)) = \varphi(n))$.  Thus between $[\psi]$ and $[\varphi]$ there are elements of $\prod_C \mc{L}$ of every solid color and also at least one element of a striped color.  Therefore achieving~\eqref{eq-MakeDense} suffices to prove the theorem, provided we also arrange $\mc{L} \iso \omega$.

Let $W$ denote the c.e.\ set $\ol{C}$, and let $(W_s)_{s \in \Nb}$ be a computable $\subseteq$\nobreakdash-increasing enumeration of $W$.  Let $(A^{i,0}, A^{i,1})_{i \in \Nb}$ be a uniformly computable sequence of pairs of sets such that
\begin{itemize}
\item for each $i$, $A^{i,0}$ and $A^{i,1}$ partition $\Nb$ into two pieces (i.e., $A^{i,1} = \ol{A^{i,0}}$) and

\smallskip

\item $(\forall n)(\forall \sigma \in \{0,1\}^n)\bigl(\text{$\bigcap_{i < n} A^{i, \sigma(i)}$ is infinite}\bigr)$.
\end{itemize}
This can be accomplished by partitioning $\Nb$ into successive pieces of size $2^i$, letting $A^{i,0}$ consist of every other piece, and letting $A^{i,1} = \ol{A^{i,0}}$.

In this proof, denote the projection functions associated to the pairing function $\la \cdot, \cdot \ra$ by $\ell$ and $r$, for \emph{left} and \emph{right}, instead of by $\pi_0$ and $\pi_1$.  So $\ell(\la x, y \ra) = x$ and $r(\la x, y \ra) = y$.

The tension in the construction is between achieving~\eqref{eq-MakeDense} and ensuring that for every $z$, there are only finitely many $x$ with $x \prec_\mc{L} z$.  Think of a $p \in \Nb$ as coding a pair $(\varphi_{\ell(p)}, \varphi_{r(p)})$ of partial computable functions for which we would like to achieve~\eqref{eq-MakeDense}, with $\varphi_{\ell(p)}$ playing the role of $\psi$ and $\varphi_{r(p)}$ playing the role of $\varphi$.  We assign the partition $(A^{2p,0}, A^{2p,1})$ to $\varphi_{\ell(p)}$ and the partition $(A^{2p+1,0}, A^{2p+1,1})$ to $\varphi_{r(p)}$.  The sets $\{n : \varphi_{\ell(p)}(n) \in A^{2p,0}\}$ and $\{n : \varphi_{\ell(p)}(n) \in A^{2p,1}\}$ are both c.e., so if $C \subseteq^* \dom(\varphi_{\ell(p)})$, then either $(\forae n \in C)(\varphi_{\ell(p)}(n) \in A^{2p,0})$ or $(\forae n \in C)(\varphi_{\ell(p)}(n) \in A^{2p,1})$; and similarly for $\varphi_{r(p)}$ and $(A^{2p+1,0}, A^{2p+1,1})$.  As the construction proceeds, we consider each $p$ paired with larger and larger guesses $N$ of a threshold by which the cohesive behavior of $\varphi_{\ell(p)}$ and $\varphi_{r(p)}$ begins with respect to the partitions $(A^{2p,0}, A^{2p,1})$ and $(A^{2p+1,0}, A^{2p+1,1})$.  The pair $\la p, N \ra$ means we guess that there is an $(a,b) \in \{0,1\} \times \{0,1\}$ such that $\varphi_{\ell(p)}(n) \in A^{2p,a}$ and $\varphi_{r(p)}(n) \in A^{2p+1,b}$ whenever $n \geq N$ and $n \in C$.  For each fixed $p$, the pairs $\la p, 0 \ra, \la p, 1 \ra, \la p, 2 \ra, \dots$ all try to achieve~\eqref{eq-MakeDense} for $\varphi_{\ell(p)}$ and $\varphi_{r(p)}$.  If $N$ is too small, then pair $\la p , N \ra$ eventually stops acting.  If $N$ is big enough, then pair $\la p, N \ra$ eventually settles on the correct sides $(a, b)$ of the partitions $(A^{2p,0}, A^{2p,1})$ and $(A^{2p+1,0}, A^{2p+1,1})$.

To help satisfy~\eqref{eq-MakeDense}, eventually pair $\la p, N \ra$ will want to add an element $k_0$ between some $\varphi_{\ell(p)}(n)$ and $\varphi_{r(p)}(n)$ for an $n$ that looks like it may be in $C$.  However, later some pair $\la q, M \ra$ (possibly even with $q = p$) may want to add an element $k_1$ between some $\varphi_{\ell(q)}(m)$ and $\varphi_{r(q)}(m)$ for an $m$ that looks like it may be in $C$, and it may also so happen that $\varphi_{r(q)}(m) = k_0$.  In this case, $\la q, M \ra$ would add $k_1 \prec_\mc{L} k_0$.  If this behavior were to continue, then it would lead to a descending sequence $k_0 \succ_\mc{L} k_1 \succ_\mc{L} k_2 \succ_\mc{L} \cdots$, which means that $\mc{L}$ would not have order-type $\omega$.  To avoid these descending sequences, pair $\la p, N \ra$ tries to choose the $k$ that it adds to avoid the images $\varphi_{\ell(q)}(C)$ and $\varphi_{r(q)}(C)$ for all $q$ corresponding to higher-or-equal priority pairs $\la q, M \ra \leq \la p, N \ra$.  To do this, first, for each such $q$, pair $\la p, N \ra$ looks up the most recent guess $(a_q, b_q)$ of sides such that $\varphi_{\ell(q)}(C) \subseteq^* A^{2q, a_q}$ and $\varphi_{r(q)}(C) \subseteq^* A^{2q+1, b_q}$ made by any of the $\la q, M \ra \leq \la p, N \ra$ for this $q$.  Then pair $\la p, N \ra$ chooses 
\begin{align*}
k \in \bigcap_{\la q, M \ra \leq \la p, N \ra} A^{2q,1-a_q} \cap A^{2q+1,1-b_q}
\end{align*}
from the opposite sides of all these partitions in an attempt to avoid $\varphi_{\ell(q)}(C)$ and $\varphi_{r(q)}(C)$ for every $\la q, M \ra \leq \la p, N \ra$.  The staggering of the partitions ensures that there are infinitely many such $k$ to choose among.

We now give the construction.  Define $\prec_\mc{L}$ and $F$ in stages.  By the end of stage $s$, $\prec_\mc{L}$ will have been defined on $X_s \times X_s$, and $F$ will have been defined on $X_s$ for some finite $X_s \supseteq \{0, 1, \dots, s\}$.

At stage $0$, set $X_0 = \{0\}$ with $0 \nprec_\mc{L} 0$ and $F(0) = 0$.  At stage $s > 0$, initially set $X_s = X_{s-1}$.  If $s \notin X_s$, then add $s$ to $X_s$, define it to be the $\prec_\mc{L}$\nobreakdash-maximum element of $X_s$, and define $F(s) = 0$.  Then proceed as follows.

Consider each pair $\la p, N \ra < s$ in order.  Think of $\la p, N \ra$ as coding a pair $(\varphi_{\ell(p)}, \varphi_{r(p)})$ of partial computable functions and a guess $N$ of a threshold by which the cohesive behavior of $\varphi_{\ell(p)}$ and $\varphi_{r(p)}$ begins with respect to the partitions $(A^{2p,0}, A^{2p,1})$ and $(A^{2p+1,0}, A^{2p+1,1})$ as described above.  The pair $\la p, N \ra$ \emph{demands action} if there is an $(a, b, n) \in \{0,1\} \times \{0,1\} \times \{N, N+1, \dots, s\}$ meeting the following conditions.
\begin{enumerate}[(1)]
\item\label{it:tot} For all $m \leq n$, $\varphi_{\ell(p),s}(m)\da$ and $\varphi_{r(p),s}(m)\da$.

\medskip

\item\label{it:checkA} Both $\varphi_{\ell(p)}(n) \in A^{2p, a}$ and $\varphi_{r(p)}(n) \in A^{2p+1, b}$.

\medskip

\item\label{it:checkW} For all $m$ with $N \leq m \leq n$,
	\begin{itemize}
	\item $\varphi_{\ell(p)}(m) \in A^{2p, 1-a} \;\Imp\; m \in W_s$, and
	\item $\varphi_{r(p)}(m) \in A^{2p+1, 1-b} \;\Imp\; m \in W_s$.
	\end{itemize}

\medskip

\item\label{it:gap} We have that $\varphi_{\ell(p)}(n), \varphi_{r(p)}(n) \in X_s$ and $\varphi_{\ell(p)}(n) \prec_\mc{L} \varphi_{r(p)}(n)$, but currently there is a $d \leq \max_<\{\varphi_{\ell(p)}(n), \varphi_{r(p)}(n)\}$ for which there is no $k \in X_s$ with $\varphi_{\ell(p)}(n) \prec_\mc{L} k \prec_\mc{L} \varphi_{r(p)}(n)$ and $F(k) = d$.

\medskip

\item\label{it:priority} The element $\varphi_{\ell(p)}(n)$ is not $\preceq_\mc{L}$\nobreakdash-below any of $0, 1, \dots, \la p, N \ra$.
\end{enumerate}

If $\la p, N \ra$ demands action, let $(a_p, b_p, n) \in \{0,1\} \times \{0,1\} \times \{N, N+1, \dots, s\}$ be the lexicographically least witness to this, call $(a_p, b_p, n)$ the \emph{action witness} for $\la p, N \ra$, call the first two coordinates $(a_p, b_p)$ of the action witness the \emph{action sides} for $\la p, N \ra$, and call the last coordinate $n$ of the action witness the \emph{action input} for $\la p, N \ra$.

Let $r$ be the $<$\nobreakdash-greatest number for which there is an $M$ with $\la r, M \ra \leq \la p, N \ra$.  For each $q \leq r$, let $(a_q, b_q)$ be the most recently used action sides by any pair of the form $\la q, M \ra$ with $\la q, M \ra \leq \la p, N \ra$.  If no $\la q, M \ra \leq \la p, N \ra$ has yet demanded action, then let $(a_q, b_q) = (0, 0)$.  Let $c = \max_<\{\varphi_{\ell(p)}(n), \varphi_{r(p)}(n)\}$, and let $k_0 < k_1 < \cdots < k_c$ be the $c+1$ least members of
\begin{align*}\label{fmla:intersect}
\bigcap_{q \leq r}\left(A^{2q, 1-a_q} \cap A^{2q+1, 1-b_q}\right) \setminus X_s, \tag{$\star$}
\end{align*}
which exist because the intersection is infinite and $X_s$ is finite.  Add $k_0, \dots, k_c$ to $X_s$.  Let $x \in X_s$ be the current $\prec_\mc{L}$\nobreakdash-greatest element of the interval $(\varphi_{\ell(p)}(n), \varphi_{r(p)}(n))_\mc{L}$ (or $x = \varphi_{\ell(p)}(n)$ if the interval is empty), and set
\begin{align*}
\varphi_{\ell(p)}(n) \preceq_\mc{L} x \prec_\mc{L} k_0 \prec_\mc{L} \cdots \prec_\mc{L} k_c \prec_\mc{L} \varphi_{r(p)}(n).
\end{align*}
Also set $F(k_i) = i$ for each $i \leq c$, and say that \emph{$\la p, N \ra$ has acted and added $k$'s}.  This completes the construction.

The constructed $\mc{L}$ is a computable linear order.  We show that $\mc{L} \iso \omega$ by showing that for each $z$, there are only finitely many elements $\prec_\mc{L}$\nobreakdash-below $z$.  So fix $z$.  Note that $z$ appears in $X_s$ at stage $s = z$ at the latest, so we consider the development of the construction at stages $s > z$.

Consider the actions of a pair $\la p, N \ra$.  If $\la p, N \ra \geq z$ and $\la p, N \ra$ acts at stage $s > z$ with action input $n$, then, by condition~\ref{it:priority}, it must be that $z \prec_\mc{L} \varphi_{\ell(p)}(n) \prec_\mc{L} \varphi_{r(p)}(n)$.  In this case, the action adds elements to $X_s$ and places them $\prec_\mc{L}$\nobreakdash-between $\varphi_{\ell(p)}(n)$ and $\varphi_{r(p)}(n)$ and hence places them $\prec_\mc{L}$\nobreakdash-above $z$.  Therefore, only the actions of $\la p, N \ra$ with $\la p, N \ra < z$ can add elements $\prec_\mc{L}$\nobreakdash-below $z$ at stages $s > z$.

We show that each $\la p, N \ra < z$ only ever acts to add finitely many elements $k \prec_\mc{L} z$.  It follows that there are only finitely many elements $\prec_\mc{L}$\nobreakdash-below $z$ because the $\la p, N \ra \geq z$ add no elements $\prec_\mc{L}$\nobreakdash-below $z$ after stage $z$, and each $\la p, N \ra < z$ adds only finitely many elements $\prec_\mc{L}$\nobreakdash-below $z$.  So let $\la p, N \ra < z$, and assume inductively that there is a stage $s_0 > z$ such that no pair $\la q, M \ra < \la p, N \ra$ acts to add elements $k \prec_\mc{L} z$ after stage $s_0$.

Notice that a given $n$ can be the action input for $\la p, N \ra$ at most once.  If $\la p, N \ra$ demands action with action input $n$ at stage $s$, it adds elements of every color $\leq \max_<\{\varphi_{\ell(p)}(n), \varphi_{r(p)}(n)\}$ to $X_s$ and places them $\prec_\mc{L}$\nobreakdash-between $\varphi_{\ell(p)}(n)$ and $\varphi_{r(p)}(n)$.  Thus condition~\ref{it:gap} is never again satisfied for $\la p, N \ra$ with action input $n$ at any stage $t > s$.

Suppose that either $\varphi_{\ell(p)}(m)\ua$ or $\varphi_{r(p)}(m)\ua$ for some $m$.  Then no $n \geq m$ can be an action input for $\la p, N \ra$ because condition~\ref{it:tot} always fails when $n \geq m$.  Thus only finitely many numbers $n$ can be action inputs for $\la p, N \ra$.  Because each of these $n$ can be an action input for $\la p, N \ra$ at most once, the pair $\la p, N \ra$ demands action only finitely many times.  Thus in this case, $\la p, N \ra$ adds only finitely many elements $\prec_\mc{L}$\nobreakdash-below $z$.

We now focus on the case in which both $\varphi_{\ell(p)}$ and $\varphi_{r(p)}$ are total.  By cohesiveness, let $(a, b) \in \{0,1\} \times \{0,1\}$ be such that $(\forae n \in C)(\varphi_{\ell(p)}(n) \in A^{2p, a})$ and $(\forae n \in C)(\varphi_{r(p)}(n) \in A^{2p+1, b})$.  The following Claims~\ref{claim-SettleSides}--\ref{claim-AddFin} establish that $\la p, N \ra$ adds only finitely many elements $\prec_\mc{L}$\nobreakdash-below $z$.

First, consider all pairs $\la p, M \ra < z$ with this fixed $p$.  

\begin{ClaimC}\label{claim-SettleSides}
There is a stage $s_1 \geq s_0$ such that for every $M$ with $\la p, M \ra < z$, whenever $\la p, M \ra$ demands action at a stage $s \geq s_1$, it always has action sides $(a,b)$.
\end{ClaimC}

\begin{proof}[Proof of Claim~\ref{claim-SettleSides}]
There are only finitely many $\la p, M \ra < z$, so it suffices to show that for each $\la p, M \ra < z$, there is a stage $t$ such that $\la p, M \ra$ has action sides $(a,b)$ whenever it demands action (if it ever demands action) after stage $t$.

Let $m$ be the least member of $C$ with $m \geq M$, $\varphi_{\ell(p)}(m) \in A^{2p, a}$, and $\varphi_{r(p)}(m) \in A^{2p+1, b}$.  Then whenever $\la p, M \ra$ demands action and the action witness $(a_p, b_p, n)$ has $n \geq m$, it must be that $(a_p, b_p) = (a, b)$ because otherwise condition~\ref{it:checkW} would fail.  Suppose, for example, that $\la p, M \ra$ demands action at stage $s$ with action witness $(a_p, b_p, n)$ where $n \geq m$ and $a_p = 1-a$.  Then $M \leq m \leq n$ and $\varphi_{\ell(p)}(m) \in A^{2p, 1-a_p}$, but $m \notin W_s$ because $m \in C$.  Thus condition~\ref{it:checkW} fails, so $\la p, M \ra$ could not have demanded action with action witness $(a_p, b_p, n)$.  The assumption $b_p = 1-b$ in place of $a_p = 1-a$ leads to the same contradiction.  On the other hand, each $n < m$ can be the action input for $\la p, M \ra$ at most once.  Therefore, there is a stage $t \geq s_0$ such that whenever $\la p, M \ra$ demands action at a later stage $s \geq t$, the action witness must have action input $n \geq m$ and therefore must have action sides $(a, b)$.
\end{proof}

Assume that $\la p, N \ra$ demands action infinitely often because otherwise we can immediately conclude that it adds only finitely many elements $\prec_\mc{L}$\nobreakdash-below $z$.  Let $s_1$ be as in Claim~\ref{claim-SettleSides}, let $t > s_1$ be a stage at which $\la p, N \ra$ demands action, and let $s_2 = t+1$.  Then $\la p, N \ra$ has action sides $(a,b)$ at stage $t < s_2$, and whenever some $\la p, M \ra < z$ demands action at a stage $s \geq s_2 > s_1$, it also has action sides $(a,b)$.  Thus at every stage $s \geq s_2$, the most recently used action sides by a $\la p, M \ra < z$ is always $(a,b)$.

\begin{ClaimC}\label{claim-AddGoodK}
Suppose that an element $k$ is added to $X_s$ and $k \prec_\mc{L} z$ is defined at some stage $s \geq s_2$.  Then $k \in A^{2p, 1-a} \cap A^{2p+1, 1-b}$.
\end{ClaimC}

\begin{proof}[Proof of Claim~\ref{claim-AddGoodK}]
We already know that if $\la q, M \ra \geq z$, then $\la q, M \ra$ does not add elements $k \prec_\mc{L} z$ after stage $s_2$.  Thus we need only consider pairs $\la q, M \ra < z$.  For these pairs, we have assumed inductively that if $\la q, M \ra < \la p, N \ra$, then $\la q, M \ra$ does not add elements $k \prec_\mc{L} z$ after stage $s_2$.  Thus we need only consider pairs $\la q, M \ra$ with $\la p, N \ra \leq \la q, M \ra < z$.  Suppose such a $\la q, M \ra$ acts after stage $s_2$.  When $\la q, M \ra$ chooses the $k$'s to add, it uses an $r \geq p$ in the intersection~\eqref{fmla:intersect} because $\la p, N \ra \leq \la q, M \ra$.  The action of pair $\la q, M \ra$ must use $(a_p, b_p) = (a,b)$.  This is because after stage $s_2$, $(a,b)$ is always the most recently used action sides by the pairs of the form $\la p, K \ra$ with $\la p, K \ra < z$.  Because $\la p, N \ra \leq \la q, M \ra < z$, it is thus also the case that $(a,b)$ is always the most recently used action sides by the pairs of the form $\la p, K \ra \leq \la q, M \ra$ at every stage after $s_2$.  Thus when $\la q, M \ra$ acts at some stage $s \geq s_2$, it uses $(a_p, b_p) = (a,b)$, and therefore the $k$'s it adds to $X_s$ are chosen from $A^{2p, 1-a} \cap A^{2p+1, 1-b}$, as claimed.
\end{proof}

We can now show that $\la p, N \ra$ adds only finitely many elements $k \prec_\mc{L} z$.

\begin{ClaimC}\label{claim-AddFin}
The pair $\la p, N \ra$ adds only finitely many elements $k \prec_\mc{L} z$.
\end{ClaimC}

\begin{proof}[Proof of Claim~\ref{claim-AddFin}]
Suppose that $\la p, N \ra$ acts at some stage $s \geq s_2$, adds an element $k$ to $X_s$, and defines $k \prec_\mc{L} z$.  Then at stage $s$, the action witness for $\la p, N \ra$ must be $(a,b,n)$ for some $n$, where $\varphi_{\ell(p)}(n) = x$ for some $x \in A^{2p, a}$, $\varphi_{r(p)}(n) = y$ for some $y \in A^{2p+1, b}$, and $x \prec_\mc{L} y \preceq_\mc{L} z$.  The action then places $k$'s of each color $d \leq \max_<\{x,y\}$ in the interval $(x, y)_\mc{L}$.  If $\la p, N \ra$ acts again at some later stage $t > s$ with some action input $m$, then again $\varphi_{\ell(p)}(m) \in A^{2p, a}$ and $\varphi_{r(p)}(m) \in A^{2p+1, b}$.  However, it cannot again be that $\varphi_{\ell(p)}(m) = x$ and $\varphi_{r(p)}(m) = y$ because condition~\ref{it:gap} would fail in this situation.  Thus when adding a number $k \prec_\mc{L} z$, the action input $n$ used by $\la p, N \ra$ specifies a pair $(x,y) = (\varphi_{\ell(p)}(n), \varphi_{r(p)}(n)) \in A^{2p, a} \times A^{2p+1, b}$ with $x \prec_\mc{L} y \preceq_\mc{L} z$, and each such pair can be specified by $\la p, N \ra$ at most once.  By Claim~\ref{claim-AddGoodK}, every element added $\prec_\mc{L}$\nobreakdash-below $z$ after stage $s_2$ is in $A^{2p, 1-a} \cap A^{2p+1, 1-b}$.  Therefore there are only finitely many pairs $(x,y) \in A^{2p, a} \times A^{2p+1, b}$ with $x \prec_\mc{L} y \preceq_\mc{L} z$, and therefore $\la p, N \ra$ can only add finitely many elements $k \prec_\mc{L} z$.
\end{proof}

We have shown that for every $z$, no $\la p, N \ra \geq z$ adds an element $\prec_\mc{L}$\nobreakdash-below $z$ after stage $z$ and that each $\la p, N \ra < z$ adds only finitely many elements $\prec_\mc{L}$\nobreakdash-below $z$.  Thus for every $z$, only finitely many elements are ever added $\prec_\mc{L}$\nobreakdash-below $z$.  Therefore $\mc{L} \iso \omega$.

Now let $\varphi$ and $\psi$ be total computable functions with $\lim_{n \in C}\varphi(n) = \lim_{n \in C}\psi(n) = \infty$.  We complete the proof by showing that~\eqref{eq-MakeDense} is satisfied for $\varphi$ and $\psi$.  Assume that $(\forae n \in C)(\psi(n) \prec_\mc{L} \varphi(n))$, for otherwise~\eqref{eq-MakeDense} vacuously holds.  Let $p$ be such that $\varphi_{\ell(p)} = \psi$ and $\varphi_{r(p)} = \varphi$.  By cohesiveness, let $(a,b) \in \{0,1\} \times \{0,1\}$ and $N \in \Nb$ be such that, for all $n \in C$ with $n > N$, $\varphi_{\ell(p)}(n) \in A^{2p, a}$ and $\varphi_{r(p)}(n) \in A^{2p+1, b}$.  Let $n_0 \geq N$ be large enough so that for all $n \in C$ with $n \geq n_0$, $\varphi_{\ell(p)}(n)$ is not $\preceq_\mc{L}$\nobreakdash-below any of $0, 1, \dots, \la p, N \ra$.  To choose $n_0$, notice that the set $Z$ of elements that are $\preceq_\mc{L}$\nobreakdash-below any of $0, 1, \dots, \la p, N \ra$  is finite because $\mc{L} \iso \omega$.  Then $(\forae n \in C)(\varphi_{\ell(p)}(n) \notin Z)$ because $\lim_{n \in C}\varphi_{\ell(p)}(n)  = \infty$.

Suppose that $n \in C$ and $n \geq n_0$, and furthermore suppose for a contradiction that there is a $d < \max_<\{\varphi_{\ell(p)}(n), \varphi_{r(p)}(n)\}$ such that there is no $k$ with $\varphi_{\ell(p)}(n) \prec_\mc{L} k \prec_\mc{L} \varphi_{r(p)}(n)$ and $F(k) = d$.  Then conditions \ref{it:tot}--\ref{it:priority} are satisfied by $(a,b,n)$ at all sufficiently large stages $s$.  Condition~\ref{it:tot} is satisfied because $\varphi_{\ell(p)}$ and $\varphi_{r(p)}$ are total.  Condition~\ref{it:checkA} is satisfied because $n \geq N$ and $n \in C$.  Condition~\ref{it:checkW} is satisfied by the choice of $N$.  Condition~\ref{it:gap} is satisfied by the assumption that there is no $k$ with $\varphi_{\ell(p)}(n) \prec_\mc{L} k \prec_\mc{L} \varphi_{r(p)}(n)$ and $F(k) = d$ and hence there is no such $k$ at every stage $s$ in which both $\varphi_{\ell(p)}(n)$ and $\varphi_{r(p)}(n)$ are present in $X_s$.  Condition~\ref{it:priority} is satisfied by the choice of $n_0$.  Each $m < n$ can be the action input for $\la p, N \ra$ at most once, and, at sufficiently large stages, $(a,b)$ is the only possible action sides for $\la p, N \ra$.  Thus at some stage the pair $\la p, N \ra$ eventually demands action with action witness $(a,b,n)$.  The action of $\la p, N \ra$ defines $\varphi_{\ell(p)}(n) \prec_\mc{L} k \prec_\mc{L} \varphi_{r(p)}(n)$ and $F(k) = d$ for some $k$, which contradicts that there is no such $k$.  This shows that~\eqref{eq-MakeDense} holds for $\varphi = \varphi_{r(p)}$ and $\psi = \varphi_{\ell(p)}$, which completes the proof.
\end{proof}

Let $C$ be a co-c.e.\ cohesive set, and, by Theorem~\ref{thm-ColorsDenseNonstd}, let $\mc{O} =  (L, \Nb, \prec_\mc{L}, F)$ be a computable colored copy of $\omega$ for which $\prod_C \mc{O}$ is colorful.  Then $\mc{L} = (L, \prec_\mc{L})$ is an example of a computable copy of $\omega$ with $\prod_C \mc{L} \iso \omega + \eta$.

\begin{Corollary}\label{cor-DenseNonstd}
Let $C$ be a co-c.e.\ cohesive set.  Then there is a computable copy $\mc{L}$ of $\omega$ where the cohesive power $\prod_C \mc{L}$ has order-type $\omega + \eta$.
\end{Corollary}

\begin{proof}
Let $C$ be co-c.e.\ and cohesive.  Let $\mc{O} =  (L, \Nb, \prec_\mc{L}, F)$ be the computable colored copy of $\omega$ from Theorem~\ref{thm-ColorsDenseNonstd} for $C$.  Let $\mc{L}$ denote the computable copy $\mc{L} = (L, \prec_\mc{L})$ of $\omega$.  The cohesive power $\prod_C \mc{L}$ has an initial segment of order-type $\omega$ by Lemma~\ref{lem-StdInitSeg}.  There is neither a least nor greatest non-standard element of $\prod_C \mc{L}$ by Lemma~\ref{lem-NonStdNoEnd}.  Theorem~\ref{thm-ColorsDenseNonstd} implies that the non-standard elements of $\prod_C \mc{L}$ are dense.  So $\prod_C \mc{L}$ consists of a standard part of order-type $\omega$ and a non-standard part that forms a countable dense linear order without endpoints.  So $\prod_C \mc{L} \iso \omega + \eta$.
\end{proof}

\begin{Example}\label{ex-NonElemEquiv}
Let $C$ be a co-c.e.\ cohesive set, and let $\mc{L}$ be a computable copy of $\omega$ with $\prod_C \mc{L} \iso \omega + \eta$ as in Corollary~\ref{cor-DenseNonstd}.

\begin{enumerate}[(1)]
\item\label{it-kxL} There is a countable collection of computable copies of $\omega$ whose cohesive powers over $C$ are pairwise non-elementarily equivalent.  Let $k \geq 1$, and let $\bm{k}$ denote the $k$\nobreakdash-element linear order $0 < 1 < \cdots < k-1$ as well as its order-type.  Then $\bm{k}\mc{L}$ has order-type $\omega$ because $\mc{L}$ has order-type $\omega$, and $\prod_C \bm{k} \iso \bm{k}$ by the discussion following Definition~\ref{def-CohProd}.  Using Theorem~\ref{thm-CohPres}, we calculate
\begin{align*}
\prod\nolimits_C(\bm{k}\mc{L}) \;\iso\; \bigl( \prod\nolimits_C \bm{k} \bigr) \bigl( \prod\nolimits_C \mc{L} \bigr) \;\iso\; \bm{k}(\omega + \eta) \;\iso\; \omega + \bm{k}\eta.
\end{align*}
The linear orders $\omega + \bm{k}\eta$ for $k \geq 1$ are pairwise non-elementarily equivalent.  The sentence ``there are $x_0 \prec \cdots \prec x_{k-1}$ such that every other $y$ satisfies either $y \prec x_0$ or $x_{k-1} \prec y$; if $y \prec x_0$, then there is a $z$ with $y \prec z \prec x_0$; and if $x_{k-1} \prec y$, then there is a $z$ with $x_{k-1} \prec z \prec y$'' expressing that there is a maximal block of size $k$ is true of $\omega + \bm{k}\eta$, but not of $\omega + \bm{m}\eta$ if $m \neq k$.  Thus $\bm{1}\mc{L}, \bm{2}\mc{L}, \dots$ is a sequence of computable copies of $\omega$ whose cohesive powers $\prod_C (\bm{k}\mc{L})$ are pairwise non-elementarily equivalent.

\medskip

\item It is possible for non-elementarily equivalent computable linear orders to have isomorphic cohesive powers.  Consider the computable linear orders $\mc{L}$ and $\mc{L} + \Qb$.  They are not elementarily equivalent because the sentence ``every element has an immediate successor'' is true of $\mc{L}$ but not of $\mc{L} + \Qb$.  However, using Theorem~\ref{thm-CohPres} and the fact that $\prod_C \Qb \iso \eta$, we calculate
\begin{align*}
\prod\nolimits_C(\mc{L} + \Qb) \;\iso\; \prod\nolimits_C \mc{L} + \prod\nolimits_C \Qb \;\iso\; (\omega + \eta) + \eta \;\iso\; \omega + \eta \;\iso\; \prod\nolimits_C \mc{L}.
\end{align*}
Thus the cohesive powers $\prod_C \mc{L}$ and $\prod_C(\mc{L} + \Qb)$ of $\mc{L}$ and $\mc{L} + \Qb$ are isomorphic.
\end{enumerate}
\end{Example}

\section{Shuffling finite linear orders into cohesive powers of \texorpdfstring{$\omega$}{omega}}\label{sec-Shuffle}

The goal of this section is to prove that if $X \subseteq \Nb \setminus \{0\}$ is a Boolean combination of $\Sigma_2$ sets, thought of as a set of finite order-types, and $C$ is a co-c.e.\ cohesive set, then there is a computable copy $\mc{L}$ of $\omega$ for which $\prod_C \mc{L}$ has order-type $\omega + \bm{\sigma}(X \cup \{\omega + \zeta\eta + \omega^*\})$.  Here $\bm{\sigma}$ denotes the \emph{shuffle} operation (see Definition~\ref{def-Shuffle} below).  We prove this in a modular way by abstracting the cohesive set away from the computable copy of $\omega$ being constructed.  The key technical step is Lemma~\ref{lem-ShuffleBcSig2}, which states that from a computable colored copy $\mc{O}$ of $\omega$ and a Boolean combination of $\Sigma_2$ sets $X \subseteq \Nb \setminus \{0\}$, we can construct a computable copy $\mc{L}$ of $\omega$ such that $\prod_C \mc{L} \;\iso\; \omega + \bm{\sigma}(X \cup \{\omega + \zeta\eta + \omega^*\})$ whenever $C$ is a cohesive set for which $\prod_C \mc{O}$ is colorful.  Combining Lemma~\ref{lem-ShuffleBcSig2} with Theorem~\ref{thm-ColorsDenseNonstd} then gives the desired result.

Given a linear order $\mc{L}$ and a sequence of linear orders $(\mc{M}_\ell : \ell \in |\mc{L}|)$ indexed by $|\mc{L}|$, the \emph{generalized sum} of $(\mc{M}_\ell : \ell \in |\mc{L}|)$ over $\mc{L}$ is obtained by replacing each element $\ell$ of $\mc{L}$ by a copy of $\mc{M}_\ell$.

\begin{Definition}[see~\cite{RosBook}*{Definition~1.38}]\label{def-GenSum}
Let $\mc{L}$ be a linear order, and let $(\mc{M}_\ell : \ell \in |\mc{L}|)$ be a sequence of linear orders indexed by $|\mc{L}|$.  The \emph{generalized sum} $\sum_{\ell \in |\mc{L}|}\mc{M}_\ell$ of $(\mc{M}_\ell : \ell \in |\mc{L}|)$ over $\mc{L}$ is the linear order $\mc{S} = (S, \prec_\mc{S})$ defined as follows.  Write $\mc{L} = (L, \prec_\mc{L})$, and write $\mc{M}_\ell = (M_\ell, \prec_{\mc{M}_\ell})$ for each $\ell \in L$.  Define $S = \{(\ell, m) : \ell \in L \,\andd\, m \in M_\ell\}$, and define
\begin{align*}
(\ell_0, m_0) \prec_{\mc{S}} (\ell_1, m_1) \quad\text{if and only if}\quad (\ell_0 \prec_\mc{L} \ell_1) \;\orr\; (\ell_0 = \ell_1 \,\andd\, m_0 \prec_{\mc{M}_{\ell_0}} m_1).
\end{align*}
\end{Definition}

Let $\mc{S} = \sum_{\ell \in |\mc{L}|}\mc{M}_\ell$ be the generalized sum of a sequence of linear orders $(\mc{M}_\ell : \ell \in |\mc{L}|)$ over a linear order $\mc{L}$ as in Definition~\ref{def-GenSum}.  Each $\mc{M}_\ell$ for $\ell \in |\mc{L}|$ corresponds to an interval of $\mc{S}$, which naturally gives rise to the \emph{sum condensation} of $\mc{S}$.  For $(\ell, m) \in |\mc{S}|$, let $\condSum((\ell, m)) = \{(x, y) \in |\mc{S}| : \ell = x\}$.  The \emph{sum condensation} $\condSum(\mc{S})$ is the condensation obtained from the partition $\{\condSum((\ell, m)) : (\ell, m) \in |\mc{S}|\}$.  Observe that $\condSum((\ell, m)) \iso \mc{M}_\ell$ for each $\ell \in |\mc{L}|$ and that $\condSum(\mc{S}) \iso \mc{L}$.

Generalized sums generalize both the sum and product constructions of Definition~\ref{def-SumProdRev}.  View the ordinary sum $\mc{L}_0 + \mc{L}_1$ as the generalized sum $\sum_{\ell \in |\bm{2}|}\mc{L}_\ell$ of $\mc{L}_0$ and $\mc{L}_1$ over the $2$\nobreakdash-element linear order $\bm{2} = (\{0,1\}, <)$; and view the product $\mc{L}_0 \mc{L}_1$ as the generalized sum $\sum_{\ell \in |\mc{L}_1|}\mc{L}_0$ of copies of $\mc{L}_0$ over $\mc{L}_1$.  We may also use generalized sums to define \emph{shuffles}.

The \emph{shuffle} $\bm{\sigma}(X)$ of an at-most-countable non-empty collection $X$ of linear orders is obtained by densely coloring $\Qb$ with colors from $X$ and then replacing each $q \in \Qb$ by its color.

\begin{Definition}[see~\cite{RosBook}*{Definition~7.14}]\label{def-Shuffle}
Let $X$ be a non-empty collection of linear orders with $|X| \leq \aleph_0$.  Let $f \colon \Qb \imp X$ be a function such that $f^{-1}(\mc{M})$ is dense in $\Qb$ for each linear order $\mc{M} \in X$.  Let $\mc{S} = \sum_{q \in \Qb} f(q)$ be the generalized sum of the sequence $(f(q) : q \in \Qb)$ over $\Qb$.  By density, the order-type of $\mc{S}$ does not depend on the particular choice of $f$.  Therefore $\mc{S}$ is called the \emph{shuffle} of $X$ and is denoted $\bm{\sigma}(X)$.
\end{Definition}

We usually think of $X$ in a shuffle $\bm{\sigma}(X)$ as a collection of order-types instead of as a collection of concrete linear orders.

Let $\mc{L}$ be a computable linear order, and let $(\mc{M}_\ell : \ell \in |\mc{L}|)$ be a uniformly computable sequence of linear orders.  Then one may use the pairing function to compute a copy of $\sum_{\ell \in |\mc{L}|}\mc{M}_\ell$.  Likewise, if $(\mc{M}_n : n \in \Nb)$ is a uniformly computable sequence of linear orders, then one may compute a function $f \colon \Qb \imp \Nb$ such that $f^{-1}(n)$ is dense in $\Qb$ for each $n \in \Nb$ and thereby compute a copy of $\bm{\sigma}(\{\mc{M}_n : n \in \Nb\})$.

Let $C$ be co-c.e.\ and cohesive, let $\mc{L}$ be the linear order from Corollary~\ref{cor-DenseNonstd} for $C$, and consider the linear order $\bm{2}\mc{L}$ from Example~\ref{ex-NonElemEquiv} item~\ref{it-kxL}.  We can think of $\bm{2}\mc{L}$ as being obtained from $\mc{L}$ by replacing each element of $\mc{L}$ by a copy of $\bm{2}$.  This operation of replacing each element by a copy of $\bm{2}$ is reflected in the cohesive power, and we have that $\prod_C (\bm{2}\mc{L}) \iso \omega + \bm{2}\eta$.

Again let $C$ be co-c.e.\ and cohesive, and now consider the computable colored copy $\mc{O} = (R, \Nb, \prec_\mc{R}, F)$ of $\omega$ from Theorem~\ref{thm-ColorsDenseNonstd}.  Let $\mc{R}$ denote $(R, \prec_\mc{R})$.  Collapse $F$ into a coloring $G \colon R \imp \{0,1\}$, where $G(r) = 0$ if $F(r) = 0$ and $G(r) = 1$ if $F(r) \geq 1$.  Then the coloring $G^{\prod_C \mc{O}}$ of $\prod_C \mc{R}$ induced by $G$ uses exactly two colors:  $\llb 0 \rrb$ represented by the constant function with value $0$, and $\llb 1 \rrb$ represented by the constant function with value $1$.  Both of these colors occur densely in the non-standard part of $\prod_C \mc{R}$.  Compute a linear order $\mc{L}$ by starting with $\mc{R}$, replacing each $r \in R$ with $G(r) = 0$ by a copy of $\bm{2}$, and replacing each $r \in R$ with $G(r) = 1$ by a copy of $\bm{3}$.  The cohesive power $\prod_C \mc{L}$ reflects this construction, and we get the linear order obtained from $\prod_C \mc{R}$ by replacing each point of $G^{\prod_C \mc{O}}$\nobreakdash-color $\llb 0 \rrb$ by a copy of $\bm{2}$ and replacing each point of $G^{\prod_C \mc{O}}$\nobreakdash-color $\llb 1 \rrb$ by a copy of $\bm{3}$.  Thus we have a computable copy $\mc{L}$ of $\omega$ with $\prod_C \mc{L} \iso \omega + \bm{\sigma}(\{\bm{2}, \bm{3}\})$.  Using this strategy, we can shuffle any finite collection of finite linear orders into a cohesive power of a computable copy of $\omega$.

To make the above argument precise and to generalize it to more complicated shuffles, we first show that cohesive powers of linear orders respect generalized sums.  Let $\mc{L}$ be a computable linear order, and let $(\mc{M}_\ell : \ell \in |\mc{L}|)$ be a uniformly computable sequence of linear orders indexed by $|\mc{L}|$.  We wish to show that for any cohesive set $C$,
\begin{align*}
\prod\nolimits_C \sum_{\ell \in |\mc{L}|}\mc{M}_\ell \quad\iso\quad \sum_{[\theta] \in \left|\prod_C \mc{L}\right|}\prod\nolimits_C \mc{M}_{\theta(n)}.
\end{align*}
To do this, we must first explain what we mean by the structure $\prod_C \mc{M}_{\theta(n)}$.  Intuitively, $\prod_C \mc{M}_{\theta(n)}$ is the cohesive product of the sequence of structures $\mc{M}_{\theta(0)}, \mc{M}_{\theta(1)}, \mc{M}_{\theta(2)}, \dots$ over $C$, where $\mc{M}_{\theta(n)}$ is undefined if $\theta(n)\ua$.

Formally, let $\mf{L}$ be a computable language, and let $(\mc{A}_n : n \in I)$ be a uniformly computable sequence of $\mf{L}$\nobreakdash-structures indexed by a computable set $I \subseteq \Nb$.  Let $C$ be a cohesive set, and let $\theta \colon \Nb \imp I$ be a partial computable function with $C \subseteq^* \dom(\theta)$.  Then $\prod_C \mc{A}_{\theta(n)}$ is defined as in Definition~\ref{def-CohProd}, except one now considers the $=_C$\nobreakdash-equivalence classes of partial computable functions $\varphi$ such that $\dom(\varphi) \subseteq \dom(\theta)$, $\forall n \, (\varphi(n)\da \,\imp\, \varphi(n) \in |\mc{A}_{\theta(n)}|)$, and $C \subseteq^* \dom(\varphi)$.  The results of Section~\ref{sec-ProdAndPow} hold for these generalized cohesive products of the form $\prod_C \mc{A}_{\theta(n)}$ with minor modifications to the proofs.  For example, one must now consider sets of the form $\bigl\{n : \mc{A}_{\theta(n)} \models \Phi(\varphi_0(n), \dots, \varphi_{m-1}(n))\bigr\}$ for various $\mf{L}$\nobreakdash-formulas $\Phi$ and partial computable functions $\varphi_0, \dots, \varphi_{m-1}$, where $\mc{A}_{\theta(n)}$ appears in place of $\mc{A}_n$.  If $\Phi$ is uniformly decidable in $(\mc{A}_n : n \in I)$, then the preceding set remains c.e.

If $\theta_0, \theta_1 \colon \Nb \imp I$ are two partial computable functions with $C \subseteq^* \dom(\theta_0)$, $C \subseteq^* \dom(\theta_1)$, and $\theta_0 =_C \theta_1$, then it is straightforward to show that $\prod_C \mc{A}_{\theta_0(n)} \iso \prod_C \mc{A}_{\theta_1(n)}$.  In fact, we are even justified in writing $\prod_C \mc{A}_{\theta_0(n)} = \prod_C \mc{A}_{\theta_1(n)}$ because every element of either structure can be represented by a partial computable $\varphi$ with $C \subseteq^* \dom(\varphi) \subseteq \dom(\theta_0) \cap \dom(\theta_1)$.  In particular, the structure $\sum_{[\theta] \in \left|\prod_C \mc{L}\right|}\prod_C \mc{M}_{\theta(n)}$ above is well-defined.

Lastly, we point out that if $C$ is a co-c.e.\ cohesive set with $C \subseteq^* \dom(\theta)$, then the generalized cohesive product $\prod_C \mc{A}_{\theta(n)}$ can be realized as a cohesive product of the form $\prod_C \mc{B}_n$.  The argument is similar to the argument that every element of a cohesive product by a co-c.e.\ cohesive set has a total computable representative.  Fix any computable $\mf{L}$\nobreakdash-structure $\mc{D}$.  Let $N$ be such that $(\forall n > N)(n \in C \imp \theta(n)\da)$.  Define the uniformly computable sequence of $\mf{L}$\nobreakdash-structures $(\mc{B}_n : n \in \Nb)$ by
\begin{align*}
\mc{B}_n = 
\begin{cases}
\mc{A}_{\theta(n)} & \text{if $n > N$ and $\theta(n)\da$ before $n$ is enumerated into $\ol{C}$}\\
\mc{D} & \text{otherwise}.
\end{cases}
\end{align*}
Then $\prod_C \mc{B}_n \iso \prod_C \mc{A}_{\theta(n)}$.  Again, we may even write $\prod_C \mc{B}_n = \prod_C \mc{A}_{\theta(n)}$ because every element of either structure can be represented by a partial computable $\varphi$ with $C \subseteq^* \dom(\varphi) \subseteq \dom(\theta)$.

We are now prepared to show that the cohesive power of a generalized sum is a generalized sum of cohesive products in the way indicated above.  The method of Theorem~\ref{thm-CohPres} becomes unwieldy in this situation because of the infinite sequence of structures to juggle, so we opt for a more hands-on proof.

\begin{Theorem}\label{thm-IsoGenSum}
Let $\mc{L}$ be a computable linear order, and let $(\mc{M}_\ell : \ell \in |\mc{L}|)$ be a uniformly computable sequence of linear orders indexed by $|\mc{L}|$.  Let $C$ be a cohesive set.  Then
\begin{align*}
\prod\nolimits_C \sum_{\ell \in |\mc{L}|}\mc{M}_\ell \quad\iso\quad \sum_{[\theta] \in \left|\prod_C \mc{L}\right|}\prod\nolimits_C \mc{M}_{\theta(n)}.
\end{align*}
\end{Theorem}

\begin{proof}
To ease notation, let
\begin{align*}
\mc{M} &= \sum_{\ell \in |\mc{L}|}\mc{M}_\ell\\ \\
\mc{X} &= \prod\nolimits_C \mc{L}\\ \\
\mc{Y}_{[\theta]_\mc{X}} &= \prod\nolimits_C \mc{M}_{\theta(n)} \quad\text{for each $[\theta]_\mc{X} \in |\mc{X}|$}\\ \\
\mc{A} &= \prod\nolimits_C \mc{M}\\ \\
\mc{B} &= \sum_{[\theta]_\mc{X} \in |\mc{X}|} \mc{Y}_{[\theta]_\mc{X}}.
\end{align*}
The goal is to show that $\mc{A} \iso \mc{B}$.  The elements of $\mc{A}$ are of the form $[\varphi]_\mc{A}$ for partial computable functions $\varphi$ with $\forall n \, (\varphi(n)\da \,\imp\, \varphi(n) \in |\mc{M}|)$ and $C \subseteq^* \dom(\varphi)$.  The elements of $\mc{B}$ are of the form $\bigl([\theta]_\mc{X}, [\tau]_{\mc{Y}_{[\theta]_\mc{X}}}\bigr)$ for partial computable functions $\theta$ and $\tau$ with $\forall n \, (\theta(n)\da \,\imp\, \theta(n) \in |\mc{L}|)$, $C \subseteq^* \dom(\theta)$, $\dom(\tau) \subseteq \dom(\theta)$, $\forall n \, (\tau(n)\da \,\imp\, \tau(n) \in |\mc{M}_{\theta(n)}|)$, and $C \subseteq^* \dom(\tau)$.

Define a function $F \colon |\mc{A}| \imp |\mc{B}|$ as follows.  For $[\varphi]_\mc{A} \in |\mc{A}|$, we have that $\varphi(n) \in |\mc{M}|$ and therefore that $\varphi(n) = \la \ell, m \ra$ for some $\ell \in |\mc{L}|$ and $m \in | \mc{M}_\ell |$ whenever $\varphi(n)\da$.  Let $\theta = \pi_0 \circ \varphi$, and let $\tau = \pi_1 \circ \varphi$.  Then $[\theta]_\mc{X} \in |\mc{X}|$ and $[\tau]_{\mc{Y}_{[\theta]_\mc{X}}} \in \mc{Y}_{[\theta]_\mc{X}}$.  Set $F([\varphi]_\mc{A}) = \bigl([\theta]_\mc{X}, [\tau]_{\mc{Y}_{[\theta]_\mc{X}}}\bigr)$.  To see that $F$ is well-defined, observe that if $\varphi =_C \psi$, then also $\pi_0 \circ \varphi =_C \pi_0 \circ \psi$ and $\pi_1 \circ \varphi =_C \pi_1 \circ \psi$.

To show that $F$ is an isomorphism, it suffices to show that $F$ is surjective and order-preserving by Lemma~\ref{lem-LOIso}.

For surjectivity, consider an element $\bigl([\theta]_\mc{X}, [\tau]_{\mc{Y}_{[\theta]_\mc{X}}}\bigr)$ of $\mc{B}$.  Define a partial computable $\varphi$ by $\varphi(n) \keq \la \theta(n), \tau(n) \ra$.  Then $\varphi(n) \in |\mc{M}|$ whenever $\varphi(n)\da$, and $C \subseteq^* \dom(\varphi)$ because $C \subseteq^* \dom(\theta) \cap \dom(\tau)$.  It follows that $[\varphi]_\mc{A} \in |\mc{A}|$ and $F([\varphi]_\mc{A}) = \bigl([\theta]_\mc{X}, [\tau]_{\mc{Y}_{[\theta]_\mc{X}}}\bigr)$.

For order-preserving, suppose that $[\varphi]_\mc{A}$ and $[\psi]_\mc{A}$ are members of $\mc{A}$ with $[\varphi]_\mc{A} \prec_\mc{A} [\psi]_\mc{A}$.  Then $(\forae n \in C)(\varphi(n) \prec_\mc{M} \psi(n))$.  Write $\theta = \pi_0 \circ \varphi$, $\tau = \pi_1 \circ \varphi$, $\alpha = \pi_0 \circ \psi$, and $\beta = \pi_1 \circ \psi$.  By the definition of $\mc{M}$,
\begin{align*}
(\forae n \in C)\Bigl( \bigl( \theta(n) \prec_\mc{L} \alpha(n) \bigr) \;\orr\; \bigl( \theta(n) = \alpha(n) \,\andd\, \tau(n) \prec_{\mc{M}_{\theta(n)}} \beta(n) \bigr) \Bigr)
\end{align*}
Thus by cohesiveness, either
\begin{itemize}
\item $(\forae n \in C)\bigl( \theta(n) \prec_\mc{L} \alpha(n) \bigr)$ or

\smallskip

\item $(\forae n \in C)\bigl( \theta(n) = \alpha(n) \,\andd\, \tau(n) \prec_{\mc{M}_{\theta(n)}} \beta(n) \bigr)$.
\end{itemize}
In the first case, $[\theta]_\mc{X} \prec_\mc{X} [\alpha]_\mc{X}$.  In the second case, $[\theta]_\mc{X} = [\alpha]_\mc{X}$ and $[\tau]_{\mc{Y}_{[\theta]_\mc{X}}} \prec_{\mc{Y}_{[\theta]_\mc{X}}} [\beta]_{\mc{Y}_{[\theta]_\mc{X}}}$.  Thus in either case,
\begin{align*}
F([\varphi]_\mc{A}) = \bigl( [\theta]_\mc{X}, [\tau]_{\mc{Y}_{[\theta]_\mc{X}}} \bigr) \prec_\mc{B} \bigl( [\alpha]_\mc{X}, [\beta]_{\mc{Y}_{[\alpha]_\mc{X}}} \bigr) = F([\psi]_\mc{A}),
\end{align*}
as desired.
\end{proof}

Notice that Theorem~\ref{thm-CohPres} items~\ref{it-SumIso} and~\ref{it-ProdIso} follow from Theorem~\ref{thm-IsoGenSum} by viewing ordinary sums and products of linear orders as generalized sums of linear orders.

We now show how to shuffle finitely many finite order-types into the cohesive power of a computable copy of $\omega$.

\begin{Lemma}\label{lem-FixedFiniteProd}
Let $(\mc{M}_n : n \in I)$ be a uniformly computable sequence of linear orders indexed by a computable $I \subseteq \Nb$, and let $M_n$ denote the domain of $\mc{M}_n$ for each $n \in I$.  Let $C$ be a cohesive set.  Let $\theta \colon \Nb \imp I$ be a partial computable function with $C \subseteq^* \dom(\theta)$.  Suppose that there is a $k > 0$ such that $(\forae n \in C)(|M_{\theta(n)}| = k)$.  Then $\prod_C \mc{M}_{\theta(n)} \iso \bm{k}$.
\end{Lemma}

\begin{proof}
As explained at the beginning of Section~\ref{sec-ProdAndPow}, the property ``there are exactly $k$ distinct elements'' can be expressed by a $\Delta_2$ sentence $\Phi$.  We have that $C \subseteq^* \{n : \mc{M}_{\theta(n)} \models \Phi\}$ by assumption.  Therefore $\prod_C \mc{M}_{\theta(n)} \models \Phi$ by Theorem~\ref{thm-LosProdParam} item~\ref{it-LosProdPramDelta2}.  Thus $\prod_C \mc{M}_{\theta(n)}$ is a linear order with exactly $k$ elements.  So $\prod_C \mc{M}_{\theta(n)} \iso \bm{k}$.
\end{proof}

\begin{Lemma}\label{lem-ShuffleFinite}
Let $k_0, \dots, k_N$ be non-zero natural numbers, and let $\mc{O}$ be a computable colored copy of $\omega$.  There is a computable copy $\mc{L}$ of $\omega$ (constructed from $\mc{O}$) such that for every cohesive set $C$, if $\prod_C \mc{O}$ is colorful, then $\prod_C \mc{L}$ has order-type $\omega + \bm{\sigma}(\{\bm{k}_0, \dots, \bm{k}_N\})$.
\end{Lemma}

\begin{proof}
Let $\mc{O} = (R, \Nb, \prec_\mc{R}, F)$ be a computable colored copy of $\omega$, and let $\mc{R}$ denote $(R, \prec_\mc{R})$.  For $k > 0$, let $\bm{k} = (\{0, 1, \dots, k-1\}, <)$ denote the usual presentation of the $k$\nobreakdash-element linear order.  Let $(\mc{M}_r : r \in R)$ be the uniformly computable sequence of linear orders where $\mc{M}_r = \bm{k}_{F(r)}$ if $F(r) < N$ and $\mc{M}_r = \bm{k}_N$ if $F(r) \geq N$.  Let $\mc{L}$ be the generalized sum $\mc{L} = \sum_{r \in R} \mc{M}_r$.  The linear order $\mc{L}$ is obtained from the copy $\mc{R}$ of $\omega$ by replacing each element of $\mc{R}$ by a finite linear order.  Thus $\mc{L}$ is infinite, and every element has only finitely many predecessors.  So $\mc{L}$ is a computable copy of $\omega$.

Let $C$ be a cohesive set for which $\prod_C \mc{O}$ is colorful.  We need to show that $\prod_C \mc{L}$ has order-type $\omega + \bm{\sigma}(\{\bm{k}_0, \dots, \bm{k}_N\})$.

By Theorem~\ref{thm-IsoGenSum},
\begin{align*}
\prod\nolimits_C \mc{L} \quad=\quad \prod\nolimits_C \sum_{r \in R}\mc{M}_r \quad\iso\quad \sum_{[\theta] \in \left|\prod_C \mc{R}\right|}\prod\nolimits_C \mc{M}_{\theta(n)}.
\end{align*}
To ease notation, let $\mc{Z}$ denote the linear order $\sum_{[\theta] \in \left|\prod_C \mc{R}\right|}\prod_C \mc{M}_{\theta(n)}$.  Let $|\prod_C \mc{R}|_\std$ and $|\prod_C \mc{R}|_\nonstd$ denote the standard and non-standard parts of $\prod_C \mc{R}$, respectively.  Then let
\begin{align*}
\mc{Z}_\std \quad&=\quad \sum_{[\theta] \in \left|\prod_C \mc{R}\right|_\std}\prod\nolimits_C \mc{M}_{\theta(n)}\\ \\
\mc{Z}_\nonstd \quad&=\quad \sum_{[\theta] \in \left|\prod_C \mc{R}\right|_\nonstd}\prod\nolimits_C \mc{M}_{\theta(n)},
\end{align*}
so that $\mc{Z} \iso \mc{Z}_\std + \mc{Z}_\nonstd$.  Consider the sum condensation $\condSum(\mc{Z})$ of $\mc{Z}$.  We show that the order-type of the block $\prod_C \mc{M}_{\theta(n)}$ of the sum condensation corresponding to $[\theta] \in |\prod_C \mc{R}|$ is determined by the color $F^{\prod_C \mc{O}}([\theta])$ of $[\theta]$ in $\prod_C \mc{O}$.

\begin{ClaimFinShuffle}\label{claim-FinShuffSmall}
If $[\theta] \in |\prod_C \mc{R}|$ and $F^{\prod_C \mc{O}}([\theta])$ is solid color $\llb i \rrb$ for an $i < N$, then $\prod_C \mc{M}_{\theta(n)} \iso \bm{k}_i$.
\end{ClaimFinShuffle}

\begin{proof}[Proof of Claim~\ref{claim-FinShuffSmall}]
That $F^{\prod_C \mc{O}}([\theta]) = \llb i \rrb$ means that $(\forae n \in C)(F(\theta(n)) = i)$.  Therefore $\mc{M}_{\theta(n)} = \bm{k}_{F(\theta(n))} = \bm{k}_i$ for almost every $n \in C$ because $i < N$.  Thus $\prod_C \mc{M}_{\theta(n)} \iso \bm{k}_i$ by Lemma~\ref{lem-FixedFiniteProd}.
\end{proof}

\begin{ClaimFinShuffle}\label{claim-FinShuffBig}
If $[\theta] \in |\prod_C \mc{R}|$ and either $F^{\prod_C \mc{O}}([\theta])$ is solid color $\llb i \rrb$ for an $i \geq N$ or $F^{\prod_C \mc{O}}([\theta])$ is a striped color, then $\prod_C \mc{M}_{\theta(n)} \iso \bm{k}_N$.
\end{ClaimFinShuffle}

\begin{proof}[Proof of Claim~\ref{claim-FinShuffBig}]
If $F^{\prod_C \mc{O}}([\theta])$ is a striped color, then $\lim_{n \in C} F(\theta(n)) = \infty$.  Therefore $(\forae n \in C)(F(\theta(n)) \geq N)$ in both cases, so $\mc{M}_{\theta(n)} = \bm{k}_N$ for almost every $n \in C$.  Thus $\prod_C \mc{M}_{\theta(n)} \iso \bm{k}_N$ by Lemma~\ref{lem-FixedFiniteProd}.
\end{proof}

Notice that block $\prod_C \mc{M}_{\theta(n)}$ is a finite linear order for every $[\theta] \in |\prod_C \mc{R}|$.  Thus $\mc{Z}_\std$ is a generalized sum of finite linear orders over the copy $|\prod_C \mc{R}|_\std$ of $\omega$, so $\mc{Z}_\std \iso \omega$.

Think of the sum condensation $\condSum(\mc{Z}_\nonstd)$ as being colored by $F^{\prod_C \mc{O}}$, where the block $\prod_C \mc{M}_{\theta(n)}$ corresponding to $[\theta] \in |\prod_C \mc{R}|_\nonstd$ gets color $F^{\prod_C \mc{O}}([\theta])$.  The product $\prod_C \mc{O}$ is colorful, which means that $\condSum(\mc{Z}_\nonstd) \iso |\prod_C \mc{R}|_\nonstd \iso \eta$ and that each solid color occurs densely.  By Claims~\ref{claim-FinShuffSmall} and~\ref{claim-FinShuffBig}, the order-type of block $\prod_C \mc{M}_{\theta(n)}$ for $[\theta] \in |\prod_C \mc{R}|_\nonstd$ is:
\begin{itemize}
\item $\bm{k}_i$ if $[\theta]$ has solid color $\llb i \rrb$ for an $i < N$;

\medskip

\item $\bm{k}_N$ if $[\theta]$ has either solid color $\llb i \rrb$ for an $i \geq N$ or a striped color.
\end{itemize}
Therefore $\mc{Z}_\nonstd \;\iso\; \bm{\sigma}(\{\bm{k}_0, \dots, \bm{k}_N\})$.  Thus
\begin{align*}
\prod\nolimits_C \mc{L} \quad\iso\quad \mc{Z} \quad\iso\quad \mc{Z}_\std + \mc{Z}_\nonstd \quad\iso\quad \omega + \bm{\sigma}(\{\bm{k}_0, \dots, \bm{k}_N\})
\end{align*}
as desired.
\end{proof}

To handle more complicated patterns of shuffles, we consider sequences of finite linear orders in which we know the successor relation, we know the least element, we do not necessarily know the greatest element, but we do know that there are at most three elements that the greatest element could be.  Expand the language of linear orders to $\mf{O} = \{\prec, S, B, T_0, T_1, T_2\}$, where $S$ is a binary relation symbol and $B$, $T_0$, $T_1$, and $T_2$ are unary relation symbols.  Our intent is to describe finite linear orders with immediate successor relation $S$, least element given by $B$, and greatest element given by either $T_0$, $T_1$, or $T_2$.  Thus $S$ stands for `successor,' $B$ stands for `bottom,' and $T$ stands for `top.'

Let $\Gamma$ be the set of $\mf{O}$\nobreakdash-sentences consisting of the linear order axioms from the beginning of Section~\ref{sec-LOandPow} along with the following sentences.
\begin{itemize}
\item $\forall x \forall y \, (S(x,y) \;\biimp\; \text{$y$ is the $\prec$\nobreakdash-immediate successor of $x$})$\\
I.e., $S(x,y)$ describes the immediate successor relation.

\medskip

\item $\forall x \, (\exists y \, (x \prec y) \;\imp\; \exists y \, S(x,y))$\\
I.e., every element except the last element has an immediate successor.

\medskip

\item $\forall x \, (\exists y \, (y \prec x) \,\imp\, \exists y \, S(y,x))$\\
I.e., every element except the first element has an immediate predecessor.

\medskip

\item $\exists x \, B(x)$\\
I.e., $B(x)$ holds for some $x$.

\medskip

\item $\forall x \forall y \, (B(x) \,\imp\, x \preceq y)$\\
I.e., if $B(x)$ holds, then $x$ is least.

\medskip

\item $\exists x \, T_0(x)$\\
I.e., $T_0(x)$ holds for some $x$.

\medskip

\item $\exists x \, T_2(x) \;\imp\; \exists x \, T_1(x)$\\
I.e., if $T_2(x)$ holds for some $x$, then $T_1(x)$ holds for some $x$.

\medskip

\item For each $i \leq 2$, $\neg\exists^{\geq 2} x \, T_i(x)$\\
I.e., for each $i \leq 2$, there is at most one $x$ for which $T_i(x)$ holds.

\medskip

\item $\forall x \forall y \, (T_2(x) \,\imp\, y \preceq x)$\\
I.e., if $T_2(x)$ holds for $x$, then $x$ is greatest.

\medskip

\item $\neg\exists x\, T_2(x) \;\imp\; \forall x \forall y \, (T_1(x) \,\imp\, y \preceq x)$\\
I.e., if $T_2(x)$ does not hold for any element but $T_1(x)$ holds for $x$, then $x$ is greatest.

\medskip

\item $\neg\exists x\, T_1(x) \;\imp\; \forall x \forall y \, (T_0(x) \,\imp\, y \preceq x)$\\
I.e., if $T_1(x)$ does not hold for any element but $T_0(x)$ holds for $x$, then $x$ is greatest.
\end{itemize}

Notice that from $\Gamma$ it can be deduced that there is a unique element satisfying $B(x)$ and a unique element satisfying $T_0(x)$.  So we could have used constant symbols in place of $B$ and $T_0$.  We prefer the symmetry of a relational language.  Also notice that from $\Gamma$ it can be deduced that there is a $\prec$\nobreakdash-least element and a $\prec$\nobreakdash-greatest element.  Finally, notice that every sentence in $\Gamma$ is equivalent to a $\Pi_2$ sentence.

When shuffling infinite collections of finite linear orders into a cohesive power of a computable copy of $\omega$, we start with a computable colored copy of $\omega$ and replace its elements by arbitrarily large finite linear orders.  If the finite linear orders can be uniformly computably expanded to models of $\Gamma$, then this replacement process naturally shuffles the linear order $\omega + \zeta\eta + \omega^*$ into the cohesive power.  Lemma~\ref{lem-SuffleZero}, which implies that $\omega + \bm{\sigma}(\{\omega + \zeta\eta + \omega^*\})$ can be achieved as the order-type of a cohesive power of a computable copy of $\omega$, serves as an example explaining this phenomenon.

\begin{Lemma}\label{lem-LimitFiniteProd}
Let $(\mc{M}_n : n \in I)$ be a uniformly computable sequence of $\mf{O}$\nobreakdash-structures that are all finite models of $\Gamma$, indexed by a computable $I \subseteq \Nb$.  Let $M_n$ denote the domain of $\mc{M}_n$ for each $n \in I$.  Let $C$ be a cohesive set.  Let $\theta \colon \Nb \imp I$ be a partial computable function with $C \subseteq^* \dom(\theta)$.  Suppose that $\lim_{n \in C} |M_{\theta(n)}| = \infty$.  Then, as a linear order, $\prod_C \mc{M}_{\theta(n)}$ has order-type $\omega + \zeta\eta + \omega^*$.
\end{Lemma}

\begin{proof}
We have that $\mc{M}_n \models \Gamma$ for each $n \in I$ by assumption and that each sentence of $\Gamma$ is equivalent to a $\Pi_2$ sentence.  Therefore $\prod_C \mc{M}_{\theta(n)} \models \Gamma$ by Theorem~\ref{thm-LosProdParam} item~\ref{it-LosProdPramPi2}.  Thus $\prod_C \mc{M}_{\theta(n)}$ has a $\prec_{\prod_C \mc{M}_{\theta(n)}}$\nobreakdash-least element and a $\prec_{\prod_C \mc{M}_{\theta(n)}}$\nobreakdash-greatest element (as these facts can be deduced from $\Gamma$), every element that is not $\prec_{\prod_C \mc{M}_{\theta(n)}}$\nobreakdash-least has a $\prec_{\prod_C \mc{M}_{\theta(n)}}$\nobreakdash-immediate successor, and every element that is not $\prec_{\prod_C \mc{M}_{\theta(n)}}$\nobreakdash-greatest has a $\prec_{\prod_C \mc{M}_{\theta(n)}}$\nobreakdash-immediate predecessor.  Furthermore, $\prod_C \mc{M}_{\theta(n)}$ is infinite because $\lim_{n \in C} |M_{\theta(n)}| = \infty$.  Thus as a linear order, $\prod_C \mc{M}_{\theta(n)}$ must consist of an initial block of order-type $\omega$, a final block of order-type $\omega^*$, and intermediate blocks of order-type $\zeta$.  We show that the blocks of $\prod_C \mc{M}_{\theta(n)}$ are dense.

If $C$ were co-c.e.\ or the sequence $(\mc{M}_n : n \in I)$ were uniformly $1$\nobreakdash-decidable, then we could use a saturation argument to conclude that the blocks of $\prod_C \mc{M}_{\theta(n)}$ are dense.  However, $C$ is not necessarily co-c.e., and, although each individual structure $\mc{M}_n$ is finite and hence decidable, the sequence $(\mc{M}_n : n \in I)$ need not be uniformly $1$\nobreakdash-decidable.  We therefore resort to an \emph{ad hoc} argument.

Let $[\varphi], [\psi] \in |\prod_C \mc{M}_{\theta(n)}|$ be such that $[\psi] \pprec_{\prod_C \mc{M}_{\theta(n)}} [\varphi]$.  Then $\lim_{n \in C}|(\psi(n), \varphi(n))_{\mc{M}_{\theta(n)}}| = \infty$ by Lemma~\ref{lem-DifferentBlocks}.  If $\theta(n)\da$, $\varphi(n)\da$, $\psi(n)\da$, and $\psi(n) \prec_{\mc{M}_{\theta(n)}} \varphi(n)$, then we can effectively determine the size of the interval $(\psi(n), \varphi(n))_{\mc{M}_{\theta(n)}}$ as follows.  Search for $x_0, \dots, x_{k-1} \in M_{\theta(n)}$ such that $S^{\mc{M}_{\theta(n)}}(\psi(n), x_0)$ holds, such that $S^{\mc{M}_{\theta(n)}}(x_i, x_{i+1})$ holds for each $i < k-1$, and such that $S^{\mc{M}_{\theta(n)}}(x_{k-1}, \varphi(n))$ holds.  Then $|(\psi(n), \varphi(n))_{\mc{M}_{\theta(n)}}| = k$.  Such a sequence $x_0, \dots, x_{k-1}$ exists because $\mc{M}_{\theta(n)}$ is finite and $S^{\mc{M}_{\theta(n)}}$ is the $\prec_{\mc{M}_{\theta(n)}}$\nobreakdash-immediate successor relation on $\mc{M}_{\theta(n)}$.

Define a partial computable function $\rho$ as follows.  Given $n$, if $\theta(n)\da$, $\varphi(n)\da$, $\psi(n)\da$, and $\psi(n) \prec_{\mc{M}_{\theta(n)}} \varphi(n)$, then determine the size $k$ of the interval $(\psi(n), \varphi(n))_{\mc{M}_{\theta(n)}}$ according to the procedure described above.  If $k > 0$, then locate the $\lceil \nicefrac{k}{2} \rceil$\textsuperscript{th} $\prec_{\mc{M}_{\theta(n)}}$\nobreakdash-least element $x$ of $(\psi(n), \varphi(n))_{\mc{M}_{\theta(n)}}$, and output $\rho(n) = x$.  Otherwise $\rho(n)\ua$.  As $\lim_{n \in C}|(\psi(n), \varphi(n))_{\mc{M}_{\theta(n)}}| = \infty$, it follows that both $\lim_{n \in C}|(\psi(n), \rho(n))_{\mc{M}_{\theta(n)}}| = \infty$ and $\lim_{n \in C}|(\rho(n), \varphi(n))_{\mc{M}_{\theta(n)}}| = \infty$.  Thus $[\psi] \pprec_{\prod_C \mc{M}_{\theta(n)}} [\rho] \pprec_{\prod_C \mc{M}_{\theta(n)}} [\varphi]$ again by Lemma~\ref{lem-DifferentBlocks}.  Therefore the blocks of $\prod_C \mc{M}_{\theta(n)}$ are dense.

As a linear order, $\prod_C \mc{M}_{\theta(n)}$ is countably infinite, has a least block of order-type $\omega$, has a greatest block of order-type $\omega^*$, has intermediate blocks of order-type $\zeta$, and the blocks are dense.  Thus $\prod_C \mc{M}_{\theta(n)}$ has order-type $\omega + \zeta\eta + \omega^*$ as a linear order.
\end{proof}

In Lemma~\ref{lem-LimitFiniteProd}, it is necessary that the structures $(\mc{M}_n : n \in I)$ satisfy additional assumptions (such as being models of $\Gamma$) beyond merely being finite linear orders.  For example, recall from Proposition~\ref{prop-LosProdParamTight} that a cohesive product of finite linear orders need not have a maximum element.

\begin{Lemma}\label{lem-SuffleZero}
Let $\mc{O}$ be a computable colored copy of $\omega$.  There is a computable copy $\mc{L}$ of $\omega$ (constructed from $\mc{O}$) such that for every cohesive set $C$, if $\prod_C \mc{O}$ is colorful, then $\prod_C \mc{L}$ has order-type $\omega + (\omega + \zeta\eta + \omega^*)\eta$, which is the same as $\omega + \bm{\sigma}(\{\omega + \zeta\eta + \omega^*\})$.
\end{Lemma}

\begin{proof}
Let $\mc{O} = (R, \Nb, \prec_\mc{R}, F)$ be a computable colored copy of $\omega$, and let $\mc{R}$ denote $(R, \prec_\mc{R})$.  Let $(\mc{M}_r : r \in R)$ be the uniformly computable sequence of linear orders where $\mc{M}_r = \bm{r+1}$ for each $r \in R$.  Let $\mc{L}$ be the generalized sum $\mc{L} = \sum_{r \in R} \mc{M}_r$.  Then $\mc{L}$ is a computable copy of $\omega$.

Let $C$ be a cohesive set for which $\prod_C \mc{O}$ is colorful.  We need to show that $\prod_C \mc{L}$ has order-type $\omega + (\omega + \zeta\eta + \omega^*)\eta$.

By Theorem~\ref{thm-IsoGenSum},
\begin{align*}
\prod\nolimits_C \mc{L} \quad=\quad \prod\nolimits_C \sum_{r \in R}\mc{M}_r \quad\iso\quad \sum_{[\theta] \in \left|\prod_C \mc{R}\right|}\prod\nolimits_C \mc{M}_{\theta(n)}.
\end{align*}
As in the proof of Lemma~\ref{lem-ShuffleFinite}, let $\mc{Z}$ denote $\sum_{[\theta] \in \left|\prod_C \mc{R}\right|}\prod_C \mc{M}_{\theta(n)}$; let $|\prod_C \mc{R}|_\std$ and $|\prod_C \mc{R}|_\nonstd$ denote the standard and non-standard parts of $\prod_C \mc{R}$; and let
\begin{align*}
\mc{Z}_\std \quad&=\quad \sum_{[\theta] \in \left|\prod_C \mc{R}\right|_\std}\prod\nolimits_C \mc{M}_{\theta(n)}\\ \\
\mc{Z}_\nonstd \quad&=\quad \sum_{[\theta] \in \left|\prod_C \mc{R}\right|_\nonstd}\prod\nolimits_C \mc{M}_{\theta(n)},
\end{align*}
so that $\mc{Z} \iso \mc{Z}_\std + \mc{Z}_\nonstd$.  The fact that $\prod_C \mc{O}$ is colorful implies that $|\prod_C \mc{R}|_\nonstd \iso \eta$.

If $[\theta] \in |\prod_C \mc{R}|_\std$, then there is an $r \in R$ such that $(\forae n \in C)(\theta(n) = r)$.  Therefore $(\forae n \in C)(\mc{M}_{\theta(n)} = \bm{r+1})$, so $\prod_C \mc{M}_{\theta(n)} \iso \bm{r+1}$ by Lemma~\ref{lem-FixedFiniteProd}.  This means that $\mc{Z}_\std$ is a generalized sum of finite linear orders over a copy of $\omega$, so $\mc{Z}_\std \iso \omega$.

If $[\theta] \in |\prod_C \mc{R}|_\nonstd$, then $\lim_{n \in C} \theta(n) = \infty$ by Lemma~\ref{lem-NonstdUnbdd}.  Letting $M_r$ denote the domain $\{0, 1, \dots, r\}$ of $\mc{M}_r$ for each $r \in R$, we have that $\lim_{n \in C}|M_{\theta(n)}| = \infty$.  Then $\prod_C \mc{M}_{\theta(n)} \iso \omega + \zeta\eta + \omega^*$ by Lemma~\ref{lem-LimitFiniteProd}.  To see that Lemma~\ref{lem-LimitFiniteProd} applies in this simplified situation, note that we can uniformly computably expand the linear orders $(\mc{M}_r : r \in R)$ to $\mf{O}$\nobreakdash-structures that are models of $\Gamma$.  For $\mc{M}_r = \bm{r+1}$, define $S$ to be the usual immediate successor relation, define $B(x)$ to hold exactly when $x = 0$, define $T_0(x)$ to hold exactly when $x = r$, and define $T_1(x)$ and $T_2(x)$ to hold of no element.

We just showed that $\prod_C \mc{M}_{\theta(n)} \iso \omega + \zeta\eta + \omega^*$ for each $[\theta] \in |\prod_C \mc{R}|_\nonstd$.  Therefore $\mc{Z}_\nonstd \;\iso\; \omega + (\omega + \zeta\eta + \omega^*)\eta \;\iso\; \omega + \bm{\sigma}(\{\omega + \zeta\eta + \omega^*\})$.  Thus
\begin{align*}
\prod\nolimits_C \mc{L} \quad\iso\quad \mc{Z} \quad\iso\quad \mc{Z}_\std + \mc{Z}_\nonstd \quad\iso\quad \omega + (\omega + \zeta\eta + \omega^*)\eta \quad\iso\quad \omega + \bm{\sigma}(\{\omega + \zeta\eta + \omega^*\}),
\end{align*}
as desired.
\end{proof}

We are finally ready to handle shuffles of the form $\bm{\sigma}(X \cup \{\omega + \zeta\eta + \omega^*\})$, where $X \subseteq \Nb \setminus \{0\}$ is a Boolean combination of $\Sigma_2$ sets thought of as a set of finite order-types.  We proceed in two steps.  Lemma~\ref{lem-ShuffleSig2CapPi2} handles the case where $X$ is the intersection of a $\Sigma_2$ set and a $\Pi_2$ set, and Lemma~\ref{lem-ShuffleBcSig2} extends the method to finite unions of such sets.

\begin{Lemma}\label{lem-ShuffleSig2CapPi2}
Let $X \subseteq \Nb \setminus \{0\}$ be the intersection of a $\Sigma_2$ set and a $\Pi_2$ set, thought of as a set of finite order-types.  Let $\mc{O}$ be a computable colored copy of $\omega$.  There is a computable copy $\mc{L}$ of $\omega$ (constructed from $\mc{O}$) such that for every cohesive set $C$, if $\prod_C \mc{O}$ is colorful, then $\prod_C \mc{L}$ has order-type $\omega + \bm{\sigma}(X \cup \{\omega + \zeta\eta + \omega^*\})$.
\end{Lemma}

\begin{proof}
The $X = \emptyset$ case is Lemma~\ref{lem-SuffleZero}, so we may assume that $X \neq \emptyset$.  Let $k_0$ be the least element of $X$.  Let $P$ and $Q$ be computable predicates for which
\begin{align*}
X = \{k : \exists a\, \forall b\, P(k, a, b)\} \cap \{k : \forall a\, \exists b\, Q(k, a, b)\}.
\end{align*}
Let $\mc{O} = (R, \Nb, \prec_\mc{R}, F)$ be a computable colored copy of $\omega$, and let $\mc{R}$ denote $(R, \prec_\mc{R})$.  We define a uniformly computable sequence $(\mc{M}_r : r \in R)$ of $\mf{O}$\nobreakdash-structures that are finite models of $\Gamma$ and have the following properties.  Let $M_r$ denote the domain of $\mc{M}_r$ for each $r \in R$.
\begin{enumerate}[(1)]
\item\label{it-KinX} If $k \in X$, then $(\forae r \in R)(F(r) = k \;\imp\; |M_r| = k)$.

\medskip

\item\label{it-SmallKoutX} If $k < k_0$, then $(\forae r \in R)(F(r) = k \;\imp\; |M_r| = r+1)$.

\medskip

\item\label{it-BigKoutX} If $k > k_0$ and $k \notin X$, then either $(\forae r \in R)(F(r) = k \;\imp\; |M_r| = k_0)$ or $(\forae r \in R)(F(r) = k \;\imp\; |M_r| = r+1)$.
\end{enumerate}
We then take $\mc{L}$ to be the generalized sum $\sum_{r \in R} (\mc{M}_r \rst {\prec})$ of the sequence $(\mc{M}_r : r \in R)$, viewed as a sequence of finite linear orders, over the linear order $\mc{R}$.

To compute $\mc{M}_r = (M_r, \prec, S, B, T_0, T_1, T_2)$, first initialize $\mc{M}_r$ on $\{0, \dots, k_0 - 1\}$ as follows:
\begin{itemize}
\item $\{0, \dots, k_0 - 1\} \subseteq M_r$;

\smallskip

\item $\prec$ agrees with the usual order $<$ on $\{0, \dots, k_0 - 1\}$;

\smallskip

\item $S$ is the usual successor relation on $\{0, \dots, k_0 - 1\}$;

\smallskip

\item $B(x)$ holds if and only if $x = 0$;

\smallskip

\item $T_0(x)$ holds if and only if $x = k_0 - 1$;

\smallskip

\item neither $T_1(x)$ nor $T_2(x)$ hold of any $x \in \{0, \dots, k_0 - 1\}$.
\end{itemize}
If $F(r) = k_0$, or if $F(r) < k_0$ and $r < k_0$, then define $x \notin M_r$ for all $x \geq k_0$.  In this case, $\mc{M}_r$ is the usual presentation of the linear order $\bm{k}_0$ expanded by $S$, $B$, $T_0$, $T_1$, and $T_2$ as described above.  It is straightforward to check that $\mc{M}_r \models \Gamma$.

If $F(r) < k_0$ and $r \geq k_0$, then add $k_0, \dots, r$ to $M_r$ so that $M_r = \{0, \dots, r\}$.  Extend $\prec$ to agree with the usual order on $\{0, \dots, r\}$, and extend $S$ to be the corresponding successor relation.  Define $T_1(x)$ to hold if and only if $x = r$, and define $T_2(x)$ to hold of no $x \in M_r$.  Define $x \notin M_r$ for all $x > r$.  In this case, $\mc{M}_r$ is the usual presentation of the linear order $\bm{r+1}$ expanded by $S$, $B$, $T_0$, $T_1$, and $T_2$, where $T_0(x)$ holds only of $k_0 - 1$, $T_1(x)$ holds only of $r$, and $T_2(x)$ holds of no element.  It is again straightforward to check that $\mc{M}_r \models \Gamma$.

If $F(r) > k_0$, then compute $\mc{M}_r$ in stages.  At all stages $s$, we maintain that $\mc{M}_r \models \Gamma$ and that $\prec$ agrees with the usual order $<$ on the elements of $\mc{M}_r$.  For every $x$, we decide whether or not $x \in M_r$ at stage $x$ at the latest.  The initialization of $\mc{M}_r$ on $\{0, \dots, k_0 - 1\}$ described above counts as stage $0$, so $\mc{M}_r \models \Gamma$ at the end of stage $0$.  Proceed as follows at stage $s > 0$.  If it has not yet been decided whether $s \in M_r$ by the beginning of stage $s$, then define $s \notin M_r$.  Then act according to the following cases.

\begin{itemize}
\item[Case~1:] $M_r$ is still $\{0, \dots, k_0 - 1\}$ at the beginning of stage $s$, $(\exists a < r)(\forall b < s)\, P(F(r), a, b)$ holds, and $(\forall a < r)(\exists b < s)\, Q(F(r), a, b)$ holds.  In this case, let $s < x_0 < x_1 < \dots < x_{F(r) - k_0 - 1}$ be the $F(r) - k_0$ least numbers $x > s$ for which it has not yet been decided whether $x \in M_r$.  Add $x_0, \dots, x_{F(r) - k_0 - 1}$ to $M_r$ so that $M_r = \{0, \dots, k_0 - 1, x_0, \dots, x_{F(r) - k_0 - 1}\}$.  Extend $\prec$ to agree with the usual order on $\{0, \dots, k_0 - 1, x_0, \dots, x_{F(r) - k_0 - 1}\}$, and extend $S$ to be the corresponding successor relation.  So $S(k_0 - 1, x_0)$ holds, and $S(x_i, x_{i+1})$ holds for all $i < F(r) - k_0 - 1$.  Finally, define $T_1(x)$ to hold if and only if $x =  x_{F(r) - k_0 - 1}$, and define $T_2(x)$ to hold of no $x \in M_r$.  Go on to stage $s+1$.  Observe that $|M_r| = F(r)$ at the end of stage $s$.

To see that $\mc{M}_r \models \Gamma$ at the end of stage $s$, observe that no elements were added to $\mc{M}_r$ between its initialization and the start of stage $s$.  Thus at the start of stage $s$, $T_0(x)$ holds of one element, and $T_1(x)$ and $T_2(x)$ hold of no element.  During stage $s$, a new greatest element $x_{F(r) - k_0 - 1}$ is added, and $T_1(x)$ is defined to hold of exactly this element.  Also, $T_2(x)$ still holds of no element at the end of stage $s$.  One may now check that $\mc{M}_r \models \Gamma$.

\medskip

\item[Case~2:] $(\forall a < r)(\exists b < s)\, \neg P(F(r), a, b)$ holds and $|M_r| \leq r$ at the beginning of stage $s$.  In this case, let $m = |M_r|$, and let $\ell_0 \prec \ell_1 \prec \cdots \prec \ell_{m-1}$ be the elements of $M_r$ listed in $\prec$\nobreakdash-increasing order.  Recall that $\prec$ and $<$ agree on $M_r$, so also $\ell_0 < \ell_1 < \cdots < \ell_{m-1}$.  Let $s < x_0 < x_1 < \cdots < x_{r-m}$ be the $r+1 - m$ least numbers $x > s$ for which it has not yet been decided whether $x \in M_r$.  Notice that $x_0 > \ell_{m-1}$ as well, as otherwise $\ell_{m-1}$ would not have been least when it was added to $M_r$.  Add $x_0, \dots, x_{r-m}$ to $M_r$ so that $M_r = \{\ell_0, \dots, \ell_{m-1}, x_0, \dots, x_{r-m}\}$.  Extend $\prec$ to agree with the usual order on $\{\ell_0, \dots, \ell_{m-1}, x_0, \dots, x_{r-m}\}$, and extend $S$ to be the corresponding successor relation.  So $S(\ell_{m-1}, x_0)$ holds, and $S(x_i, x_{i+1})$ holds for all $i < r-m$.  If there is no $x$ for which $T_1(x)$ holds, then define $T_1(x)$ to hold if and only if $x =  x_{r-m}$ and define $T_2(x)$ to hold of no $x \in M_r$.  If there is already an $x$ for which $T_1(x)$ holds, then define $T_2(x)$ to hold if and only if $x =  x_{r-m}$.  Go on to stage $s+1$.  Observe that $|M_r| = r+1$ at the end of stage $s$.

Notice that Case~1 and Case~2 can occur at most one time each.  After Case~1 occurs, $M_r$ is never again $\{0, \dots, k_0 - 1\}$, so Case~1 never occurs again.  After Case~2 occurs, we never again have $|M_r| \leq r$, so Case~2 never occurs again.  Thus prior to stage $s$, Case~2 cannot have occurred (as it is occurring now at stage $s$), and Case~1 can have occurred at most once.  If Case~1 did not occur before stage $s$, then no elements were added to $\mc{M}_r$ between its initialization and the start of stage $s$.  The situation is then analogous to that of Case~1.  We define $T_1(x)$ to hold of exactly the new greatest element $x_{r-m}$ that is added at stage $s$, and we define $T_2(x)$ to hold of no element.  If instead Case~1 did occur before stage $s$, then at the start of stage $s$, $T_0(x)$ and $T_1(x)$ hold of exactly one element each, and $T_2(x)$ holds of no element.  During stage $s$, a new greatest element $x_{r-m}$ is added, and $T_2(x)$ is defined to hold of exactly this element.  In either situation, one may check that $\mc{M}_r \models \Gamma$.

\medskip

\item[Case~3:]  If neither Case~1 nor Case~2 applies, then do nothing more at stage $s$ and go on to stage $s+1$.  Then $\mc{M}_r \models \Gamma$ at the end of stage $s$ because $\mc{M}_r \models \Gamma$ at the start of stage $s$.
\end{itemize}

This concludes the construction of $(\mc{M}_r : r \in R)$.  In the computation of $\mc{M}_r$ for a given $r \in R$, Case~1 and Case~2 can occur at most once each, as observed in the discussion of Case~2 above.  Therefore elements are added to $\mc{M}_r$ at most twice, so it is finite.  It also follows that $\mc{M}_r \models \Gamma$ at the end of the construction because $\mc{M}_r \models \Gamma$ at every stage.

We show that the above items~\ref{it-KinX}--\ref{it-BigKoutX} hold.

For item~\ref{it-KinX}, consider a $k \in X$.  If $k = k_0$, then $|M_r| = k$ whenever $r \in R$ and $F(r) = k$.  Suppose instead that $k > k_0$.  Both $\exists a\, \forall b\, P(k, a, b)$ and $\forall a\, \exists b\, Q(k, a, b)$ hold because $k \in X$.  Suppose that $r \in R$ has $F(r) = k$ and is large enough so that $(\exists a < r)(\forall b)\, P(k, a, b)$.  Then Case~2 never occurs in the computation of $\mc{M}_r$ because $(\forall a < r)(\exists b < s)\, \neg P(F(r), a, b)$ always fails.  On the other hand, Case~1 occurs at the first stage $s$ such that $(\forall a < r)(\exists b < s)\, Q(F(r), a, b)$.  Therefore $|M_r| = k$.  This shows that $(\forae r \in R)(F(r) = k \;\imp\; |M_r| = k)$.

For item~\ref{it-SmallKoutX}, consider a $k < k_0$.  Then $|M_r| = r+1$ whenever $r \in R$, $r \geq k_0$, and $F(r) = k$.  Thus $(\forae r \in R)(F(r) = k \;\imp\; |M_r| = r+1)$.

For item~\ref{it-BigKoutX}, consider a $k > k_0$ with $k \notin X$.  Then either $\forall a\, \exists b\, \neg P(k, a, b)$ or $\exists a\, \forall b\, \neg Q(k, a, b)$.  First suppose that $\forall a\, \exists b\, \neg P(k, a, b)$ holds.  Let $r \in R$ have $F(r) = k$ and $r \geq k$.  Let $s$ be the first stage at which $(\forall a < r)(\exists b < s)\, \neg P(F(r), a, b)$ in the computation of $\mc{M}_r$.  Then $|M_r| \leq r$ at the beginning of stage $s$.  This is because Case~2 cannot have occurred before stage $s$ by the choice of $s$, so either $|M_r| = k_0 < r$ (if Case~1 has not occurred by stage $s$) or $|M_r| = k \leq r$ (if Case~1 has occurred by stage $s$).  Thus Case~2 occurs at stage $s$, so $|M_r| = r+1$ at the end of stage $s$.  Neither Case~1 nor Case~2 occurs after stage $s$ because Case~2 occurs at most once and Case~1 cannot occur after Case~2.  Thus $|M_r| = r+1$ at the end of the construction.  This shows that $(\forae r \in R)(F(r) = k \;\imp\; |M_r| = r+1)$.

Finally, suppose that $\exists a\, \forall b\, \neg Q(k, a, b)$ holds but $\forall a\, \exists b\, \neg P(k, a, b)$ fails.  Thus both $\exists a\, \forall b\, P(k, a, b)$ and $\exists a\, \forall b\, \neg Q(k, a, b)$ hold.  Suppose that $r \in R$ has $F(r) = k$ and is large enough so that $(\exists a < r)(\forall b)\, P(F(r), a, b)$ and $(\exists a < r)(\forall b)\, \neg Q(F(r), a, b)$.  Then neither Case~1 nor Case~2 occur at any stage in the computation of $\mc{M}_r$.  Thus $|M_r| = k_0$.  This shows that $(\forae r \in R)(F(r) = k \;\imp\; |M_r| = k_0)$ and completes the argument that item~\ref{it-BigKoutX} holds.

Let $\mc{L}$ be the generalized sum $\sum_{r \in R} (\mc{M}_r \rst {\prec})$ of the sequence $(\mc{M}_r : r \in R)$, viewed as a sequence of finite linear orders, over the linear order $\mc{R}$ as indicated above.  Then $\mc{L}$ is a computable copy of $\omega$.  Let $C$ be a cohesive set for which $\prod_C \mc{O}$ is colorful.  We need to show that $\prod_C \mc{L}$ has order-type $\omega + \bm{\sigma}(X \cup \{\omega + \zeta\eta + \omega^*\})$.

Theorem~\ref{thm-IsoGenSum} gives us that, as linear orders,
\begin{align*}
\prod\nolimits_C \mc{L} \quad=\quad \prod\nolimits_C \sum_{r \in R}\mc{M}_r \quad\iso\quad \sum_{[\theta] \in \left|\prod_C \mc{R}\right|}\prod\nolimits_C \mc{M}_{\theta(n)}.
\end{align*}
As in the proofs of Lemmas~\ref{lem-ShuffleFinite} and~\ref{lem-SuffleZero}, let $\mc{Z}$ denote the linear order $\sum_{[\theta] \in \left|\prod_C \mc{R}\right|}\prod_C \mc{M}_{\theta(n)}$; let $|\prod_C \mc{R}|_\std$ and $|\prod_C \mc{R}|_\nonstd$ denote the standard and non-standard parts of $\prod_C \mc{R}$; and let
\begin{align*}
\mc{Z}_\std \quad&=\quad \sum_{[\theta] \in \left|\prod_C \mc{R}\right|_\std}\prod\nolimits_C \mc{M}_{\theta(n)}\\ \\
\mc{Z}_\nonstd \quad&=\quad \sum_{[\theta] \in \left|\prod_C \mc{R}\right|_\nonstd}\prod\nolimits_C \mc{M}_{\theta(n)},
\end{align*}
so that $\mc{Z} \iso \mc{Z}_\std + \mc{Z}_\nonstd$.  We show that the order-type of the block $\prod_C \mc{M}_{\theta(n)}$ of $\condSum(\mc{Z}_\nonstd)$ corresponding to $[\theta] \in |\prod_C \mc{R}|_\nonstd$ is determined by the color $F^{\prod_C \mc{O}}([\theta])$ of $[\theta]$ in $\prod_C \mc{O}$.

\begin{ClaimInfShuffle}\label{claim-SolidInX}
If $[\theta] \in |\prod_C \mc{R}|_\nonstd$ and $F^{\prod_C \mc{O}}([\theta])$ is solid color $\llb k \rrb$ for a $k \in X$, then $\prod_C \mc{M}_{\theta(n)} \iso \bm{k}$ as a linear order.
\end{ClaimInfShuffle}

\begin{proof}[Proof of Claim~\ref{claim-SolidInX}]
We have that $\lim_{n \in C} \theta(n) = \infty$ by Lemma~\ref{lem-NonstdUnbdd} and that $(\forae n \in C)(F(\theta(n)) = k)$ by definition.  Furthermore, $k \in X$ implies that $(\forae r \in R)(F(r) = k \;\imp\; |M_r| = k)$ by item~\ref{it-KinX}.  Therefore $(\forae n \in C)(|M_{\theta(n)}| = k)$.  Thus $\prod_C \mc{M}_{\theta(n)} \iso \bm{k}$ by Lemma~\ref{lem-FixedFiniteProd}.
\end{proof}

\begin{ClaimInfShuffle}\label{claim-SolidSmall}
If $[\theta] \in |\prod_C \mc{R}|_\nonstd$ and $F^{\prod_C \mc{O}}([\theta])$ is solid color $\llb k \rrb$ with $k < k_0$, then $\prod_C \mc{M}_{\theta(n)} \iso \omega + \zeta\eta + \omega^*$ as a linear order.
\end{ClaimInfShuffle}

\begin{proof}[Proof of Claim~\ref{claim-SolidSmall}]
Again, $\lim_{n \in C} \theta(n) = \infty$ by Lemma~\ref{lem-NonstdUnbdd} and $(\forae n \in C)(F(\theta(n)) = k)$ by definition.  Furthermore, $k < k_0$ implies that $(\forae r \in R)(F(r) = k \;\imp\; |M_r| = r+1)$ by item~\ref{it-SmallKoutX}.  Therefore $\lim_{n \in C}|M_{\theta(n)}| = \infty$, so $\prod_C \mc{M}_{\theta(n)} \iso \omega + \zeta\eta + \omega^*$ by Lemma~\ref{lem-LimitFiniteProd}.
\end{proof}

\begin{ClaimInfShuffle}\label{claim-SolidOutX}
If $[\theta] \in |\prod_C \mc{R}|_\nonstd$ and $F^{\prod_C \mc{O}}([\theta])$ is solid color $\llb k \rrb$ for a $k > k_0$ with $k \notin X$, then, as a linear order, $\prod_C \mc{M}_{\theta(n)}$ either has type $\bm{k}_0$ or type $\omega + \zeta\eta + \omega^*$.
\end{ClaimInfShuffle}

\begin{proof}[Proof of Claim~\ref{claim-SolidOutX}]
Again, $\lim_{n \in C} \theta(n) = \infty$ by Lemma~\ref{lem-NonstdUnbdd} and $(\forae n \in C)(F(\theta(n)) = k)$ by definition.  As $k > k_0$ and $k \notin X$, either $(\forae r \in R)(F(r) = k \;\imp\; |M_r| = k_0)$ or $(\forae r \in R)(F(r) = k \;\imp\; |M_r| = r+1)$ by item~\ref{it-BigKoutX}.  The first alternative yields that $(\forae n \in C)(|M_{\theta(n)}| = k_0)$ and hence that $\prod_C \mc{M}_{\theta(n)} \iso \bm{k}_0$ by Lemma~\ref{lem-FixedFiniteProd}.  The second alternative yields that $\lim_{n \in C}|M_{\theta(n)}| = \infty$ and hence that $\prod_C \mc{M}_{\theta(n)} \iso \omega + \zeta\eta + \omega^*$ by Lemma~\ref{lem-LimitFiniteProd}.
\end{proof}

\begin{ClaimInfShuffle}\label{claim-Striped}
If $[\theta] \in |\prod_C \mc{R}|_\nonstd$ and $F^{\prod_C \mc{O}}([\theta])$ is a striped color, then, as a linear order, $\prod_C \mc{M}_{\theta(n)}$ has either type $\bm{k}_0$ or type $\omega + \zeta\eta + \omega^*$.
\end{ClaimInfShuffle}

\begin{proof}[Proof of Claim~\ref{claim-Striped}]
We have that $\lim_{n \in C} \theta(n) = \infty$ by Lemma~\ref{lem-NonstdUnbdd} and that $\lim_{n \in C} F(\theta(n)) = \infty$ because $F^{\prod_C \mc{O}}([\theta])$ is a striped color.  By inspecting the construction, we see that for $r \in R$, either $|M_r| = k_0$, $|M_r| = F(r)$, or $|M_r| = r+1$.  By cohesiveness, either $(\forae n \in C)(|M_{\theta(n)}| \leq k_0)$ or $(\forae n \in C)(|M_{\theta(n)}| > k_0)$.  If $(\forae n \in C)(|M_{\theta(n)}| \leq k_0)$, then in fact $(\forae n \in C)(|M_{\theta(n)}| = k_0)$ because $\lim_{n \in C} F(\theta(n)) = \infty$.  In this case, we have that $\prod_C \mc{M}_{\theta(n)} \iso \bm{k}_0$ by Lemma~\ref{lem-FixedFiniteProd}.  If instead $(\forae n \in C)(|M_{\theta(n)}| > k_0)$, then it must be that $\lim_{n \in C}|M_{\theta(n)}| = \infty$.  This is because $|M_{\theta(n)}|$ is either $F(\theta(n))$ or $\theta(n) + 1$ for almost every $n \in C$, and both $\lim_{n \in C} F(\theta(n)) = \infty$ and $\lim_{n \in C} \theta(n) = \infty$.  Therefore $\prod_C \mc{M}_{\theta(n)} \iso \omega + \zeta\eta + \omega^*$ by Lemma~\ref{lem-LimitFiniteProd}.
\end{proof}

If $[\theta] \in |\prod_C \mc{R}|_\std$, then there is an $r \in R$ such that $(\forae n \in C)(\theta(n) = r)$.  Therefore there is a $k > 0$ such that $(\forae n \in C)(|M_{\theta(n)}| = k)$, in which case $\prod_C \mc{M}_{\theta(n)} \iso \bm{k}$ by Lemma~\ref{lem-FixedFiniteProd}.  This means that $\mc{Z}_\std$ is a generalized sum of finite linear orders over the copy $|\prod_C \mc{R}|_\std$ of $\omega$, so $\mc{Z}_\std \iso \omega$.

Think of the sum condensation $\condSum(\mc{Z}_\nonstd)$ as being colored by $F^{\prod_C \mc{O}}$, where the block $\prod_C \mc{M}_{\theta(n)}$ corresponding to $[\theta] \in |\prod_C \mc{R}|_\nonstd$ gets color $F^{\prod_C \mc{O}}([\theta])$.  The product $\prod_C \mc{O}$ is colorful, which means that $\condSum(\mc{Z}_\nonstd) \iso |\prod_C \mc{R}|_\nonstd \iso \eta$ and that each solid color occurs densely.  By Claims~\ref{claim-SolidInX}--\ref{claim-Striped} the order-type of block $\prod_C \mc{M}_{\theta(n)}$ for $[\theta] \in |\prod_C \mc{R}|_\nonstd$ is:
\begin{itemize}
\item $\bm{k}$ if $[\theta]$ has solid color $\llb k \rrb$ with $k \in X$;

\medskip

\item $\omega + \zeta\eta + \omega^*$ if $[\theta]$ has solid color $\llb k \rrb$ with $k < k_0$ (which includes $k = 0$ because $k_0 > 0$);

\medskip

\item either $\bm{k}_0$ or $\omega + \zeta\eta + \omega^*$ if $[\theta]$ has solid color $\llb k \rrb$ with $k > k_0$ and $k \notin X$;

\medskip

\item either $\bm{k}_0$ or $\omega + \zeta\eta + \omega^*$ if $[\theta]$ has a striped color.
\end{itemize}
Therefore $\mc{Z}_\nonstd \;\iso\; \bm{\sigma}(X \cup \{\omega + \zeta\eta + \omega^*\})$.  Thus
\begin{align*}
\prod\nolimits_C \mc{L} \quad\iso\quad \mc{Z} \quad\iso\quad \mc{Z}_\std + \mc{Z}_\nonstd \quad\iso\quad \omega + \bm{\sigma}(X \cup \{\omega + \zeta\eta + \omega^*\})
\end{align*}
as desired.
\end{proof}

\begin{Lemma}\label{lem-ShuffleBcSig2}
Let $X \subseteq \Nb \setminus \{0\}$ be a Boolean combination of $\Sigma_2$ sets, thought of as a set of finite order-types.  Let $\mc{O}$ be a computable colored copy of $\omega$.  There is a computable copy $\mc{L}$ of $\omega$ (constructed from $\mc{O}$) such that for every cohesive set $C$, if $\prod_C \mc{O}$ is colorful, then $\prod_C \mc{L}$ has order-type $\omega + \bm{\sigma}(X \cup \{\omega + \zeta\eta + \omega^*\})$.
\end{Lemma}

\begin{proof}
The $X = \emptyset$ case is Lemma~\ref{lem-SuffleZero}, so we may assume that $X \neq \emptyset$.  Let $k_0$ be the least element of $X$.  By putting $X$ in disjunctive normal form and noticing that finite intersections of $\Sigma_2$ sets are $\Sigma_2$ and that finite intersections of $\Pi_2$ sets are $\Pi_2$, we may write $X$ as a finite union $X = \bigcup_{i < N} X_i$, where each $X_i$ is the intersection of a $\Sigma_2$ set and a $\Pi_2$ set.  It is convenient to further assume that $k_0 \in X_i$ for each $i < N$.

Let $\mc{O} = (R, \Nb, \prec_\mc{R}, G)$ be a computable colored copy of $\omega$, and let $\mc{R}$ denote $(R, \prec_\mc{R})$.  We define a uniformly computable sequence $(\mc{M}_r : r \in R)$ of $\mf{O}$\nobreakdash-structures that are finite models of $\Gamma$.  Let $\la \cdot, \cdot \ra \colon \{0, \dots, N-1\} \times \Nb \imp \Nb$ be a computable bijection with computable projections $\pi_0$ and $\pi_1$.  Compute each $\mc{M}_r$ as in the proof of Lemma~\ref{lem-ShuffleSig2CapPi2}, but for set $X_{\pi_0(G(r))}$ and color $F(r) = \pi_1(G(r))$.

Let $\mc{L}$ be the generalized sum $\sum_{r \in R} \mc{M}_r$, as in the proof of Lemma~\ref{lem-ShuffleSig2CapPi2}.  Then $\mc{L}$ is a computable copy of $\omega$.  Let $C$ be a cohesive set for which $\prod_C \mc{O}$ is colorful.  Then again $\prod_C \mc{L} \;\iso\; \mc{Z}_\std + \mc{Z}_\nonstd$, where $|\prod_C \mc{R}|_\std$ and $|\prod_C \mc{R}|_\nonstd$ denote the standard and non-standard parts of $\prod_C \mc{R}$,
\begin{align*}
\mc{Z}_\std \quad&=\quad \sum_{[\theta] \in \left|\prod_C \mc{R}\right|_\std}\prod\nolimits_C \mc{M}_{\theta(n)}\\ \\
\mc{Z}_\nonstd \quad&=\quad \sum_{[\theta] \in \left|\prod_C \mc{R}\right|_\nonstd}\prod\nolimits_C \mc{M}_{\theta(n)},
\end{align*}
and $\mc{Z}_\std \iso \omega$.

Again, think of $\condSum(\mc{Z}_\nonstd)$ as being colored by $G^{\prod_C \mc{O}}$, where the block $\prod_C \mc{M}_{\theta(n)}$ corresponding to $[\theta] \in |\prod_C \mc{R}|_\nonstd$ gets color $G^{\prod_C \mc{O}}([\theta])$.  The product $\prod_C \mc{O}$ is colorful, which means that $\condSum(\mc{Z}_\nonstd) \iso |\prod_C \mc{R}|_\nonstd \iso \eta$ and that each solid color occurs densely.  The order-type of block $\prod_C \mc{M}_{\theta(n)}$ for each $[\theta] \in |\prod_C \mc{R}|_\nonstd$ can be determined as in the proof of Lemma~\ref{lem-ShuffleSig2CapPi2}.
\begin{itemize}
\item If $G^{\prod_C \mc{O}}([\theta])$ is solid color $\llb \la i, k \ra \rrb$ where $k \in X_i$, then $\prod_C \mc{M}_{\theta(n)} \iso \bm{k}$.

\medskip

\item If $G^{\prod_C \mc{O}}([\theta])$ is solid color $\llb \la i, k \ra \rrb$ where $k < k_0$, then $\prod_C \mc{M}_{\theta(n)} \iso \omega + \zeta\eta + \omega^*$.

\medskip

\item If $G^{\prod_C \mc{O}}([\theta])$ is solid color $\llb \la i, k \ra \rrb$ where $k > k_0$ and $k \notin X_i$, then $\prod_C \mc{M}_{\theta(n)}$ either has order-type $\bm{k}_0$ or order-type $\omega + \zeta\eta + \omega^*$.

\medskip

\item Suppose that $G^{\prod_C \mc{O}}([\theta])$ is a striped color.  As $\pi_0(G(r)) < N$ for every $r \in R$, the cohesiveness of $C$ implies that there is an $i < N$ such that $(\forae n \in C)(\pi_0(G(\theta(n))) = i)$.  Therefore $\lim_{n \in C} \pi_1(G(\theta(n))) = \infty$ because $G^{\prod_C \mc{O}}([\theta])$ is striped, and so $\prod_C \mc{M}_{\theta(n)}$ either has order-type $\bm{k}_0$ or order-type $\omega + \zeta\eta + \omega^*$ as in the proof of Lemma~\ref{lem-ShuffleSig2CapPi2}.
\end{itemize}
Therefore $\mc{Z}_\nonstd \;\iso\; \bm{\sigma}(\bigcup_{i < N}X_i \cup \{\omega + \zeta\eta + \omega^*\}) = \bm{\sigma}(X \cup \{\omega + \zeta\eta + \omega^*\})$.  Thus $\prod_C \mc{L} \;\iso\; \omega + \bm{\sigma}(X \cup \{\omega + \zeta\eta + \omega^*\})$ as desired.
\end{proof}

We end with the main result of this section by combining Theorem~\ref{thm-ColorsDenseNonstd} with Lemmas~\ref{lem-ShuffleFinite} and~\ref{lem-ShuffleBcSig2}.

\begin{Theorem}\label{thm-BCSigma2Shuffle}
Let $X \subseteq \Nb \setminus \{0\}$ be a Boolean combination of $\Sigma_2$ sets, thought of as a set of finite order-types.  Let $C$ be a co-c.e.\ cohesive set.  Then there is a computable copy $\mc{L}$ of $\omega$ where the cohesive power $\prod_C \mc{L}$ has order-type $\omega + \bm{\sigma}(X \cup \{\omega + \zeta\eta + \omega^*\})$.  Moreover, if $X$ is finite and non-empty, then there is also a computable copy $\mc{L}$ of $\omega$ where the cohesive power $\prod_C \mc{L}$ has order-type $\omega + \bm{\sigma}(X)$.
\end{Theorem}

\begin{proof}
Let $C$ be a co-c.e.\ cohesive set.  By Theorem~\ref{thm-ColorsDenseNonstd}, let $\mc{O}$ be a computable colored copy of $\omega$ such that $\prod_C \mc{O}$ is colorful.  Let $X \subseteq \Nb \setminus \{0\}$ be a Boolean combination of $\Sigma_2$ sets.  Let $\mc{L}$ be the computable copy of $\omega$ constructed from $\mc{O}$ for $X$ as provided by Lemma~\ref{lem-ShuffleBcSig2}.  Then $\prod_C \mc{L} \;\iso\; \omega + \bm{\sigma}(X \cup \{\omega + \zeta\eta + \omega^*\})$.  If $X$ is finite and non-empty, then we may alternatively apply Lemma~\ref{lem-ShuffleFinite} instead of Lemma~\ref{lem-ShuffleBcSig2} to obtain a computable copy $\mc{L}$ of $\omega$ with $\prod_C \mc{L} \;\iso\; \omega + \bm{\sigma}(X)$.
\end{proof}

\section*{Acknowledgments}
We thank our anonymous reviewers for their many thoughtful suggestions of generalizations and alternative proofs that greatly improved the cohesiveness of this work.
This project was partially supported by 
the John Templeton Foundation grant ID 60842 (Shafer);
EPSRC grant EP/T031476/1 (Shafer);
the University of Leeds School of Mathematics Research Visitors' Centre (Dimitrov, Shafer);
the FWO Pegasus program (Shafer);
FWO/BAS project VS09816 (Shafer, Soskova);
BNSF, KP-06-Austria-04/19 (Soskova, Vatev);
SU, FNI Grant \#80-10-128/16.04.2020 (Soskova, Vatev);
NSF FRG grant DMS-2152095 (Harizanov);
Simons Foundation Grant \#581896 (Harizanov);
Simons Foundation Grant \#853762 (Harizanov);
and
NSF grant DMS-1600625 (Dimitrov, Harizanov, Morozov, Soskova, Vatev).
The opinions expressed in this work are those of the authors and do not necessarily reflect the views of the John Templeton Foundation.

\bibliographystyle{amsplain}
\bibliography{PowersOfOmega}

\newpage

\vfill

\end{document}